\documentclass[11pt]{amsart}

\usepackage{amsmath}
\usepackage{amsthm}
\usepackage{graphicx}\graphicspath{{Figures/}}
\usepackage{tikz}
\usepackage{enumerate}
\usepackage{mathtools}
\usepackage{bm}
\usepackage[scr=boondoxo]{mathalpha}

\newtheorem{theorem}{Theorem}
\newtheorem{lemma}[theorem]{Lemma}
\newtheorem{proposition}[theorem]{Proposition}

\theoremstyle{definition}
\newtheorem{definition}[theorem]{Definition}
\newtheorem{notation}[theorem]{Notation}

\newtheorem{remark}[theorem]{Remark}

\numberwithin{theorem}{section}

\newcommand{\bC}{\mathbb{C}}

\newcommand{\bZ}{\mathbb{Z}}
\newcommand{\tC}{{\widetilde{\mathbb{C}}}}
\newcommand{\ttau}{\tilde{\tau}}

\newcommand{\cA}{\mathcal{A}}
\newcommand{\cB}{\mathcal{B}}
\newcommand{\cL}{\mathcal{L}}

\newcommand{\tA}{\tilde{\cA}}
\newcommand{\cP}{\mathcal{P}}
\newcommand{\cU}{\mathcal{U}}

\let\phi=\varphi
\newcommand{\centre}[1]{\hbox{$\mathop{#1}\limits^\circ$}}

\newcommand{\gO}{\Omega}
\newcommand{\gS}{\Sigma}
\newcommand{\sM}{{\scriptscriptstyle M}}
\newcommand{\sN}{{\scriptscriptstyle N}}
\newcommand{\st}[1]{{\scriptscriptstyle#1}}
\newcommand{\sone}{{\scriptscriptstyle 1}}
\newcommand{\sV}{{\scriptscriptstyle V}}
\newcommand{\sW}{{\scriptscriptstyle W}}

\newcommand{\NC}{\mathit{NC}}
\newcommand{\vm}{{\vec{m}}}
\newcommand{\sC}{{{\scriptscriptstyle (c)}}}
\newcommand{\spC}{{{\scriptscriptstyle\prime (c)}}}
\newcommand{\sz}{{\scriptscriptstyle0}}
\newcommand{\sep}{\mathit{\ sep.\ }}
\newcommand*{\vv}[1]{\vec{\mkern0mu#1}}
\makeatletter\newcommand{\hathat}[1]{%
\begingroup%
  \let\macc@kerna\z@%
  \let\macc@kernb\z@%
  \let\macc@nucleus\@empty%
  \skew{3.5}\widehat{\mathchoice%
    {\raisebox{.4ex}{\vphantom{\ensuremath{\displaystyle #1}}}}%
    {\raisebox{.4ex}{\vphantom{\ensuremath{\textstyle #1}}}}%
    {\raisebox{.16ex}{\vphantom{\ensuremath{\scriptstyle #1}}}}%
    {\raisebox{.14ex}{\vphantom{\ensuremath{\scriptscriptstyle #1}}}}%
    \smash{{\widehat{{#1}}}}}%
\endgroup%
} \makeatother
\newcommand{\peq}{{p \mskip 2mu\cdot_\epsilon q}}

\newcommand{\cov}{\mathrm{Cov}}
\newcommand{\rE} {\mathbb{E}}
\newcommand{\id}{\mathit{id}}
\newcommand{\rP}{\mathrm{P}}
\newcommand{\Tr}{\mathrm{Tr}}
\newcommand{\tr}{\mathrm{tr}}
\newcommand{\Wg}{\mathrm{Wg}}

\allowdisplaybreaks

\newcommand{\ab}{\allowbreak}

\newcommand{\ds}{\displaystyle}

\newcommand{\xx}[2]{\relax}
\newcommand{\thebottomline}{\renewcommand{\thefootnote}{}
  \renewcommand{\footnoterule}{}
  \phantom{M}\footnotetext{\tiny\textit{\jobname.tex}\hfill
    \textit{\noindent\romannumeral\day.%
\romannumeral\month.\expandafter\xx\romannumeral\year}}}

\renewcommand{\thefootnote}{(\arabic{footnote})}

\title[real infinitesimal freeness]%
{infinitesimal freeness for orthogonally\\[5pt]
invariant random matrices%
}

\author[cébron]{Guillaume Cébron}\address{Institut de
  Mathématiques de Toulouse; UMR5219; Université de
  Toulouse; CNRS; UPS, F-31062 Toulouse, France}
\email{guillaume.cebron@math.univ-toulouse.fr}

\thanks{G.C. is supported by the Project MESA
  (ANR-18-CE40-006) and by the Project STARS
  (ANR-20-CE40-0008) of the French National Research Agency
  (ANR)}

\author[mingo]{James A. Mingo} \address{Department of
  Mathematics and Statistics, Queen's University, Jeffery
  Hall, Kingston, Ontario, K7L 3N6, Canada}

\email{mingo@mast.queensu.ca} 

\thanks{J.A.M. is supported by a Discovery Grant from the
  Natural Sciences and Engineering Research Council of
  Canada}

\begin{document}

\begin{abstract}
We introduce a new kind of free independence, called real
infinitesimal freeness. We show that independent
orthogonally invariant with infinitesimal laws are
asymptotically real infinitesimally free. We introduce new
cumulants, called real infinitesimal cumulants and show that
real infinitesimal freeness is equivalent to vanishing of
mixed cumulants. We prove the formula for cumulants with
products as entries.
\end{abstract}

\maketitle


\section{Introduction and Statement of Results}

The first result in the asymptotic theory of random matrices
was Wigner's semicircle law which gave us the limit
eigenvalue distribution of a Wigner matrix, and in
particular a GUE random matrix. Later Voiculescu,
\cite{voi}, showed that independent GUE random matrices are
asymptotically free. The scope of Voiculescu's theorem has
now been extended by many authors so as not to require the
entries be Gaussian and weakening the assumptions on the
joint distribution of the entries. This made the results of
free probability more widely applicable, in this paper we
extend this further to the infinitesimal law of orthogonally
invariant matrices.

Let us recall the definition of a non-commutative
distribution. Let $(\cA, \tau)$ be a non-commutative
probability space. This means that $\cA$ is a unital algebra
over $\bC$ and $\tau: \cA \to \bC$ is linear with $\tau(1) =
1$. Given $a_1, \dots, a_s \in \cA$ the set $\{ \tau(a_{i_1}
\cdots a_{i_n}) \mid 1 \leq i_1, \dots, i_n \leq s\}$ is the
non-commutative distribution of the random variables $a_1,
\dots, a_s$. Freeness is a rule for computing the joint
distribution from the distribution of the individual
variables, see \cite{ms2017, ns}.

If, for each $N$, we have a non-commutative probability
space $(\cA_\sN, \tau_\sN)$ and random variables $a_{1,
  \sN}, \dots, a_{s, \sN} \in \cA_\sN$, we get, for each
$N$, a non-commutative distribution. Asymptotic freeness
means that the joint distribution tends to the joint
distribution of free random variables.

Frequently the joint distribution of $a_{1, \sN}, \dots,
a_{s, \sN} $ can be expanded into a series in $1/N$, with
the leading term being the limit distribution. The
subleading terms of this expansion have been given a lot of
attention for their connections to Hurwitz numbers,
unitarily invariant ensembles in quantum gravity,
topological recursion, analysis of spike models and
principal minors. Examples of this phenomenon are presented
in the work of Johansson \cite{kj}, Dumitriu and Edelman
\cite{de}, Ledoux \cite{l}, Enriquez and Ménard \cite{em},
Kerov \cite{k}, and Bufetov, \cite{buf, bz}.

When asymptotic freeness holds up to an error of order
$o(1/N)$ in this expansion, the underlying rule which
appears for the first subleading term is the infinitesimal
freeness of Belinschi and Shlyakhtenko~\cite{bs}, or
equivalently the type $B$ freeness of Biane, Goodman and
Nica~\cite{bgn}, or Popa \cite{p}. For example, the
computation of the mixed moments of GUE random matrices up
to $O(1/N^2)$ in the pioneering paper of Thorbjørnsen
\cite{tho} can be rephrased as asymptotic infinitesimal
freeness of independent GUE matrices.

More generally, this first-order expansion of asymptotic
freeness, which yields infinitesimal freeness, holds
whenever the matrix ensembles are unitarily invariant in the
sense that the joint distributions of the entries of the
ensembles are invariant under conjugation by a unitary
matrix. It has first been proven by Curran and Speicher
\cite[Theorem 5.11]{cs}, where this expansion is done up to
$O(1/N^2)$ for bounded deterministic matrices which are
randomly rotated by unitary matrices, implying asymptotic
infinitesimal freeness.  In the work of Shlyakhtenko
\cite{s} and Collins, Hasebe, and Sakuma \cite{chs}
asymptotic infinitesimal freeness of a family of unitarily
invariant matrices from finite rank matrices was
shown. Similarly, Au \cite{a} proved the asymptotic
infinitesimal freeness of Wigner matrices from finite rank
matrices. Beyond the finite rank case, Dallaporta and
Février \cite{df} proved the asymptotic infinitesimal
freeness of independent GUE matrices from bounded
deterministic matrices. Finally, the asymptotic
infinitesimal freeness of two independent random matrices,
at least one of them being unitarily invariant, is a
consequence of the general theory of surfaced free
probability of Borot, Charbonnier, Garcia-Failde, Leid, and
Shadrin \cite{bcgls}, and also a consequence of the
computation of the matricial cumulants by the first author,
Dahlqvist and Gabriel \cite{cdg}. Indeed, both works contain
an explicit and complete expansion in powers of $1/N^{2}$ of
the mixed moments of independent and unitarily invariant
random matrices from which asymptotic infinitesimal freeness
can be recovered.

The main achievement of this paper is to weaken the
assumption of unitary invariance to invariance under the
smaller group of orthogonal matrices. As has been known
since the work of Goulden and Jackson \cite{gj}, this means
we now have to consider both non-orientable and orientable
surfaces in our analyses. On the matrix side this means
working with the transposes of the matrices in our
ensemble. The resulting infinitesimal freeness we call
\textit{real infinitesimal freeness}.

The need for this investigation was shown in \cite{m}, where
it was shown that independent GOE random matrices were not
asymptotically infinitesimally free in the sense of
\cite{fn}, but that there was a universal rule for computing
joint distributions. This was extended in \cite{mvb} to the
infinitesimal law of real Wishart random matrices. In
\cite{cgth}, Chen, Garza-Vargas, Tropp, and van Handel
demonstrated the existence of limit distributions; however
the method of calculating them was left open. This paper
provides a method for finding the $1/N$ distribution in the
orthogonal case.

In this paper we show that given ensembles $\{ \cA_{1, \sN},
\dots, \cA_{s, \sN} \}_N$, all of which, or all but one of
which, are orthogonally invariant with entries from
different ensembles independent and with appropriate limit
distributions, then the joint infinitesimal distribution
converges to a joint distribution which satisfies our new
rule of real infinitesimal freeness.

In detail, in \textbf{\S
  \ref{section:infinitesimal_cumulants}} we recall the
notions of (complex) infinitesimal freeness from
\cite{fn}. Then in \textbf{ \S \ref{sec:real infinitesimal
    probability spaces}} we introduce in Definition
\ref{def:real infinitesimal freeness}, real infinitesimal
freeness, and present an equivalent formulation in
Proposition \ref{prop:non-tracial case} that we use to make
the connection to random matrix ensembles. In \textbf{\S
  \ref{sec:integration by parts for random matrices}} we
present a formula for integration by parts needed for the
evaluation of orthogonally invariant matrix integrals. This
is needed for our proof, but is quite general and of
independent interest. In \textbf{\S \ref{sec:asymptotic
    freeness of orthogonally invariant matrices}}, we prove
our main result, Theorem \ref{thm:asymptotic real
  infinitesimal freeness}, which shows that orthogonally
invariant ensembles are asymptotically infinitesimally free.
In \textbf{\S \ref{section:real infinitesimal cumulants}} we
introduce real infinitesimal cumulants and prove Theorem
\ref{thm:moment cumulant}, the moment-cumulant formula. In
\textbf{\S \ref{sec:real infinitesimal cumulants and real
    infinitesimal freeness}} we show, in Theorem
\ref{thm:freeness and the vanishing of mixed cumulants},
that real infinitesimal freeness is equivalent to the
vanishing of mixed cumulants. In order to prove Theorem
\ref{thm:freeness and the vanishing of mixed cumulants} we
need the formula for cumulants with products as entries. In
\textbf{\S \ref{section:product formula}} we present the
statement of this formula in Theorem \ref{thm:product
  rule}. In addition to being necessary for the proof of
Theorem \ref{thm:freeness and the vanishing of mixed
  cumulants}, product formulas such as these are a key
computational tool with many applications in free
probability. In \textbf{\S \ref{section:small example}} we
present a small example as to how our formula works using
the square of semi-circular operator.  In \textbf{\S
  \ref{sec:product rule outline}}, we prepare the proof of
Theorem \ref{thm:product rule} by breaking it into the two
steps followed in the subsequent sections. In \textbf{\S
  \ref{sec:product formula first step}}, we present the
first part of the general case, with a focus on the
non-annular terms. The proof is concluded in \textbf{\S
  \ref{sec:product formula second step}} with a discussion
of the annular terms.

This paper is a part (the other parts being \cite{m} and
\cite{mvb}) of a series of papers investigating real
infinitesimal freeness. In \cite{m} it was shown that
independent GOE random matrices are not asymptotically
free. In \cite{mvb} the set of symmetric non-crossing
annular permutations was introduced in the context of the
infinitesimal law of a real Wishart matrix. It was shown
there that independent real Wishart matrices were
asymptotically real infinitesimally free, although at that
time the rules presented here were unavailable. In a
forthcoming fourth part, the connection between real
infinitesimal freeness and the subleading term of the finite
free cumulants of \cite{ap, agp, apv}. In \cite{ap},
Arizmendi and Perales showed that the leading order of a
finite free cumulant converges as $d \to \infty$ to the
corresponding free cumulants. In particular it will be shown
that the subleading term of \cite{apv} is the same as the
real infinitesimal freeness presented here.

\section{Complex Infinitesimal Free Freeness}
\label{section:infinitesimal_cumulants}

We present here a quick review of complex infinitesimal
freeness as this is our point of departure. We don't use any
results of this section in the rest of the paper, but
understanding the complex case is very useful for following
the rest of the paper. We let $\tC$ be the commutative 2
dimensional algebra of upper triangular matrices which are
constant on the diagonal
\[
\tC = \bigg\{\begin{bmatrix} \alpha & \alpha' \\ 0 &
\alpha \end{bmatrix} \ \bigg\vert\ \alpha, \alpha' \in \bC
\bigg\}.
\]

Now we suppose that $(\cA, \tau, \tau')$ is a complex
infinitesimal probability space. This means $\cA$ is a
unital algebra over $\bC$, $\tau, \tau': \cA \to \bC$ with
$\tau(1) = 1$ and $\tau'(1) = 0$. We next let $\tA$ be the
same construction applied to $\cA$:
\[
\tA = \bigg\{\begin{bmatrix} a & a' \\ 0 & a \end{bmatrix}
\ \bigg\vert\ a, a' \in \cA \bigg\}.
\]
Then we define the linear map $\ttau: \tA \rightarrow \tC$
written symbolically as
\[
\ttau = \begin{bmatrix} \tau & \tau' \\ 0 & \tau \end{bmatrix};
\]
by this we mean
\[
\ttau \bigg(\begin{bmatrix} a & a' \\ 0 & a \end{bmatrix}
\bigg) = \begin{bmatrix} \tau & \tau' \\ 0 &
  \tau \end{bmatrix} \begin{bmatrix} a & a' \\ 0 &
  a \end{bmatrix} =
\begin{bmatrix} \tau(a) & \tau'(a) + \tau(a') \\ 0 & \tau(a) \end{bmatrix}.
\]
Note that $\ttau$ is a conditional expectation of $\tA$ onto
$\tC$.  In \cite{t} it was shown that the definition of
complex infinitesimal freeness presented in Remark
\ref{remark:comparison} $(i)$, $(ii)$, and $(v)$ is
equivalent to freeness over $\tC$. Let us recall the
statement. Suppose $(\cA, \tau, \tau')$ is an infinitesimal
probability space and $\cA_1, \dots, \cA_s$ are unital
subalgebras of $\cA$. Let $\tilde{\cA}_1, \tilde{\cA}_1,
\dots, \tilde{\cA}_1$ be the upper triangular subalgebras of
$\tA$ obtained from the construction above. Then $\cA_1,
\dots, \cA_s$ are free with respect to $(\tau, \tau')$ if
and only if $\tilde{\cA}_1, \dots, \tilde{\cA}_s$ are free
over $\tC$ with respect to $\ttau$. As mentioned above $\tC$
is a commutative ring with unit, so nearly all of the
combinatorial theorems of \cite{ns} remain valid over
$\tC$. This has some far reaching consequences which greatly
simplify many proofs (see Equation (\ref{eq:complex product
  formula}) below). Let us recall our notation for the
infinitesimal cumulants. Recall first the moment cumulant
formula. Given $a_1, \dots, a_n \in \cA$ we have
\begin{equation}\label{eq:moment_cumulant}
\tau(a_1 \cdots a_n)
=
\sum_{\pi \in \NC(n)} \kappa_\pi(a_1, \dots, a_n)
\end{equation}
see \cite[Lecture 11]{ns}. 

We do the same thing in the complex infinitesimal case by
setting \[\partial \kappa_n^{\sC}(a_1, \dots, a_n) =
\kappa_n^{{\spC}}(a_1, \dots, a_n)\] and for $\pi \in
\NC(n)$ we set
\[
\partial \kappa_\pi^{\sC} (a_1, \dots, a_n)
=
\sum_{V \in \pi} \kappa_{|V|}^{{\spC}}(a_1, \dots, a_n \mid V)
\prod_{W \not = V} \kappa_{|W|}(a_1, \dots, a_n \mid W). 
\]
Then the complex infinitesimal version of equation
(\ref{eq:moment_cumulant}) is obtained by formal implicit
differentiation
\begin{equation}\label{eq:complex infinitesimal_moment_cumulant}
\tau'(a_1 \cdots a_n)
=
\sum_{\pi \in \NC(n)} \partial \kappa_\pi^{\sC}(a_1, \dots, a_n).
\end{equation}
When $n=1$ we have $\tau'(a_1) = \kappa^{\spC}_1(a_1)$ and
when $n = 2$ we have $\tau'(a_1 a_2) = \kappa^\spC_2(a_1,
a_2) + \kappa^\spC_1(a_1) \kappa_1(a_2) + \kappa_1(a_1)
\kappa^\spC_1(a_2)$. From these two equations we can obtain
formulas for $\kappa^\spC_1$ and $\kappa^\spC_2$ in terms of
$\tau$ and $\tau'$. By using matricial cumulants $\{ \tilde
\kappa_n \}_n$ with values in $\tC$ we can write this quite
simply as:
\[
\tilde\kappa_\pi(A_1, \dots, A_n)
=
\begin{bmatrix} \kappa_\pi(a_1, \dots, a_n ) &
                \partial\kappa^\sC_\pi(a_1, \dots, a_n ) \\
                &  \mbox{} +
                \ds\sum_{k=1}^n \kappa_\pi(a_1, \dots, a'_k, \dots, a_n ) \\
                0 & \kappa_\pi(a_1, \dots, a_n )
\end{bmatrix},                               
\]
with
\[
A_1 = \begin{bmatrix} a_1 & a'_1 \\ 0 & a_1 \end{bmatrix}, \dots, 
A_n = \begin{bmatrix} a_n & a'_n \\ 0 & a_n \end{bmatrix}
\in \tA. 
\]
Then we have the usual moment-cumulant relations:
\[
\tilde\tau(A_1 \cdots A_n) = \sum_{\pi \in \NC(n)} \tilde\kappa_\pi(A_1, \dots, A_n)
\]
and
\[
\tilde\kappa(A_1, \dots, A_n) = \sum_{\pi \in \NC(n)} \mu(\pi, 1_n)\tilde\tau_\pi(A_1, \dots, A_n). 
\]

When we work with upper triangular matrices the formula for
cumulants with products with entries follows from
\cite[Theorem 14.4]{ns} because the algebra $\tC$ is
commutative. Thus, when we examine the $(1, 2)$ entry of the
cumulant matrix we find that when we have $a_1, \dots, a_n
\in \cA$ with $(\cA, \tau, \tau')$ a complex infinitesimal
probability space and we let ${\bm a_1} = a_1 \cdots
a_{n_1}$, \dots, ${\bm a_r} = a_{n_1 + \cdots + n_{r-1} + 1}
\cdots a_{n_1 + \cdots + n_r}$ then
\begin{align}
\label{eq:complex product formula}
\kappa_r'({\bm a_1}, \dots, {\bm a_r})
=
\mathop{\sum_{\pi \in \NC(n)}}_{\pi \vee \rho_r = 1_n}
\partial \kappa_\pi(a_1, \dots, a_n)
\end{align}
where $\rho_r$ is the interval partition with intervals $\{
\{n_1 + \cdots + n_{l-1} + 1, \dots, n_1 + \cdots + n_l\}
\}_{l=1}^r$. In Theorem \ref{thm:product rule} we present
the `real' version of this formula.

In Equation (\ref{eq:infinitesimal moment cumulant}) of \S
\ref{section:real infinitesimal cumulants}, we replace
equation (\ref{eq:complex infinitesimal_moment_cumulant}) by
\begin{equation}
\tau'(a_1 \cdots a_n) = \sum_{\pi \in \NC(n)}
\nabla\kappa_\pi (a_1, \dots, a_n)
\tag{\ref{eq:infinitesimal moment cumulant}}
\end{equation}
where $\nabla = \partial + \delta$ and $\partial$ is as
above and $\delta$ is something new for the real case, which
we call the \textit{spatial derivative}, see Notation
\ref{notation:spatial derivative}. Since the left hand side
of Equation (\ref{eq:infinitesimal moment cumulant}) doesn't
change in passing to the real case, this changes the values
of the real infinitesimal cumulants so that they capture the
properties we seek from the random matrix models. This fixes
the problem with the complex infinitesimal cumulants of the
GOE reported in \cite[Prop. 29]{m}.


\section{Real Infinitesimal Probability Spaces}\label{sec:real infinitesimal probability spaces}

In this section we review some notions of free independence
with the addition of an involution. These have already
appeared in the work of Redelmeier \cite{r} others, but we
repeat them here for clarity.

Let $\cA = \bC\langle x_1, \dots, x_s, x_1^t, \dots,
x_s^t\rangle$ where $\{ x_1, \dots, x_s, x_1^t, \dots, x_s^t
\}$ are $2 s$ non-commuting variables. For this part of the
discussion it is useful to adopt the notation that
$x^{\st{(1)}}_i = x_i$ and $x^{\st{(-1)}}_i = x_i^t$. We
define an involution, $w \to w^t$, on $\cA$ by mapping
$x^{(\epsilon_1)}_{i_1} \cdots x^{(\epsilon_n)}_{i_n}$ to
$x^{(-\epsilon_n)}_{i_n} \cdots x^{(-\epsilon_1)}_{i_1}$,
where $\epsilon_1, \dots, \epsilon_n \in \{-1, 1\}$, and
then extend to all of $\cA$ by linearity. If $a \in \cA$ is
such that $a = a^t$ we say that is \textit{symmetric}. A
linear subspace $\cB \subseteq \cA$ is \textit{symmetric} if
$b^t \in \cB$ whenever $b \in \cB$.  We say that a linear
map $\tau: \cA \to \bC$ is \textit{symmetric} if $\tau(a^t)
= \tau(a)$. A linear map $\tau: \cA \to \bC$ is a
\textit{state} if $\tau$ is symmetric and $\tau(1) = 1$. The
triple $(\cA, t, \tau)$ is just a special case of a
\textit{real non-commutative probability space}: $\cA$ is a
unital algebra over $\bC$, $t$ is an involution, and $\tau:
\cA \to \bC$ is a state.

Another important example of a real probability space comes
from random matrices. Let $(\gO, \gS, \rP)$ be a probability
space , and $\cL^{\infty-}$ be the commutative algebra of
random variables, $X$, on $\gO$ such that $\rE(|X|^n) <
\infty$ for all $n$. For a fixed $N$, we let $\cA_\sN =
M_\sN(\cL^{\infty-})$ be the matrices with entries from
$\cL^{\infty-}$, and for $A \in \cA_\sN$ we let $\tau_\sN(A)
= \frac{\sone}{\sN} \Tr(A)$. The involution, $t$, is the
usual transpose.

Given any real non-commutative probability space $(\cB,
\tau, t)$ and elements $b_1 \dots , b_n \in \cB$ we define a
state, $\tau_{\vec b}$, on $\bC\langle x_1, \dots, x_s,
x_1^t, \dots, x_s^t\rangle$ by $\tau_{\vec b\, }(p) =
\tau(p(b_1, \dots, b_s, b_1^t, \dots, b_s^t))$. We call
$\tau_{\vec b\,}$ the \textit{joint distribution} of the
$n$-tuple $\vec b = (b_1, \dots, b_n)$. The variables $(b_1,
\dots, b_n)$ are free with respect to $\tau$ if and only if
the variables $(x_1, \dots, x_n)$ are free with respect to
$\tau_{\vec b}$, because they have the same joint
distribution. We say that the random variables $(b_1, \dots,
b_n)$ are $t$-free with respect to $\tau$ if the random
variables $(x_1, \dots, x_n, x_1^t, \dots, x_n^t)$ are free
with respect to $\tau_{\vec b}$.

If we have a sequence $\{(\cB_\sN, \tau_\sN, t)\}_\sN$ of
real non-commutative probability spaces and for each $N$ a
$n$-tuple of random variables $b_{1, \sN}, \dots, b_{n, \sN}
\in \cB_\sN$, we say the tuples $\{\vec {b}_\sN\}_\sN$
\textit{converge in distribution} if the sequence of states
$\{ \tau_{\vec{b}_\sN} \}_\sN$ converges point-wise on
$\bC\langle x_1, \dots, x_s, x_1^t, \dots,
x_s^t\rangle$. The state to which $\{ \tau_{\vec{b}_\sN}
\}_\sN$ converges is called the \textit{limit distribution}.

If $(\cA, \tau, t)$ is a real probability space and $\tau':
\cA \to \bC$ is a symmetric linear map with $\tau'(1) = 0$
we call the quadruple $(\cA, \tau, \tau', t)$ a \textit{real
  infinitesimal non-commutative probability space}, or to be
brief, a \textit{real infinitesimal probability space}.

Given a sequence, $\{\vec {b}_\sN\}_\sN$, of random
variables with joint distributions $\{ \tau_{\vec{b}_\sN}
\}_\sN$ converging to the state $\tau_{\vec b}$ we can
create a sequence of linear maps $\{ \tau'_{\vec{b}_\sN}
\}_\sN$ by setting
\[
\tau'_{\vec{b}_\sN}(p) = N( \tau_{\vec{b}_\sN}(p) - \tau_{\vec{b}\,}(p)) 
\]
for $p \in \bC\langle x_1, \dots, x_s, x_1^t, \dots,
x_s^t\rangle$. We have $\tau'_{\vec{b}_\sN}(p^t) =
\tau'_{\vec{b}_\sN}(p)$, but $\tau'_{\vec{b}_\sN}(1) = 0$;
so $\tau'_{\vec{b}_\sN}$ is not a state. If we let $\cB_\sN$
be the unital subalgebra generated by ${b}_{1_\sN}, \dots,
{b}_{_\sN}, {b}^t_{1_\sN}, \dots, {b}^t_{_\sN}$, we have a
real infinitesimal probability space $(\cB_\sN, \tau_{\vec
  b, \sN}, \tau'_{\vec b, \sN}, t)$. If the sequence of
linear functionals $\{\tau'_{\vec{b}_\sN}\}_\sN$ converges
point-wise on $\bC\langle x_1, \dots, x_s, x_1^t, \dots,
x_s^t\rangle$ to $\tau'_{\vec b}$ we say that the sequence
of variables $\{ \vec{b}_\sN\}_\sN$ has a \textit{limit real
  infinitesimal law}. The inclusion of the word `real' is to
signal that we always include transposes in this notation by
requiring convergence on $\bC\langle x_1, \dots, x_s, x_1^t,
\dots, x_s^t\rangle$.

\begin{definition}\label{def:real infinitesimal freeness}
Let $(\cA, \tau, \tau', t)$ be a real infinitesimal
probability space and $\cA_1, \dots, \cA_s \subseteq \cA$
symmetric unital subalgebras. We say that the subalgebras
$\cA_1, \dots, \cA_s$ are \textit{real infinitesimally free}
if: whenever $a_1, \dots, a_n \in \cA$ with $\tau(a_i) = 0$
and $a_i \in \cA_{j_i}$ with $j_1\not = j_2$, \dots,
$j_{n-1} \not= j_n$; we have

\begin{enumerate}
\item\smallskip
$\tau(a_1 \cdots a_n) = 0$, and

\item\smallskip
when $n = 2$, $\tau'(a_1 a_2) = 0$ 

\item\smallskip
when $n = 2k - 1 \geq 3$, we have
\begin{align*}
\tau'(a_1 \cdots a_n) =\mbox{}  & \tau(a_1\tau'(a_2 \cdots a_{n-1})a_n) \\
& \mbox{} +
\tau(a_1 a_k^t a_n) \tau(a_2 a_{k+1}^t) \cdots \tau(a_{k-1}a_{n-1}^t) 
\end{align*}

\item\smallskip
when $n = 2k \geq 4$, we have
\begin{align*}
\tau'(a_1 \cdots a_n) =\mbox{}  & \tau(a_1\tau'(a_2 \cdots a_{n-1})a_n) \\
& \mbox{} +
\tau(a_1 a_{k+1}^t) \tau(a_2 a_{k+2}^t) \cdots \tau(a_{k}a_{n}^t)
\end{align*}
\end{enumerate}
\end{definition}

\begin{remark}\label{remark:comparison}
It is worth comparing the properties $(i)$, $(ii)$, and
$(iii)$ with type $B$ freeness or complex infinitesimal
freeness from \cite{fn}. Note that $(i)$ just says that the
algebras $\cA_1, \dots, \cA_s$ are free with respect to
$\tau$.

Recall that the condition that $a_i \in \cA_{j_i}$ with $j_1
\not= j_2$, \dots, $j_{n-1} \not= j_n$ is called
\textit{alternating} and that if we assume in addition that
$j_n \not = j_1$ this stronger property is called
\textit{cyclically alternating}. These properties are always
relative to a specified set of subalgebras.

For type $B$ or complex infinitesimal freeness we have that
if $a_1, \dots, a_n$ are centred and alternating then
\begin{enumerate}

\item\smallskip
$\tau(a_1 \cdots a_n) = 0$, and

\item[$(ii)$]\smallskip
when $n = 2$, $\tau'(a_1 a_2) = 0$  

\item[$(v)$]\smallskip
for $n \geq 3$, $\tau'(a_1 \cdots a_n) = \tau(a_1\tau'(a_2 \cdots a_{n-1})a_n)$.
\end{enumerate}
So items $(i)$ and $(ii)$ from Definition \ref{def:real
  infinitesimal freeness} are unchanged, and $(iii)$ and
$(iv)$ are replaced by $(v)$. Putting $(ii)$ and $(v)$
together we see that $\tau'(a_1 \cdots a_n) = 0$ for $n$
even, whereas for $n = 2k -1 $ odd, we have
\[
\tau'(a_1 \cdots a_n) = \tau(a_1a_n)\cdots \tau(a_{k-1}a_{k+1}) \tau'(a_k), 
\]
which is the usual expression of the rule for type $B$ or
infinitesimal freeness. In parts $(iii)$ and $(iv)$ of
Definition \ref{def:real infinitesimal freeness} we each
have two terms: one involving $\tau'$ and the other
involving $\tau$ but having a transpose on some of the
arguments. To give these parts a name we refer to term
involving $\tau'$, the \textit{time derivative term} and the
one involving the transpose the \textit{space derivative
  term}. When both $\tau$ and $\tau'$ are tracial and we
assume the arguments are cyclically alternating then the
time derivative term disappears and we only have the space
derivative.

\end{remark}

\begin{remark}
When we start with a random matrix ensemble and use
$\frac{\sone}{\sN}\Tr$ as the state we will get in the limit
a non-commutative probability space with state $\tau$ which
is a trace: $\tau(ab) = \tau(ba)$. In the case of freeness
there is no simplification when $\tau$ is a trace, however
in the next section we shall see a different rule (Equation
(\ref{eq:intermediate limit distribution}) in Lemma
\ref{lemma:inside induction}) for real infinitesimal
freeness arising from our random matrix model.  In the
Proposition below we show that these two rules are
equivalent.
\end{remark}

\begin{proposition}\label{prop:non-tracial case}
Let $(\cA, \tau, \tau', t)$ be a real infinitesimal
probability space with $\tau$ and $\tau'$ tracial. Let
$\cA_1, \dots, \cA_s \subseteq \cA$ be unital symmetric
subalgebras which are free with respect to $\tau$. Then
$\cA_1, \dots, \cA_n$ are real infinitesimal free if and
only if whenever $a_1, \dots, a_n \in \cA_1 \cup \cdots \cup
\cA_s$ are centred and cyclically alternating, we have

\begin{enumerate}

\item
when $n = 2$ or $n$ is odd, $\tau'(a_1  \cdots a_n) = 0$ and, \medskip

\item
when $n = 2 k \geq 4$ is even 
$\tau'(a_1 \cdots a_n) = \tau(a_1a_{k+1}^t) \cdots \tau(a_ka_n^t)$.
\end{enumerate}
\end{proposition}

\begin{proof}
First, let us assume that $\cA_1, \dots, \cA_s$ are real
infinitesimally free and $a_i \in \cA_{j_i}$ are centred and
cyclically alternating. We shall then prove conditions $(i)$
and $(ii)$ of the Proposition. We have $\tau(a_1\tau'(a_2
\cdots a_{n-1})a_n)\ab = 0$, as $j_n \not= j_1$. Next
suppose $2 \leq k \leq n-1$ and consider the indices $j_1,
j_k$, and $j_n$. We have three cases: $j_1$, $j_k$, $j_n$
distinct, $j_1 = j_k$, or $j_k = j_n$. In all three cases
$\tau(a_1 a_k^t a_n) = 0$ by the freeness of $\cA_{j_1}$ and
$\cA_{j_n}$. Thus when $n$ is odd or equal to $2$ we have
$\tau'(a_1 \cdots a_n) = 0$. Now when $n \geq 4$ is even we
have by $(iv)$ of Definition \ref{def:real infinitesimal
  freeness} that $(ii)$ above holds.

To prove the reverse implication, let us assume that $(i)$
and $(ii)$ above hold whenever $a_1, \dots, a_n$ are centred
and cyclically alternating. Let us show that $\cA_1, \dots,
\cA_s$ are real infinitesimally free. This means that we
have to show that whenever we have $a_1, \dots, a_n$ are
centred and alternating, the conditions $(iii)$ and $(iv)$
of Definition \ref{def:real infinitesimal freeness} hold.

Let us assume that $n = 2k -1$ is odd. If $j_n \not = j_1$
then just by the freeness of $\cA_1, \dots, \cA_n$ we have
that $\tau(a_1 \tau'(a_2 \cdots a_{n-1})a_n) = 0$ and
$\tau(a_1 a_k^t a_n) = 0$. Thus $(i)$ above, implies $(iii)$
of Definition \ref{def:real infinitesimal freeness}. Now
assume that $j_n = j_1$ and let $\tilde{a}_1 = a_na_1 -
\tau(a_n a_1)$. Now $\tilde{a}_1, a_2, \dots, a_{n-1}$ are
centred and cyclically alternating and $n-1$ is even. Thus
by $(ii)$ above we have that
\[
\tau'(\tilde{a}_1 a_2 \cdots a_{n-1}) = \tau(\tilde{a}_1 a_k^t) \cdots \tau(a_{k-1}a_{n-1}^t) = \tau(a_1 a_k^t a_n) \cdots \tau(a_{k-1}a_{n-1}^t)
\]
\noindent
because $\tau(\tilde{a}_1a_k^t) = \tau(a_1 a_k^t a_n)$. Finally we have
\[
\tau'(a_1 \cdots a_n) = \tau'(a_n a_1 \cdots a_{n-1}) 
=
\tau'(\tilde{a}_n a_2 \cdots a_{n-1}) + \tau(a_na_1) \tau'(a_2 \cdots a_{n-1})
\]
\[
= \tau(a_1 \tau'(a_2 \cdots a_{n-1}) a_n) =
\tau(a_1 a_k^t a_n) \cdots \tau(a_{k-1}a_{n-1}^t)
\]
which is exactly condition $(iii)$ of Definition
\ref{def:real infinitesimal freeness}.

Now let us assume $n = 2k \geq 4$. We must prove $(iv)$ of
Definition \ref{def:real infinitesimal freeness}. We do this
in two cases: $j_1 \not = j_n$ and secondly $j_1 = j_n$. In
case $j_i \not = j_n$ the right hand side of $(iv)$ in
Definition \ref{def:real infinitesimal freeness} becomes the
right hand side of $(ii)$ in Proposition
\ref{prop:non-tracial case}, because $j_1 \not = j_n$
implies that
\[
\tau(a_1 \tau'(a_2 \cdots a_{n-1}) a_n)
=
\tau(a_1 a_n) \tau'(a_2 \cdots a_{n-1}) =0 . 
\]
So we are done in the case $j_1 \not = j_n$. Before going
further we need to prove two subclaims.

\medskip\noindent\textit{Sub-claim 1.}\medskip

Suppose $\tau'$ satisfies $(i)$ and $(ii)$ of Proposition
\ref{prop:non-tracial case} and $a_1, \dots, a_n$ are
centred and alternating with $j_1 = j_n$ and $n$ even. Then
$\tau'(a_1 \cdots a_n) = 0$.

\medskip\noindent We prove this by induction on $n$. When $n
= 2$, the claim holds by our assumption: $(i)$ above. For $n
> 2$ we have, letting $\tilde{a}_1 = a_n a_1 - \tau(a_n
a_1)$,
\begin{align*}
\tau'(a_1 \cdots a_n) &= \tau'(a_na_1 a_2 \cdots a_{n-1})
=
\tau'(\tilde{a}_1 a_2 \cdots a_{n-1}) \\
& \qquad\mbox{} + \tau(a_1 \tau'(a_2 \cdots a_{n-1})a_n)
 = 0,
\end{align*}
where the first term vanishes by $(i)$ above, because $n-1$
is odd, and the second term vanishes by our induction
hypothesis. This proves Sub-claim 1.

\medskip\noindent\textit{Sub-claim 2.}\medskip

If $n = 2 k$ is even, $a_1, \dots, a_n$ are centred and alternating, but with $j_1 = j_n$, then 
\[
\tau(a_1 a_{k+1}^t) \cdots \tau(a_k a_n^t) = 0. 
\]

\medskip\noindent In order to have $\tau(a_1 a_{k+1}^t)
\cdots \tau(a_k a_n^t) \not = 0$ we must have $j_1 =
j_{k+1}$ and $j_k = j_n$. But by assumption $j_1 = j_n$,
thus $j_k = j_{k+1}$ contrary to our assumption that $a_1,
\dots, a_n$ are alternating. This proves Sub-claim
2. \medskip

Now let us conclude the proof of the Proposition. We assume
$n = 2 k$ and that $\tau'$ satisfies $(i)$ and $(ii)$ of
Proposition \ref{prop:non-tracial case}, and $a_1, \dots,
a_n$ are centred and alternating, with $j_1 = j_n$. We must
prove that
\begin{equation}\label{eq:even case final step}
\tau'(a_1 \cdots a_n)
=
\tau(a_1\tau'(a_2 \cdots a_{n-1}) a_n)
+
\tau(a_1 a_{k+1}^t) \cdots \tau(a_k a_n^t)
\end{equation}
which is $(iv)$ of Definition \ref{def:real infinitesimal
  freeness}. By Sub-claim 1 we have $\tau'(a_1 \cdots a_n) =
\tau(a_1\tau'(a_2 \cdots a_{n-1}) a_n) = 0$. By Sub-claim 2
we have $\tau(a_1 a_{k+1}^t) \cdots \tau(a_k a_n^t) =
0$. Thus both sides of (\ref{eq:even case final step})
vanish and this concludes the proof of the Proposition.
\end{proof}


\section{Integration by parts for random matrices}\label{sec:integration by parts for random matrices}

We need some general definitions in order to do integration
by parts on the orthogonal group. Let $\mathfrak{so}(N)$ be
the linear spaces of real skew-symmetric matrices of size
$N$, and let $(K_{ab})_{1\leq a< b\leq N}$ be the basis of
$\mathfrak{so}(N)$ given by $K_{ab}=E_{ab}-E_{ba}$ where
$E_{ab}$ is the matrix with $1$ in the $(a, b)$-entry and
$0$ elsewhere. We have
\begin{align*}
\sum_{1 \leq a < b \leq N} K_{a  b}\otimes K_{a b}
& = \kern-0.75em
\sum_{1 \leq a < b \leq N} \kern-0.5em E_{ab}\otimes (E_{ab} - E_{ba})
- \kern-0.75em
\sum_{1 \leq a > b \leq N} \kern-0.5em E_{ab} \otimes (E_{ba} - E_{ab}) \\
&=
\sum_{1 \leq a, b \leq N}  E_{ab}\otimes E_{ab} - \sum_{1\leq a, b\leq N}E_{ab}\otimes E_{ba}\\
&=P-T
\end{align*}
where $P= \sum_{1\leq a, b\leq N}E_{ab}\otimes E_{ab}$ and
$T=\sum_{1\leq a, b\leq N}E_{ab}\otimes E_{ba}$. For a
differentiable function $f: O(N) \to M_N(\mathbb{C})$, we
define the left derivative $\partial f(O): \mathfrak{so}(N)
\to M_N(\bC)$ by
$$
\partial f(O) (K_{ab})=\left.\frac{d}{d t}\right|_{t=0} f(e^{tK_{ab}}O).
$$ For convenience we denote this by $\partial_{K_{ab}}f$
and with this notation we consider it as a map $O(N) \to
M_N(\bC)$. If $f$ has two derivatives, we can differentiate
$\partial_{K_{ab}}f$ to get $\partial^2_{K_{ab}}f$. For such
$f$ we then define the Laplacian operator by
\[
\Delta f=\sum_{1\leq a< b\leq N}\partial_{K_{ab}}^2f.
\]
Denoting by $\id:SO(N)\to M_N(\mathbb{C})$ the map
$\id(O)=O$ and by $\iota:SO(N)\to M_N(\mathbb{C})$ the map
$\iota(O)=O^t=O^{-1}$. We have
\[
\partial_{K_{ab}}id(O)=K_{ab}O\ \text{ and }\ \partial_{K_{ab}}\iota(O)=-O^{-1}K_{ab},
\]
or more concisely, $\partial_{K_{ab}}\id=K_{ab}\id$ and
$\partial_{K_{ab}}\iota = - \iota K_{ab}$. In particular we
have, $\Delta \id =\ab\ds\sum_{1\leq a< b\leq N}K_{ab}
K_{ab}\id$ and using $\ds\sum_{1\leq a< b\leq
  N}K_{ab}\otimes K_{ab}=P-T$ we get
$$\Delta \id=\sum_{1 \leq a , b\leq N} -(E_{ba}E_{ab}+E_{ab} E_{ba}) \id = (1-N)\id.$$
Let us define the \textit{carré du champ} operator $\Gamma(f,g):SO(N)\to M_N(\mathbb{C})\otimes  M_N(\mathbb{C})$ by
\[
\Gamma(f,g)= \sum_{1\leq a< b\leq N}\partial_{K_{ab}}f\otimes \partial_{K_{ab}}g.
\]
If both $f$ and $g$ are twice differentiable functions from
$O(N)$ to $M_N(\bC)$ then we define $f \otimes g: O(N) \to
M_N(\bC) \otimes M_N(\bC)$ by $f\otimes g(O) = f(O) \otimes
g(O)$. Then
\begin{align*} 
\partial^2_{K_{a, b}}(f \otimes g)(O)
\mbox{}\ = \ \mbox{} & 
\partial^2_{K_{a, b}}(f)(O) \otimes g(O)   \\
& \mbox{}+ 
2 \partial_{K_{a, b}}(f)(O) \otimes \partial_{K_{a, b}}(g)(O) + f(O) \otimes \partial^2_{K_{a, b}}(g)(O),
\end{align*}
and so
\[
\Delta(f\otimes g)=\Delta(f)\otimes g+2\Gamma(f,g)+f\otimes \Delta(g).
\]
Because the Haar measure on $O(N)$ is invariant under
multiplication by $e^{tK_{a, b}}$ we have
$\mathbb{E}(\Delta(f \otimes g)) = 0$. In addition we have
by using integration by parts twice that $\mathbb{E}(f
\otimes \Delta(g)) = \mathbb{E}(\Delta(f) \otimes g)$. Thus
we have the following basic integration by parts formula:
\begin{equation}
\mathbb{E}[\Gamma(f,g)(O_N)]
= \mbox{} -
\mathbb{E}[\Delta f(O_N)\otimes g(O_N)].\label{intbypart}
\end{equation}

\begin{proposition}\label{prop:int_by_parts}
Let $O_N$ be a Haar distributed orthogonal random matrix of
size $N$, $n$ even, and $M_1,\ldots,M_n\in
M_N(\mathbb{C})$. We have
\begin{align*}\lefteqn{
(N-1)\cdot \mathbb{E}\big[\Tr(O_NM_1O_N^t\cdot  M_2 \cdot O_NM_3O_N^t\cdots O_NM_{n-1}O_N^t\cdot M_n)\big]  }\\
=&- \kern-0.5em
\mathop{\sum_{k = 1}^{n-1}}_{k\mathrm{\ odd}}
\mathbb{E}\big[\Tr(O_N M_1 O_N^t \cdots O_N M_{k} O_N^t \cdot (M_{k+1}\cdots  O_NM_{n-1}O_N^t\cdot M_n)^t)\big]\\
+&
\mathop{\sum_{k = 1}^{n-1}}_{k\mathrm{\ odd}}
\mathbb{E}\big[\Tr(O_N M_1O_N^t\cdots O_NM_{k}O_N^t)\cdot \Tr(M_{k+1}\cdots  O_NM_{n-1}O_N^t\cdot M_n)\big]\\
+&
\mathop{\sum_{k = 3}^{n-1}}_{k\mathrm{\ odd}}
\mathbb{E}\big[\Tr(O_N M_1O_N^t\cdots M_{k-1}\cdot (O_NM_{k}O_N^t\cdots O_NM_{n-1}O_N^t\cdot  M_n)^t)\big]\\
-&
\mathop{\sum_{k = 3}^{n-1}}_{k\mathrm{\ odd}}
\mathbb{E}\big[\Tr(O_N M_1O_N^t\cdots M_{k-1})\cdot \Tr(O_NM_{k}O_N^t\cdots O_NM_{n-1}O_N^t\cdot  M_n)\big]
\end{align*}
\end{proposition}
Proposition \ref{prop:int_by_parts} can be seen as an
orthogonal version of the Sch\-winger-Dyson (or master loop)
equation for unitary matrices (see \cite[Proof of Theorem
  5.4.10.]{agz} and \cite{gm,gn}).
\begin{proof}
For legibility we shall write $O$ instead of $O_N$. In order
to compute
$$\mathbb{E}\left[\Tr(O M_1 O^t M_2 O M_3 O^t \cdots O M_{n-1}O ^t M_n)\right]=\mathbb{E}\left[\Tr(O g(O))\right],$$
where $g(O)=M_1 O^t M_2 O M_3 O^t \cdots O M_{n-1} O^t M_n$,
we will proceed using integration by parts \eqref{intbypart}. Recall that
\begin{equation*}
\mathbb{E}[\Gamma(f,g)(O)]=\mathbb{E}[-\Delta f(O)\otimes g(O)]
\end{equation*}
where $f=id$, and $g(O)=M_1O^tM_2OM_3O^t\cdots OM_{n-1}O^t M_n$. We have
$$-\Delta f \otimes g=(N-1)id\otimes g$$
and
\[
\Gamma(id,g)=\sum_{1\leq a< b\leq N}K_{ab}id\otimes
\partial_{K_{ab}}g.
\]
Using $\partial_{K_{ab}}id=K_{ab}id$ and
$\partial_{K_{ab}}\iota=-\iota K_{ab}$, we compute more
explicitly
\begin{align*}
\partial_{K_{ab}}g(O)&=\sum_{\substack{1\leq k \leq n\\k\text{ odd}}}M_1O^t\cdots OM_{k}(-O^tK_{ab})M_{k+1}\cdots M_{n-1}O^t M_n\\
&+\sum_{\substack{3\leq k \leq n-1\\k\text{ odd}}}M_1O^t\cdots M_{k-1}K_{ab}OM_{k}O^t\cdots M_{n-1}O^t M_n,
\end{align*}
from which we get
\begin{align*} \lefteqn{
\Gamma(\id,   g)(O) } \\
 = &
\mathop{\sum_{1 \leq k \leq n-1,}}_{k \text{ odd} }
\sum_{1 \leq a < b \leq N} \kern-1em 
-K_{ab}O \otimes M_1O^t \cdots M_{k}O^tK_{ab}M_{k+1} \cdots M_{n-1}O^t M_n \\
+ & 
\mathop{\sum_{3 \leq k \leq n-1,}}_{k \text{ odd} }
\sum_{1\leq a< b\leq N} 
K_{ab}O\otimes M_1O^t\cdots M_{k-1}K_{ab}OM_{k}\cdots M_{n-1}O^t M_n\\
=&\sum_{\substack{1\leq k \leq n-1,\\k\text{ odd}}}\sum_{1\leq a, b\leq N}-E_{ab}O\otimes M_1O^t\cdots M_{k}O^tE_{ab}M_{k+1}\cdots M_{n-1}O^t M_n\\
+&\sum_{\substack{1\leq k \leq n-1,\\k\text{ odd}}}\sum_{1\leq a, b\leq N}E_{ab}O\otimes M_1O^t\cdots M_{k}O^tE_{ba}M_{k+1}\cdots M_{n-1}O^t M_n\\
+&\sum_{\substack{3\leq k \leq n-1,\\k\text{ odd}}}\sum_{1\leq a, b\leq N} E_{ab}O\otimes M_1O^t\cdots M_{k-1}E_{ab}OM_{k}\cdots M_{n-1}O^t M_n\\
-&\sum_{\substack{3\leq k \leq n-1,\\k\text{ odd}}}\sum_{1\leq a, b\leq N} E_{ab}O\otimes M_1O^t\cdots M_{k-1}E_{ba}OM_{k}\cdots M_{n-1}O^t M_n,
\end{align*}
where we used again $$\sum_{1\leq a< b\leq N}K_{ab}\otimes K_{ab}=P-T.$$
Now, we consider the map $F(X\otimes Y)=\Tr(XY)$, and the equation~\eqref{intbypart} reduces to
\begin{equation}
\mathbb{E}[F(-\Delta f(O_N)\otimes g(O_N)]=\mathbb{E}[F(\Gamma(f,g)(O_N))].\label{intbyparttwo}
\end{equation}
One one hand,
\begin{align*}
F(-\Delta f(O_N)\otimes g(O)=&F((N-1)O\otimes g(O))\\
=&(N-1)\Tr(Og(O)).
\end{align*}
On the other hand, using 
\[
\sum_{a,b}\Tr(E_{ab}XE_{ab}Y)=\Tr(XY^t) \text{ and } \sum_{a,b}\Tr(E_{ab}XE_{ba}Y)=\Tr(X)\Tr(Y), 
\]
we have
\begin{align*}\lefteqn{%
F(\Gamma(f,g)(O)) } \\
= & \mathop{\sum_{1 \leq k \leq n-1, }}_{k \text{ odd} }
\sum_{1\leq a, b\leq N}-\Tr(E_{ab}O M_1O^t\cdots OM_{k}O^tE_{ab}M_{k+1}\cdots  OM_{n-1}O^t\cdot M_n)\\
+ &
\sum_{\substack{1\leq k \leq n-1,\\k\text{ odd}}}\sum_{1\leq a, b\leq N}\Tr(E_{ab}O M_1O^t\cdots OM_{k}O^tE_{ba}M_{k+1}\cdots  OM_{n-1}O^t\cdot M_n)\\
+ &
\sum_{\substack{1\leq k \leq n-1,\\k\text{ odd}}}\sum_{1\leq a, b\leq N}\Tr( E_{ab}O M_1O^t\cdots M_{k-1}E_{ab}OM_{k}O^t\cdots OM_{n-1}O^t\cdot  M_n)\\
+ &
\sum_{\substack{1\leq k \leq n-1,\\k\text{ odd}}}\sum_{1\leq a, b\leq N} -\Tr(E_{ab} O M_1O^t\cdots M_{k-1}E_{ba}OM_{k}O^t\cdots OM_{n-1}O^t\cdot  M_n)\\
= &
-\sum_{\substack{1\leq k \leq n-1\\k\text{ odd}}}\Tr(O M_1O^t\cdots OM_{k}O^t\cdot (M_{k+1}\cdots  OM_{n-1}O^t\cdot M_n)^t)\\
+ &
\sum_{\substack{1\leq k \leq n-1\\k\text{ odd}}}\Tr(O M_1O^t\cdots OM_{k}O^t)\cdot \Tr(M_{k+1}\cdots  OM_{n-1}O^t\cdot M_n)\\
+ &
\sum_{\substack{3\leq k \leq n-1\\k\text{ odd}}}\Tr(O M_1O^t\cdots M_{k-1}\cdot (OM_{k}O^t\cdots OM_{n-1}O^t\cdot  M_n)^t)\\
- &
\sum_{\substack{3\leq k \leq n-1\\k\text{ odd}}}\Tr(O M_1O^t\cdots M_{k-1})\cdot \Tr(OM_{k}O^t\cdots OM_{n-1}O^t\cdot  M_n)
\end{align*}
So \eqref{intbyparttwo} can be written
\begin{align*}
&(N-1)\cdot \mathbb{E}\left[\Tr(O_NM_1O_N^t\cdot  M_2 \cdot O_NM_3O_N^t\cdots O_NM_{n-1}O_N^t\cdot M_n)\right]=\\
&-\sum_{\substack{1\leq k \leq n-1\\k\text{ odd}}}\mathbb{E}\left[\Tr(O_N M_1O_N^t\cdots O_NM_{k}O_N^t\cdot (M_{k+1}\cdots  O_NM_{n-1}O_N^t\cdot M_n)^t)\right]\\
&+\sum_{\substack{1\leq k \leq n-1\\k\text{ odd}}}\mathbb{E}\left[\Tr(O_N M_1O_N^t\cdots O_NM_{k}O_N^t)\cdot \Tr(M_{k+1}\cdots  O_NM_{n-1}O_N^t\cdot M_n)\right]\\
&+\sum_{\substack{3\leq k \leq n-1\\k\text{ odd}}}\mathbb{E}\left[\Tr(O_N M_1O_N^t\cdots M_{k-1}\cdot (O_NM_{k}O_N^t\cdots O_NM_{n-1}O_N^t\cdot  M_n)^t)\right]\\
&-\sum_{\substack{3\leq k \leq n-1\\k\text{ odd}}}\mathbb{E}\left[\Tr(O_N M_1O_N^t\cdots M_{k-1})\cdot \Tr(O_NM_{k}O_N^t\cdots O_NM_{n-1}O_N^t\cdot  M_n)\right]
\end{align*}
\end{proof}

\begin{remark}
Let us consider the statement of Proposition
\ref{prop:int_by_parts} when $n = 1$. We have
\[
(N - 1) \rE(\Tr(OM_1 O^t M_2))
=
\rE(\Tr(O M_1 O^t)\Tr(M_2)) - \rE(\Tr(O M_1 O^t M_2^t))
\]
\[
=
\Tr(M_1) \Tr(M_2) - \rE(\Tr(O M_1 O^t M_2^t)), 
\]
and then 
\[
(N - 1) \rE(\Tr(OM_1 O^t M_2^t))
=
\rE(\Tr(O M_1 O^t)\Tr(M_2^t)) - \rE(\Tr(O M_1 O^t M_2))
\]
\[
=
\Tr(M_1) \Tr(M_2) - \rE(\Tr(O M_1 O^t M_2)). 
\]
Putting these together we compute that
\[
\rE(\Tr( O M_1 O^t M_2)) = N^{-1}\Tr( M_1) \Tr(M_2).
\]
When we use Proposition \ref{prop:int_by_parts}, we will be
assuming that the $M_i$'s are random but independent from
$O$. We can then write
\[
\mathbb{E}\big(\Tr(O M_1 O^t  M_2 \cdots O M_{n-1}O^t M_n)\big)
\]
as double integral (by Fubini's theorem). Since our
functions $f$ and $g$ are polynomial, the derivative with
respect to $t$ is also continuous, thus justifying
differentiation under the integral sign. This elevates the
conclusion of Proposition \ref{prop:int_by_parts} to the
case where the $M_i$'s are independent from $O$. For
example, when $n=2$ we get
\[
\rE(\Tr( O M_1 O^t M_2)) = N^{-1}\rE( \Tr( M_1) \Tr(M_2) ).
\]
\end{remark}


\section{Asymptotic Freeness of Orthogonally Invariant Ensembles}\label{sec:asymptotic freeness of orthogonally invariant matrices}

Suppose $\cA_{1, \sN}, \dots, \cA_{s, \sN} \subseteq
M_N(\cL^{\infty-})$ are symmetric subalgebras of $N \times
N$ random matrices, such that the entries of the ensembles
form independent sets of random variables. The notations
$o(N^{-k})$ and $O(N^{-k})$ mean as $N \to \infty$.

We assume that each subalgebra $\cA_{j, \sN}$ is finitely
generated and the number of generators is independent of
$N$. We also assume that we have $n$ polynomials, which do
not depend on $N$, and that $P_1, \dots, P_n \in \cA_{j,
  \sN}$ are the result of applying the polynomials to the
generators.

We assume that the elements of each $\cA_{i, \sN}$ have a
limit real second order distribution and a limit real
infinitesimal distribution. The form of the second order
distribution will not be important, but the part we need is
the existence of limits for cumulants of traces and the
infinitesimal law. This means that if, for some $j$, $P_1,
\dots, P_n \in \cA_{j, \sN}$ are our polynomials, then we
have that for each $i$ we have
\begin{equation}\label{eq:assumption one}
\rE(\tr(P_i)) = \tau(p_i) + N^{-1} \tau'(p_i) + o(N^{-1}),
\end{equation}
and that
\begin{equation}\label{eq:assumption two}
k_2(\Tr(P_1), \Tr(P_2)) \to \tau_2(p_1, p_2) \mbox{\ as\ } N \to \infty
\end{equation}
and that
\begin{equation}\label{eq:assumption three}
k_r(\Tr(P_1), \dots, \Tr(P_r)) = o(1) \quad\mbox{for\ } r \geq 3
\end{equation}

We call property (\ref{eq:assumption one}) the existence of
a \textit{limit real infinitesimal law}, and properties
(\ref{eq:assumption two}), and (\ref{eq:assumption three})
the existence of a \textit{limit real second order law}.

\begin{remark}
A careful examination of the proof of Theorem 41 of
\cite{mp} shows that one only needs to assume $(iii)$ with
$o(1)$ replaced by $O(1)$. As this is the only place where
(\ref{eq:assumption one}) is used, we could just as well
assume that $k_r(\Tr(P_1), \dots, \Tr(P_r)) = o(1)$ for $r
\geq 3$. See remark 31 of \cite{mp}.

In \cite[Theorem 54]{mp} it was shown that if $\{\cA_{1,
  \sN}\}_\sN$ and $\{\cA_{2, sN}\}_\sN$ satisfy
(\ref{eq:assumption two}) and (\ref{eq:assumption three}),
then the algebra they generate also satisfies
(\ref{eq:assumption two}) and (\ref{eq:assumption three})
provided that at least one is orthogonally invariant. By
induction on $s$ and the associative law
(\cite[Prop. 29]{mp}) we get the same conclusion for $s \geq
2$ provided all or all but one of the algebras $\cA_i$ are
orthogonally invariant. In addition an explicit rule was
given for computing the limit in (\ref{eq:assumption two})
from the individual limit distributions of the $\cA_{i,
  \sN}$'s. This is the rule given by Emily Redelmeier
\cite{r} and called real second order freeness.
\end{remark}

In this section we will use the results of \cite{mp} to show
that if each of the subalgebras satisfies
(\ref{eq:assumption one}), (\ref{eq:assumption two}), and
(\ref{eq:assumption three}), then the subalgebra they
generate also satisfies (\ref{eq:assumption one}),
(\ref{eq:assumption two}), and (\ref{eq:assumption three}),
provided, again, that all or all but one are orthogonally
invariant.

We do this by a double induction. The first, or outer
induction, is on $s$ the number of subalgebras. The second,
or inner, induction is on the number of occurrences of a
fixed subalgebra in a word.

\begin{lemma}\label{lemma:outside induction}
Let $\cA_{1, \sN}, \dots, \cA_{s, \sN} \subseteq
M_N(\cL^{\infty-})$ be unital subalgebras such that the
entries of matrices from different subalgebras form
independent sets. Suppose that all, or all but one, of the
subalgebras are orthogonally invariant, and suppose that
each of the subalgebras satisfies $(\ref{eq:assumption
  one})$, $(\ref{eq:assumption two})$ and
$(\ref{eq:assumption three})$. Then the subalgebra generated
by $\cA_{1, \sN}, \dots, \cA_{s, \sN}$ satisfies
$(\ref{eq:assumption one})$, $(\ref{eq:assumption two})$ and
$(\ref{eq:assumption three})$.
\end{lemma}

\begin{proof}
We prove this by induction on $s$. When $s = 1$, there is
nothing to prove because there is only one subalgebra and it
already satisfies (\ref{eq:assumption one}),
(\ref{eq:assumption two}), and (\ref{eq:assumption
  three}). So let us start the induction with $s = 2$. This
means that we only have two subalgebras. By \cite[Prop. 29
  and Thm. 54]{mp} we have that the algebra generated by
$\cA_{1, \sN}$ and $\cA_{2, \sN}$ satisfies
(\ref{eq:assumption two}), and (\ref{eq:assumption
  three}). So we must show that (\ref{eq:assumption one})
also holds. To this end we let $P_1, \dots, P_r \in
M_N(\cL^{\infty-})$ be such that $P_i \in \cA_{j_i,\sN}$
with $j_1 \not= j_2$, $j_2 \not = j_3$, \dots, $j_{r-1} \not
= j_r$. The existence of the limit real second order
distribution means, in particular, that there is an algebra
$\cA$ with involution $a \to a^t$ and a trace $\tau$ such
that
\[
\rE(\tr(P_1 \cdots P_r)) = \tau(p_1 \cdots p_r) + o(1).
\]
To prove the lemma we must replace the convergence above with the stronger statement 
\[
\rE(\Tr(P_1 \cdots P_r)) = N \tau(p_1 \cdots p_r) + \tau'(p_1 \cdots p_r) +o(1),
\]
where $\tau'(p_1 \cdots p_r)$ is some unknown (for the
moment) function of $p_1, \dots, p_r$. (Of course it is the
goal of the paper to find this function, but first we have
to prove its existence).

As noted in the proof of \cite[Prop. 52]{mp}, by traciality
we may assume that $r$ is even and that $P_1$ is from
$\cA_{1, \sN}$, and $\cA_{1, \sN}$ is orthogonally
invariant. Then as in \cite[Eq. (33)]{mp}
\begin{align}\label{eq:joint genus expansion}\lefteqn{
\rE(\Tr(P_1 \cdots P_r)) = \rE(\Tr(OP_1 O^t P_2 \cdots  O P_{r-1} O^t P_r)) }\notag \\
& =
\sum_{p, q \in \cP_2(r)}
\Wg(p, q) \  \rE(\Tr_{\pi_{p,q}^o}(Q_1, \dots, Q_{r -1}))\ \rE(\Tr_{\pi_{p, q}^e}(Q_2, \dots, Q_{r})) 
\end{align}
where $\Wg$ is the orthogonal Weingarten function, $\pi_{p,
  q}^o$ and $\pi_{p, q}^e$ are the restrictions of
$\pi_\peq$ to the odd and even numbers of $[r]$
respectively, and each $Q_i$ is either $P_i$ or $P_i^t$, all
depending on the pairings $p$ and $q$ (see \cite[Lemmas 5 \&
  13]{mp}) for the precise dependence). How the permutation
$\pi_\peq$ is obtained and the permutation $\epsilon p
\delta q \epsilon$ will be explained in \S \ref{subsec:genus
  calculations}.

Now for each of the three factors in the sum above we have
an $1/N$-expansion of the form:
\begin{align*}\lefteqn{
\rE(\Tr_{\pi_{p,q}^o}(Q_1, \dots, Q_{r -1}))}  \\
& =
N^{\#(\pi_{p,q}^o)} \tau_{\pi_{p,q}^o}(q_1, \dots, q_{r -1}) + 
N^{\#(\pi_{p,q}^o) - 1}\, \tau'_{\pi_{p,q}^o}(q_1, \dots, q_{r -1}) \\
&\mbox{}  + o(N^{\#(\pi_{p,q}^o) - 1}),
\end{align*}
\begin{align*} \lefteqn{
\rE(\Tr_{\pi_{p,q}^e}(Q_2, \dots, Q_{r})) } \\
& =
N^{\#(\pi_{p,q}^e)} \tau_{\pi_{p,q}^e}(q_2, \dots, q_{r}) + 
N^{\#(\pi_{p,q}^e)- 1} \tau'_{\pi_{p,q}^e}(q_2, \dots, q_{r}) + o(N^{\#(\pi_{p,q}^e)- 1} ),
\end{align*}
and the orthogonal Weingarten function has the well known
asymptotic expansion in $1/N$:
\[
\Wg(p, q) = w_1(p, q)N^{-n + \#(p \vee q)} + w'_1(p, q)N^{-n + \#(p \vee q) -1} + O(N^{-n + \#(p \vee q) - 2}).
\]
For a $(p, q) \in \cP_2(r)$ we have the product 
\begin{equation}\label{eq:27 terms}
\Wg(p, q) 
\rE(\Tr_{\pi_{p,q}^o}(Q_1, \dots, Q_{r -1}))
\rE(\Tr_{\pi_{p,q}^e}(Q_2, \dots, Q_{r})) 
\end{equation}
with leading order $N^{-n + \#(p \vee q) + \#(\pi_\peq^o) +
  \#(\pi^e_\peq)}$ $ = N^{-n + \#(p \vee q) +
  \#(\pi_\peq)}$. According to Propositions
\ref{prop:disconnected genus} and \ref{prop:connected genus}
(\textit{infra}), there is an integer, $g_{p,q} \geq 0$ such
that
\[
-n + \#(p \vee q) + \#(\pi_\peq)
=
\begin{cases} 1 - 2 g_{p,q} & \epsilon p \delta q \epsilon \mbox{\ leaves\ } [n]\mbox{\ invariant}\\
- g_{p,q}& \epsilon p \delta q \epsilon \mbox{\ does not leave\ } [n]\mbox{\ invariant}.
\end{cases}
\]
The only terms in (\ref{eq:27 terms}) that we are interested
in are the terms of order $N^1$ and $N^0$; we will also show
that the remaining terms are $o(N^0)$. The only way to get a
term of order $N^1$ is for $\epsilon p \delta q \epsilon$ to
leave $[n]$ invariant and $g_{p,q} = 0$. When $g_{p,q} = 0$
we are in the non-crossing case: $\epsilon p \delta q
\epsilon|_{[n]} \in \NC(r)$, and by asymptotic freeness the
coefficient of $N^1$ is $\tau(p_1 \cdots p_r)$ as claimed in
(\ref{eq:assumption one}). The term of next lower order will
be $N^{-1}$ which is smaller than $o(1)$.

Now let us consider the case when $\epsilon p \delta q
\epsilon$ does not leave $[n]$ invariant. In this case the
highest degree term we can get is $N^0$, and this only when
$g_{p,q} = 0$ and $\epsilon p \delta q \epsilon \in
S_\NC^\delta(r, -r)$. This coefficient of $N^0$ is
$\tau'(p_1 \cdots p_r)$. The terms of next lower order are
of order $N^{-n + \#(p \vee q)} o(N^{\pi_\peq})$, which when
$g_{p,q} = 0$ is $o(1)$.

These two possibilities mean that
\[
\rE(\Tr(P_1 \cdots P_r))
=
N \tau(p_1, \dots, p_r) + \tau'(p_1, \dots, p_r) + o(1).
\]
This proves the lemma when $s = 2$. 

Now suppose $s > 2$. Let $P_1, \dots P_r$ be such that $P_i
\in \cA_{j_i, \sN}$ and $j_1 \not = j_2$ \dots, $j_{r-1}
\not = j_r$. We want to show that (\ref{eq:assumption one})
holds i.e.
\[
\rE(\Tr(P_1 \cdots P_r)) = N \tau(p_1 \cdots p_r) + \tau'(p_1 \cdots p_r) +o(1).
\]
By traciality, we may assume that $j_r \not = j_1$ and that
$\cA_{j_1, \sN}$ is orthogonally invariant. Then we write
$P_1 \cdots P_r$ as $M_1 M_2 \cdots M_n$ with $n$ even,
$M_{2l-1} \in \cA_{j_1, \sN}$, and each of $M_{2l}$ in the
algebra generated by the $\cA_{j_i,\sN}$'s where $j_i \not =
j_1$, which we denote by $\tilde{\cA}_{j_2, \sN}$, just for
the duration of this proof. By our induction hypothesis we
know $\cA_{j_1, \sN}$ and $\tilde{\cA}_{j_2,\sN}$ satisfy
(\ref{eq:assumption one}), (\ref{eq:assumption two}), and
(\ref{eq:assumption three}), so by the the first part of the
proof we get that the algebra generated by $\cA_{j_1, \sN}$
and $\tilde{\cA}_{j_2, \sN}$ satisfies (\ref{eq:assumption
  one}), (\ref{eq:assumption two}), and (\ref{eq:assumption
  three}). However this last algebra is just the algebra
generated by $\cA_{1, \sN}, \dots, \cA_{s, \sN}$.
\end{proof}

In the next lemma we assume that we have subalgebras $\cA_1,
\dots, \cA_s$ satisfying the hypotheses of Lemma
\ref{lemma:outside induction}. For notational convenience we
shall make the dependence on $N$ implicit.

The next lemma is an immediate corollary of Lemma \ref{lemma:outside induction}. 

\begin{lemma}\label{lemma:first inside induction}
Let $P_1, \dots, P_r \in M_N(\cL^{\infty-})$ be such that
$P_i \in \cA_{j_i}$ with $j_1 \not= j_2$, $j_2 \not = j_3$,
\dots, $j_{r-1} \not = j_r$, and for each $i$,
$\rE(\Tr(P_i)) = \tau'(p_i) + o(1)$.

Then for $r \geq 1$
\begin{equation}\label{eq:first intermediate limit distribution}
\rE(\Tr(P_1 \cdots P_r))
= O(1)
\end{equation} 
\end{lemma}

\begin{lemma}\label{lemma:inside induction}
Let $P_1, \dots, P_r \in M_N(\cL^{\infty-})$ be such that
$P_i \in \cA_{j_i}$ with $j_1 \not= j_2$, $j_2 \not = j_3$,
\dots, $j_{r-1} \not = j_r$, $j_r \not = j_1$, and for each
$i$, $\rE(\Tr(P_i)) = \tau'(p_i) + o(1)$. Let $V \in
\ker(j)$, with $V = \{l_1, \dots, l_n\}$ be the block
containing $1$ with $1 = l_1$, $l_{m-1} + 1 < l_m$, and $l_n
< r$.

Then for $r > 1$
\begin{equation}\label{eq:intermediate limit distribution}
\rE(\Tr(P_1 \cdots P_r))
=
\mathop{\sum_{m=2}^{n}}
\tau(p_1 \cdots p_{l_m-1} p_r^t \cdots p_{l_m}^t )  + O(N^{-1}).
\end{equation}
\end{lemma}

\begin{proof} 
By traciality we may assume that $\cA_{j_1}$ is orthogonally invariant.

As in the hypothesis, $ V \in \ker(i) \in \cP(r)$ is the
block containing $1$ and we write $V = \{ l_1, \dots, l_n\}$
with $1 = l_1$, $l_{k-1} + 1 < l_k$ and $l_n < r$ (because
we have assumed that $j_r \not = j_1$). Let
\[
M_1 = P_1, \quad M_2 = P_2 \cdots P_{l_2 -1},\quad M_3 = P_{l_2},\]
and in general 
\[
M_{2k-1} = P_{l_k} \mbox{\ and\ } M_{2k} = P_{l_k + 1} \cdots P_{l_{k+1} - 1}. 
\]
Then 
\[
P_1 \cdots P_r = M_1 M_2 \cdots M_{2n-1} M_{2n}
\]
with $M_1, M_3, \dots, M_{2k-1}$ all in $\cA_{j_1}$ and
$M_2, M_4, \dots, M_{2n}$ all in $\cA_{j_2} \cup \cdots \cup
\cA_{j_r}$. Since $\cA_{j_1}$ is orthogonally invariant we
have
\[
\rE(\Tr(M_1 M_2 \cdots M_{2n-1} M_{2n}) 
=
\rE(\Tr(OM_1 O^t M_2 \cdots O M_{2n-1} O^t M_{2n})).
\]
In the expression on the right each $M_{2k-1}$ has been
replaced by $OM_{2k-1}O^t$ and each $M_{2k}$ has been left
unchanged. In Proposition \ref{prop:int_by_parts} we assumed
that the $M_k$ matrices were constant and here they are
random but independent of the $O$ matrices, so by using
Fubini's theorem we may apply Proposition
\ref{prop:int_by_parts} to obtain that
\begin{align}\lefteqn{
\rE(\Tr(P_1 \cdots P_r)) = \rE(\Tr(M_1 M_2 \cdots M_{2n - 1} M_{2 n}))} \notag \\
&\mbox{} =
\rE(\Tr(OM_1O^t M_2 \cdots OM_{2n - 1}O^t M_{2 n})) \notag \\
& \mbox{} =
\label{eq:first case}
\frac{-1}{N - 1} \mathop{\sum_{k=1}^{2n - 1}}_{k \mathrm{\ odd}}
\rE( \Tr(O M_1 O^t M_2 \cdots O M_kO^t (M_{k + 1} O \cdots O^t M_{2n})^t )) \\
& \mbox{} +
\label{eq:second case}
\frac{1}{N - 1} \mathop{\sum_{k=3}^{2n - 1}}_{k \mathrm{\ odd}}
\rE( \Tr(O M_1 O^t M_2 \cdots O^t M_{k-1} (O M_{k} O^t \cdots O^t M_{2n})^t )) \\
& \mbox{} +
\label{eq:third case}
\frac{1}{N - 1} \mathop{\sum_{k=1}^{2n - 1}}_{k \mathrm{\ odd}}
\rE( \Tr(O M_1 O^t M_2 \cdots O M_{k}O^t) \Tr(M_{k+1} O \cdots O^t M_{2n}) )) \\
& \mbox{} -
\label{eq:fourth case}
\frac{1}{N - 1} \mathop{\sum_{k=3}^{2n - 1}}_{k \mathrm{\ odd}}
\rE( \Tr(O M_1 O^t M_2 \cdots O^t M_{k-1}) \Tr(OM_{k} O^t \cdots O^t M_{2n}) )).
\end{align}

Now let us consider the limit as $N \to \infty$ of each of
the four terms. Let us start with (\ref{eq:first case}). As
$k$ is odd we have $M_k \in \cA_{j_1}$ and $M_{2n} \in
\cA_{j_r}$ with $j_1 \not = j_r$, thus
\begin{align*} \lefteqn{
\frac{1}{N - 1}\rE( \Tr(O M_1 O^t M_2 \cdots O M_kO^t (M_{k + 1} O \cdots O^t M_{2n})^t ))} \\
& \mbox{} =
\frac{1}{N - 1}\rE( \Tr(O M_1 O^t M_2 \cdots O M_kO^t M_{2n}^t O \cdots O^t M_{k+1})) \\
& \mbox{} =
\frac{1}{N - 1}\rE( \Tr(M_1 M_2 \cdots M_k M_{2n}^t  \cdots  M_{k+1}^t)) \\
& \mbox{} =
\frac{1}{N - 1}\rE( \Tr(P_1 P_2 \cdots P_{l_k} P_{r}^t  \cdots  P_{l_k+1}^t)) = O(N^{-1}),\\
\end{align*}
where the last equality holds by Lemma \ref{lemma:outside induction}.

Next consider (\ref{eq:second case}), with $k = 2m - 1 \geq 3$ and $2 \leq m \leq n$.
\begin{align*} \lefteqn{
\frac{1}{N - 1}\rE( \Tr(O M_1 O^t M_2 \cdots O^t M_{k-1} (O M_{k} O^t \cdots O^t M_{2n})^t ))} \\
& \mbox{} =
\frac{1}{N - 1}\rE( \Tr(O M_1 O^t M_2 \cdots O^t M_{2(m-1)} M_{2n}^t O \cdots O M_{2m-1}^t O^t)) \\
& \mbox{} =
\frac{1}{N - 1}\rE( \Tr(M_1 M_2 \cdots M_{2(m-1)} M_{2n}^t  \cdots  M_{2m-1}^t)) \\
& \mbox{} =
\tau(p_1 \cdots p_{l_m-1} p_r^t \cdots p_{l_m}^t ) + O(N^{-1}),
\end{align*}
where we obtain the last equality because, by Lemma
\ref{lemma:outside induction}, Eq.~(\ref{eq:assumption one})
holds:
\begin{align*}\lefteqn{
\rE(\Tr(P_1 \cdots P_{l_m-1} P_r^t \cdots P_{l_m}^t ) } \\
& =
N \tau(p_1 \cdots p_{l_m-1} p_r^t \cdots p_{l_m}^t )
+ 
\tau'(p_1 \cdots p_{l_m-1} p_r^t \cdots p_{l_m}^t )
+
o(1).
\end{align*}
Next consider (\ref{eq:third case}), with $k = 2m - 1 \geq
3$ and $1 \leq m \leq n$.  According to our notation we have
\[
\rE(\Tr(M_1 M_2 \cdots M_{k}))\ \rE(\Tr(M_{k+1}  \cdots  M_{2n}) ))
\]
\[
=
\rE(\Tr( P_1 \cdots P_{l_m}))\ \rE(\Tr( P_{l_m +1} \cdots P_{r})),
\]
and by induction (on $n$) both of these factors are bounded
functions of $N$ by Lemma \ref{lemma:first inside
  induction}. Hence
\begin{align*}\lefteqn{
\frac{1}{N - 1}
\rE( \Tr(O M_1 O^t M_2 \cdots O M_{k}O^t) \Tr(M_{k+1} O \cdots O^t M_{2n}) ))} \\
& =
\frac{1}{N - 1} \cov(\Tr(O M_1 O^t M_2 \cdots O M_{k}O^t), 
                     \Tr(M_{k+1} O \cdots O^t M_{2n}) )) \\
& \quad+
\frac{1}{N - 1} \rE(\Tr(O M_1 O^t M_2 \cdots O M_{k}O^t)) \
                 \rE(\Tr(M_{k+1} O \cdots O^t M_{2n}) )) \\
& =
\frac{1}{N - 1} \rE(\Tr(M_1 M_2 \cdots M_{k})) \
                 \rE(\Tr(M_{k+1}  \cdots  M_{2n}) )) + O(N^{-1}) \\
&=
                 \frac{1}{N - 1} \rE(\Tr(P_1 \cdots P_{l_m})) \
                 \rE(\Tr(P_{l_m+1}  \cdots  P_{r}) )) + O(N^{-1}) \\
& = O(N^{-1}).
\end{align*}
Finally consider (\ref{eq:fourth case}), $k = 2m - 1$ odd
with $m \geq 2$ we have
\[
\rE(\Tr(M_1 M_2 \cdots M_{k-1}))\ \rE(\Tr(M_{k}  \cdots  M_{2n}) ))
\]
\[
=
\rE(\Tr( P_1 \cdots P_{l_m -1 }))\ \rE(\Tr( P_{l_m} \cdots P_{r})).
\]
Now again by Lemma \ref{lemma:first inside induction} , both
of these factors are bounded functions of $N$, hence
\begin{align*}\lefteqn{
\frac{1}{N - 1}
\rE( \Tr(O M_1 O^t M_2 \cdots O^t M_{k-1}) \Tr(O M_{k} O^t \cdots O^t M_{2n}) ))} \\
& =
\frac{1}{N - 1} \cov(\Tr(O M_1 O^t M_2 \cdots O^t M_{k-1}), 
                     \Tr(OM_{k} O^t \cdots O^t M_{2n}) )) \\
& \quad+
\frac{1}{N - 1} \rE(\Tr(O M_1 O^t M_2 \cdots O^t M_{k-1})) \
                 \rE(\Tr(O M_{k} O^t \cdots O^t M_{2n}) )) \\
& =
\frac{1}{N - 1} \rE(\Tr(M_1 M_2 \cdots M_{k-1})) \
                 \rE(\Tr(M_{k}  \cdots  M_{2n}) )) + O(N^{-1}) \\
&=
                 \frac{1}{N - 1} \rE(\Tr(P_1 \cdots P_{l_m-1})) \
                 \rE(\Tr(P_{l_m}  \cdots  P_{r}) )) + O(N^{-1}) \\
& = O(N^{-1}).
\end{align*}
Thus we have 
\[
\rE(\Tr(P_1 \cdots P_r))
=
\sum_{m=2}^n \tau(p_1 \cdots p_{l_m-1} p_r^t \cdots p_{l_m}^t ) + O(N^{-1})
\]
\end{proof}

\begin{remark}
When $n = 1$ in Equation (\ref{eq:intermediate limit
  distribution}), the Lemma implies that $\rE(\Tr(P_1
\ab\cdots\ab P_r)) = o(1)$.
\end{remark}

The idea of the next lemma is somewhat standard, for the
convenience of the reader we give here a proof.

\begin{lemma}\label{lemma:standard reduction}
Suppose $(\cA, \phi)$ is a non-commutative probability space
and $\cA_1,\ab \dots, \cA_s \subseteq \cA$ are unital
subalgebras which are freely independent. Suppose $a_1,
\dots, a_n \in \cA$ with $a_i \in \cA_{j_i}$ and $\phi(a_i)
= 0$ for $1 \leq i \leq n$. Suppose that there is $1 \leq k
\leq n-1$ such that $j_1 \not = j_2$, \dots, $j_{k-1} \not =
j_k$, $j_{k+1} \not = j_{k+2}$, \dots, $j_{n-1} \not =
j_n$. Then $\phi(a_1 \cdots a_n) = 0$ unless $n$ is even and
$k = n/2$, in which case we have
\begin{equation}\label{eq:reduction by folding}
\phi(a_1 \cdots a_n) = \phi(a_1 a_n) \cdots \phi(a_k a_{k+1}).
\end{equation}
\end{lemma}

\begin{proof}
Let $\tilde{a}_k = a_k a_{k+1} - \phi(a_k a_{k+1})$. If $j_k
\not = j_{k+1}$ then $\tilde{a}_k = a_k a_{k+1}$ and
$\phi(a_1 \cdots a_n) = 0$. If $j_k = j_{k+1}$ then
\[
\phi(a_1 \cdots a_n) = \phi(a_1 \cdots \tilde{a}_k a_{k+2} \cdots a_n)
+ \phi(a_k a_{k+1}) \phi(a_1 \cdots a_{k-1} a_{k+2} \cdots a_n)
\]
\[
=
\phi(a_k a_{k+1}) \phi(a_1 \cdots a_{k-1} a_{k+2} \cdots a_n).
\]
Then by induction on $n$ we must have $k-1 = (n-2)/2$ and if so, then 
\[
\phi(a_1 \cdots a_{k-1} a_{k+2} \cdots a_n)
=
\phi(a_1a_n) \cdots \phi(a_{k-1} a_{k+2}). 
\]
Hence we have equation (\ref{eq:reduction by folding}).
\end{proof}

\begin{lemma}\label{lemma:intermediate reduction}
Suppose we have a real non-commutative probability space
$(\cA, \tau)$ and symmetric subalgebras $\cA_1, \dots,
\cA_s$ which are free with respect to $\tau$. Suppose that
we have centred elements $p_1, \dots, p_n \in \cA$ with $p_i
\in \cA_{j_i}$ and $j_1 \not = j_2$, \dots, $j_{n-1} \not =
j_n$ and $j_n \not = j_1$. Let $V$ be the block of $\ker(j)$
containing $1$. Write $V = \{l_1, \dots, l_t\}$ with $1= l_1
< \cdots < l_n < \cdots < l_t < n$. Then
\[
\tau(p_1 \cdots p_{l_m - 1} (p_{l_m} \cdots p_n)^t) = 0
\]
unless $l_{m} - 1 = n/2$ in which case we have $($setting $k= n/2)$
\[
\tau(p_1 \cdots p_{l_m -1}(p_{l_m} \cdots p_n)^t)
=
\tau(p_1 p_{k+1}^t) \cdots \tau(p_k p_{n}^t).
\]
\end{lemma}

\begin{proof}
By Lemma \ref{lemma:standard reduction}, we must have $n$
even, or else we get $0$. When $n$ is even, again by Lemma
\ref{lemma:standard reduction}, $\tau(p_1 \cdots p_{l_m
  -1}(p_{l_m} \cdots p_n)^t)$ factors. If $l_m - 1 > n/2$
then $\tau(p_{n/2} p_{n/2 + 1}) = 0$ is a factor and we get
$0$. If $l_m - 1 < n/2$ then $\tau(p_{l_m+1}^t p_{l_m}^t) =
0$ is a factor and we get $0$. When $l_m - 1 = n/2$, we get,
again by Lemma \ref{lemma:standard reduction}, exactly what
is claimed in the Lemma.
\end{proof}

\begin{lemma}\label{lemma:final reduction}
Let $P_1, \dots, P_n \in M_N(\cL^{\infty-})$, with $n \geq
2$, be such that $P_i \in \cA_{j_i}$ with $j_1 \not= j_2$,
$j_2 \not = j_3$, \dots, $j_{n-1} \not = j_n$, $j_n \not =
j_1$, and for each $i$, $\rE(\Tr(P_i)) = \tau'(p_i) + o(1)$.
Then
\begin{equation}
\rE(\Tr(P_1 \cdots P_n))
=
\begin{cases}
O(N^{-1}) & \mbox{\ for\ } n \mbox{\ odd} \\
\tau(p_1 p_{k+1}^t) \cdots \tau(p_k p_{n}^t) + O(N^{-1}) & \mbox{\ for\ } n = 2k
\end{cases}.
\end{equation}
\end{lemma}

\begin{proof}
Let $V \in \ker(j)$ be the block containing $1$ and write $V
= \{ l_1, \dots, l_t\}$ with $l_1 = 1$, $l_{m-1}+1 < l_m$,
and $l_t < n$. By Lemma \ref{lemma:inside induction} we have
\[
\rE(\Tr(P_1 \cdots P_n))
=
\mathop{\sum_{m=2}^{t}}
\tau(p_1 \cdots p_{l_m-1} p_t^t \cdots p_{l_m}^t )  + O(N^{-1}).
\] 

When $n$ is odd we have by Lemma \ref{lemma:standard
  reduction} that $\tau(p_1 \cdots p_{l_m-1} p_t^t \cdots
p_{l_m}^t ) = 0$. When $n = 2k$ is even, we have by Lemma
\ref{lemma:intermediate reduction} that
\begin{align*}
\rE(\Tr(P_1 \cdots P_n))
& =
\mathop{\sum_{m=2}^{t}}
\tau(p_1 \cdots p_{l_m-1} p_t^t \cdots p_{l_m}^t )  + O(N^{-1})  \\
& \mbox{} = 
\tau(p_1 p_{k+1}^t) \cdots \tau(p_k p_{n}^t) + O(N^{-1}).
\end{align*}
\end{proof}

\begin{theorem}\label{thm:asymptotic real infinitesimal freeness}
Let, for each $N$, $\cA_{1, \sN}, \dots, \cA_{s, \sN}
\subseteq M_N(\cL^{\infty-})$ be unital symmetric
subalgebras such that the entries of matrices from different
subalgebras form independent sets. Suppose that all, or all
but one, of the subalgebras are orthogonally invariant, and
suppose that each of the subalgebras satisfies
$(\ref{eq:assumption one})$, $(\ref{eq:assumption two})$ and
$(\ref{eq:assumption three})$. Then the subalgebras $\cA_{1,
  \sN}, \dots, \cA_{s, \sN}$ are asymptotically real
infinitesimally free.
\end{theorem}

\begin{proof}
We know that by \cite[Thm. 54]{mp} there is a real second
order probability space $(\cA, \tau, \tau_2)$ and unital
symmetric subalgebras such that the limit distribution of
each $\{ \cA_{i, \sN} \}_\sN$ is that of $\cA_i$ and that
the subalgebras $\cA_1, \dots, \cA_s$ are real second order
free. By Lemma \ref{lemma:outside induction}, we also know
that the joint infinitesimal law of $\{\cA_{1, \sN}, \dots,
\cA_{s, \sN}\}$ has a limit infinitesimal law. By Lemma
\ref{lemma:final reduction} we know that the joint
infinitesimal distribution satisfies the conditions $(i)$
and $(ii)$ of Proposition \ref{prop:non-tracial
  case}. Finally by Proposition \ref{prop:non-tracial case}
the joint distribution is the joint distribution of real
infinitesimally free subalgebras (as defined in Definition
\ref{def:real infinitesimal freeness}).
\end{proof}

\section{Real Infinitesimal Free Cumulants \\ and the Moment-cumulant formula}\label{section:real infinitesimal cumulants}

When we pass from complex infinitesimal freeness to real
infinitesimal freeness we need to use the symmetric
non-crossing annular permutations introduced in \cite{mvb}.
\begin{figure}
\includegraphics[width=10em]{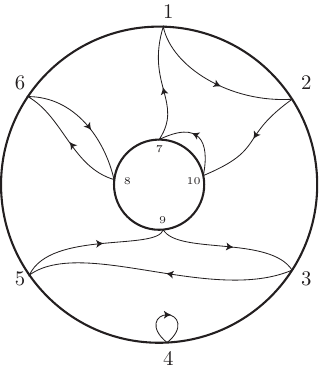}
\caption{\small\label{fig:non-crossing_annular} A non crossing permutation of a $(6, 4)$-annulus.}
\end{figure}

Let us recall that $S_{NC}(p, q)$ denotes the set of
non-crossing permutations of a $(p, q)$-annulus. These are
permutations of $[p+q]$ such that the cycles can be drawn in
an annulus, with $p$ points on the outer circle and $q$
points on the inner circle, in such a way that the cycles do
not cross, see Figure \ref{fig:non-crossing_annular}. See
\cite[\S5.1]{ms} for a full definition and examples. The
simplest characterization of these permutations is through
Euler's formula for the genus of a triangulated surface, but
now transferred into the symmetric group: $\pi \in S_\NC(p,
q)$ if and only if
\[
\pi \vee \gamma_{p,q} = 1_{p+q} \qquad \mbox{and}\qquad
\#(\pi) + \#(\pi^{-1} \gamma_{p,q})  = p + q,
\]
where $\pi \vee \gamma_{p,q} = 1_{p+q}$ means that at least
one cycle of $\pi$ meets both cycles of $\gamma_{p, q}$, and
$\gamma_{p,q} \in S_{p+q}$ is the permutation with two
cycles $(1, 2, 3, \dots, p)(p+1, \dots, p+q)$, and $\#(\pi)$
is the number of cycles in the cycle decomposition of $\pi$,
counting cycles of length $1$.

\subsection{Symmetric annular non-crossing permutations}\label{subsec:symmetric annular non-crossing permutations}

In \cite[\S2]{mvb} a subset of non-crossing annular
permutations was identified. These are the
\textit{symmetric} non-crossing annular permutations.  We
denote this subset by $S_\NC^\delta(n, -n)$; the definition
is recalled in the next paragraph.

Let $n \geq 2$ be an integer. By $S_{\pm n}$ we mean the
permutations of $[\pm n] = \{ \pm 1, \dots,\ab \pm n\}$. We
let $\delta \in S_{\pm n}$ be the permutation with $n$
cycles each of size $2$ given by $\delta(k) = - k$. Next we
let $\gamma_n \in S_n$ be the permutation with the long
cycle $(1, 2, 3, \dots, n)$. Throughout the paper we shall
observe the following convention. If $\pi \in S_n$ then we
consider $\pi$ to also be the permutation of $[\pm n]$ which
acts trivially on $\{ -1, \dots, -n\}$. With this convention
we have that given a $\pi \in S_n$, $\delta \pi \delta$ is a
permutation on $[\pm n]$ which acts trivially on $[n] = \{1,
\dots, n\}$.  Thus with $\gamma_n = (1, \dots, n)$ we have
that
\[
\gamma_n \delta \gamma_n^{-1} \delta = (1, \dots, n)(-n, \dots, -1). 
\]
We shall say that a permutation, $\pi$, is a
\textit{pairing} if all cycles have length $2$. This is
equivalent to saying that $\pi^2$ is the identity and $\pi$
has no fixed points.
\begin{figure}
\begin{center}
\includegraphics{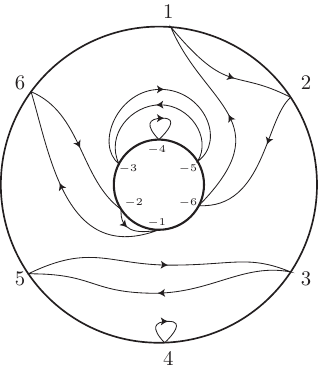}
\end{center}
\caption{\label{fig:symmetric annular}\small A \textit{symmetric} non-crossing annular permutation on a $(6, -6)$-annulus. Note that the orientation of the points on the two circles is the same. This is the opposite convention used in Figure \ref{fig:non-crossing_annular}. }
\end{figure}

\begin{notation}
Let $S_\NC^\delta(n, -n)$ be the permutations $\sigma \in S_{\pm n}$ such that
\begin{itemize}
\item
$\sigma \vee \gamma_n \delta \gamma_n^{-1} \delta = 1_{\pm n} $, and

\smallskip

\item
$\#(\sigma) + \#(\sigma^{-1} \gamma_n \delta \gamma_n^{-1} \delta) = 2 n$, and

\smallskip

\item
$\sigma \delta$ is a pairing.
\end{itemize} 
The first two assumptions mean that $\sigma$ is non-crossing
annular on a $(n, -n)$-annulus, the third is a symmetry
condition explained below. See Figure \ref{fig:symmetric
  annular}.
\end{notation}

\begin{remark}\label{remark:conjugate pairs}
It was noted in \cite[Remark 17]{mvb} that if we set $p =
\sigma \delta$ with $\sigma \in S_\NC^\delta (n, -n)$ then
$\sigma = p \delta$ and thus $\sigma$ is the product of two
pairings and hence the cycles of $\sigma$ appear in
conjugate pairs: $c$ and $c'$ with $c' = \delta c^{-1}
\delta$ (see \cite[Lemma 2]{mp}). Thus the cycle
decomposition of $\sigma$ can always be written $c_1 c_1'
\cdots c_k c_k'$ with $c_i' = \delta c_i^{-1} \delta$. We
call the pair $\{ c_i, c_i'\}$ a \textit{conjugate
  pair}. The blocks of $\sigma\delta \vee \delta$ are
exactly $c_i \cup c_i'$ (again, see \cite[Lemma 2]{mp}).
\end{remark}

\begin{notation}\label{notation:sncdelta}
Let $(\cA, \tau, \tau')$ be a tracial real non-commutative
probability space and $\sigma \in S_\NC^\delta(n, -n)$. We
define $ \kappa_{\sigma/2}(a_1, \dots, a_n) $ as
follows. For each pair of conjugate cycles $\{ c, c'\}$ of
$\sigma$, we write $c = (i_1, \dots, i_k, -j_l, \dots,
-j_1)$ with $i_1, i_2, \dots, i_k$ and $j_1, j_2, \dots, <
j_l$ in cyclic order\footnote{We say that $i_1 , i_2 \dots,
i_k$ are in \textit{cyclic order} if they are in the same
order as in the orbit of $i_1$ under $\gamma_n$.}, we have
the contribution of the pair $\{c, c'\}$ is
\[
\kappa_{k+l}(a_{i_1}, \dots, a_{i_k}, a_{-j_l}^t, \dots, a_{-j_1}^t).
\]
By taking the product over all conjugate pairs $\{c, c'\}$ we get $\kappa_{\sigma/2}$: 
\[
\kappa_{\sigma/2}(a_1, \dots, a_n)
= \kern-2em
\mathop{\prod_{\{c, c'\} \in \sigma}}_{c = (i_1, \dots, i_k, -j_{l}, \dots, -j_1)}
\kern-2em
\kappa_{k+l}(a_{i_1}, \dots, a_{i_k}, a_{-j_l}^t, \dots, a_{-j_1}^t).
\]
The $\sigma/2$ in the notation is meant to signal that we
only take one member of each conjugate pair. Since $\tau$ is
tracial and invariant under the transpose the contributions
of $c$ and $c'$ are the same.
\end{notation}

Recall from \cite[Def. 9.21]{ns} the Kreweras complement,
$K(\pi)$, of a non-crossing partition $\pi$. In \cite{ns},
$K(\pi)$ is defined is defined using the lattice property of
$\NC(n)$. Since $S_\NC^\delta(n, -n)$ is not a lattice,
another construction is needed; and here we will use an
alternative construction found in \cite{ns}. For $\pi \in
\NC(n)$, $K(\pi) = \pi^{-1}\gamma_n$, where $\gamma_n =(1,
\dots, n)$ is the permutation of $[n]$ mentioned in
\S\ref{subsec:symmetric annular non-crossing permutations}
and $\pi$ is the permutation obtained by putting the
elements of each block in increasing order, see
\cite[Ex. 18.25]{ns}.

\begin{definition}\label{def:symmetric kreweras complement}
For $\sigma \in S_\NC^\delta(n, -n)$, let $K^\delta(\sigma)
= \delta\gamma^{-1}_n\delta \sigma^{-1} \gamma_n$. Note that
$K^\delta(\sigma) \in S_\NC^\delta(n, -n)$. We call
$K^\delta(\sigma)$ the \textit{symmetric Kreweras
  complement} of $\sigma$, or for brevity, the
\textit{Kreweras complement} of $\sigma$. In \cite[Notation
  23]{mvb}, this complement was written as $K(\sigma)$. The
same construction was also used in \cite{r}.
\end{definition}

\begin{definition}\label{definition:real infinitesimal cumulants}
Let $(\cA, \tau, \tau')$ be a tracial real infinitesimal
probability space. For $a_1, \dots, a_n \in \cA$ we set for
$n = 1$
\[
\kappa_1'(a_1) = \tau'(a_1)\ \mbox{and}
\]
and for $n \geq 2$
\begin{equation}\label{eq:real infinitesimal moment cumulant}
\tau'(a_1 \cdots a_n) =
\sum_{\pi \in \NC(n)} \partial \kappa_\pi(a_1, \dots, a_n)
+
\sum_{\sigma \in S_\NC^\delta (n, -n)} \kappa_{\sigma/2}(a_1, \dots, a_n).
\end{equation}
Some explicit examples are presented in \S \ref{section:small example}. 
\end{definition}

\subsection{Spatial Derivatives}
As in the usual moment-cumulant formula
(\ref{eq:moment_cumulant}), the equation above inductively
defines the infinitesimal cumulants. For example
\[
\kappa_2'(a_1, a_2) = \tau'(a_1 a_2) -[ \tau'(a_1) \tau(a_2)  + \tau(a_1) \tau'(a_2)  + \tau(a_1a_2^t) - \tau(a_1) \tau(a_2^t)],
\]
\[
= \sum_{\pi \in \NC(2)} \mu(\pi, 1_2) \partial \tau_\pi(a_1, a_2) - \tau(a_1a_2^t) + \tau(a_1) \tau(a_2^t)
\]
where $\mu$ is the M\"obius function of $\NC(n)$. In order
to make this fit into a convenient moment-cumulant relation
we introduce the \textit{spatial} derivative,
$\dot\kappa_n$. In this example, this will amount to
rewriting the equation above as
\[
\kappa_2'(a_1, a_2)  + \dot{\kappa}_2(a_1, a_2) =
\sum_{\pi \in \NC(2)} \mu(\pi, 1_2) \partial \tau_\pi(a_1, a_2),
\]
where $\dot{\kappa}_2(a_1, a_2) = \tau(a_1a_2^t) - \tau(a_1)
\tau(a_2^t) = \kappa_2(a_1, a_2^t)$. This example is meant
to illustrate the name spatial derivative, in that we do not
consider the infinitesimal distributions of $a_1$ and $a_2$,
but the first order joint distribution of $a_1$ and $a_2^t$,
just as was done in \cite[Def.~2.10]{cmss}.

When $n = 3$ we can start with equation (\ref{eq:real
  infinitesimal moment cumulant}) and use the equation above
to write $\kappa_3'$ in terms of $\tau$ and $\tau'$. The
right hand side of (\ref{eq:real infinitesimal moment
  cumulant}) will have nine terms containing a $\kappa_i'$
(for some $i$) and six terms not containing a $\kappa_i'$,
this will be the spatial part. The terms containing a
$\kappa_i'$ can be grouped (after some calculation) into
$\sum_{\pi \in \NC(3)} \mu(\pi, 1_3) \partial\tau_\pi(a_1,
a_2, a_3)$. The six terms not containing a $\kappa_i'$:
\begin{align*} 
\kappa_3(a_1, a_2, a_3^t) &+ \kappa_3(a_2, a_3, a_1^t) + \kappa_3(a_3, a_1, a_2^t)
+ \kappa_1(a_1) \kappa_2(a_2, a_3^t)  \\
& + 
\kappa_2(a_1, a_3^t) \kappa_1(a_2) + \kappa_2(a_1, a_2^t) \kappa_1(a_3),
\end{align*}
can be expanded into 21 terms with a $\tau$ but no
$\tau'$. So instead, we write these remaining six as
\[
\dot{\kappa}_3(a_1, a_2, a_3) + \kappa_1(a_1) \dot\kappa_2(a_2, a_3) + \kappa_1(a_2) \dot\kappa_2(a_3, a_1) + \kappa_1(a_3) \dot\kappa_2(a_1, a_2),
\]
where $\dot\kappa_2$ is as above and
\[
\dot\kappa_3(a_1, a_2, a_3) =
\kappa_3(a_1, a_2, a_3^t) + \kappa_3(a_2, a_3, a_1^t) + \kappa_3(a_3, a_1, a_2^t).
\]
The three terms in $\dot\kappa_3$ are the sum of
$\kappa_{\sigma/2}$ as $\sigma$ runs over the three
permutations in Figure \ref{fig:all through blocks}.

Thus
\[
\kappa_3'(a_1, a_2, a_3) + \dot{\kappa}_3(a_1, a_2, a_3) = \sum_{\pi \in \NC(3)} \mu(\pi, 1_3)
\partial\tau_\pi(a_1, a_2, a_3), \quad\mbox{and}
\]
\[
\tau'(a_1a_2a_3)
=
\sum_{\pi \in \NC(3)} \{\partial\kappa_\pi(a_1, a_2, a_3) + \delta\kappa_\pi(a_1, a_2, a_3)\}, \quad\mbox{where}
\]
$ \delta\kappa_\pi = \sum_{V \in \pi} \dot\kappa_{{\scriptscriptstyle |V|}} \prod_{W \not= V} \kappa_{{\scriptscriptstyle |W|}}$.

The need to give a general definition for $\dot{\kappa}_n$
motivates the presentation of the set $S_\NC^{\delta, a}(n,
-n)$ given in Definition \ref{def:all through cycles} below.

\begin{definition}\label{def:all through cycles}
We let $S_\NC^{\delta, a}(n, -n) \subseteq S_\NC^\delta(n,
-n)$ be those annular permutations for which every cycle
meets both cycles of $\gamma_n \delta \gamma_n^{-1}
\delta$. The superscript `$a$' means that \textit{all}
cycles meet both cycles of $\gamma_n \delta \gamma_n^{-1}
\delta$. These permutations, arising in earlier work on
second order freeness \cite[Prop. 6.1]{cmss}, will play a
prominent role in constructing the real infinitesimal
cumulants below.
\end{definition}

\begin{notation}\label{notation:spatial derivative}
For $n \geq 2$, let $\dot{\kappa}_n(a_1, \dots, a_n) =
\kern-1em\ds\sum_{\sigma \in S_\NC^{\delta, a}(n, -n) }
\kern-1em \kappa_{\sigma/2}(a_1, \dots, a_n)$. If $\pi \in
\NC(n)$ we let
\[
\delta\kappa_\pi(a_1, \dots, a_n) =
\sum_{V \in \pi} \dot{\kappa}_{{\scriptscriptstyle |V|}}(a_1, \dots, a_n | V) \prod_{W \not= V}
\kappa_{{\scriptscriptstyle |W|}}(a_1, \dots, a_n| W), 
\]
where the product is over all blocks $W$ not equal to
$V$. Then we set $\nabla = \partial + \delta$.
\end{notation}

\begin{figure}[t]
\begin{center}
\includegraphics[width=10em]{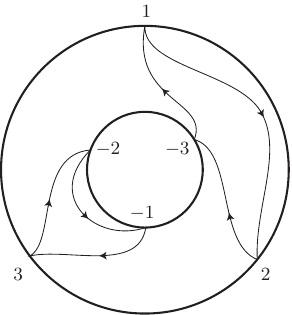}\hfill\includegraphics[width=10em]{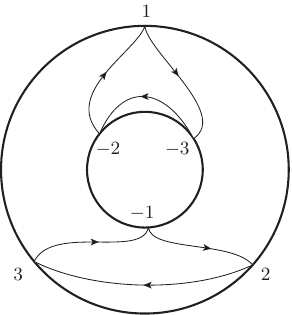}\hfill\includegraphics[width=10em]{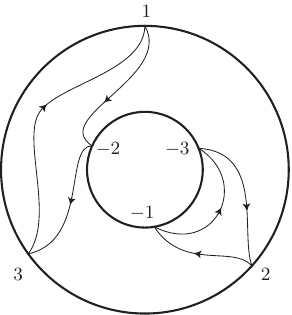}
\end{center}
\caption{\label{fig:all through blocks}\small When $n = 3$ there are three elements in $S_\NC^{\delta, a}(3, -3)$, they are displayed above.}
\end{figure}

\begin{notation}
Recall that $S_\NC^{\delta, a}(n, -n)$ is the subset of $S_\NC^{\delta}(n, -n)$ where all cycles are through cycles. Given $\pi \in \NC(n)$ and $V \in \pi$ we let $S_\NC^{\delta}(n, -n)_{\pi, V} $
\[
=
\left\{ \sigma \in S_\NC^{\delta}(n, -n) \ \vrule height 3em depth 2.5 em width 0.4pt\quad 
\vcenter{\hsize=15.1em\noindent$\mbox{every cycle of\ } \sigma \mbox{\ is either a cycle of}\\ 
\pi\delta\pi^{-1}\delta \mbox{\ or contained in\ } V \cup \delta(V)$, \\
moreover any cycle of $\sigma$ contained in $V \cup \delta(V)$ must be a through cycle}\quad
\right\}.
\]
\end{notation}

\begin{lemma}\label{lemma:annular reduction}
\[
S_\NC^\delta (n, -n)
=
\bigcup_{\pi \in \NC(n)}
\bigcup_{V \in \pi}
S_\NC^{\delta}(n, -n)_{\pi, V},
\]
and the union is disjoint.
\end{lemma}

\begin{proof}
For each $\pi$ and $V$ we have $S_\NC^{\delta}(n, -n)_{\pi, V} \subseteq S_\NC^{\delta}(n, -n)$. If 
\[
\sigma \in 
S_\NC^{\delta}(n, -n)_{\pi_1, V_1}
\bigcap
S_\NC^{\delta}(n, -n)_{\pi_2, V_2}
\]
then $V_1 \cup \delta(V_1)$ and $V_2 \cup \delta(V_2)$ are
both the union of the through cycles of $\sigma$; so $V_1 =
V_2$. All the non-through cycles of $\sigma$ are cycles of
$\pi_1 \delta\pi_1^{-1}\delta$ and of $\pi_2
\delta\pi_2^{-1}\delta$. So we also have $\pi_1 =
\pi_2$. This proves disjointness.

Given $\sigma \in S_\NC^\delta(n, -n)$ we let $V \subset
[n]$ be such that $V \cup \delta(V)$ is the union of through
cycles of $\sigma$. Let the remaining blocks of $\pi$ be the
cycles of $\sigma$ contained in $[n]$. Then $\pi$ is
non-crossing, see \cite[Def. 8 and Thm. 13]{kms} and the
proof of Proposition 19 and Figure 6 in \cite{mvb}.
\end{proof}

With this notation we can now re-write the relation
(\ref{eq:real infinitesimal moment cumulant}) between
moments and cumulants.
\begin{figure}[t]
\begin{center}
\includegraphics[width=10em]{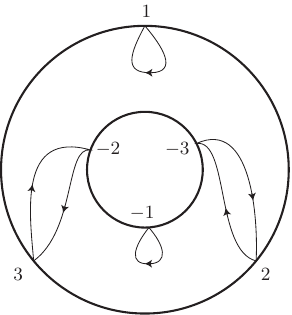}\hfill\includegraphics[width=10em]{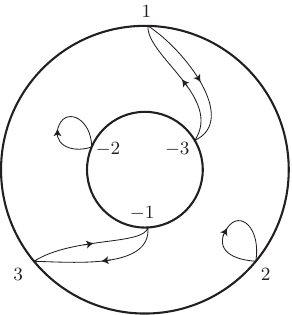}\hfill\includegraphics[width=10em]{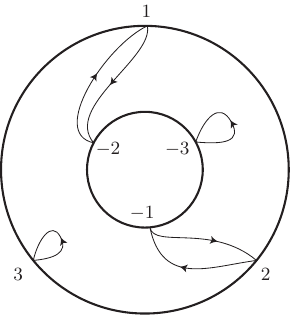}
\end{center}
\caption{\label{fig:remaining three three}\small When $n =
  3$ there are six elements in $S_\NC^{\delta}(3, -3)$; the
  first three are displayed in Figure \ref{fig:all through
    blocks}, the remaining three are displayed above.}
\end{figure}

\subsection{The Moment-Cumulant Formula}

From Definition \ref{eq:real infinitesimal moment cumulant} we have
\[
\tau'(a_1 \cdots a_n) =
\sum_{\pi \in \NC(n)} \partial \kappa_\pi(a_1, \dots, a_n)
+
\sum_{\sigma \in S_\NC^\delta (n, -n)} \kappa_{\sigma/2}(a_1, \dots, a_n).
\]
Using Notation \ref{notation:spatial derivative} and Lemma
\ref{lemma:annular reduction} we may rewrite the second
term:
\[
\sum_{\sigma \in S_\NC^\delta (n, -n)} \kappa_{\sigma/2}(a_1, \dots, a_n)
=
\sum_{\pi \in \NC(n)} \delta \kappa_\pi(a_1, \dots, a_n).
\]

This gives us a way of writing an infinitesimal moment as a
sum of cumulants:
\begin{align}\label{eq:infinitesimal moment cumulant}
\tau'(a_1 \cdots a_n)
=
\sum_{\pi \in \NC(n)} \nabla \kappa_\pi(a_1, \dots, a_n) . 
\end{align}

\begin{theorem}\label{thm:moment cumulant}

\begin{align}\label{eq:infinitesimal cumulant moment}
\nabla \kappa_n(a_1, \dots, a_n) & =  \kappa_n'(a_1, \dots, a_n) + \dot \kappa_n(a_1, \dots, a_n) \\
& =
\sum_{\pi \in \NC(n)} \mu(\pi, 1_n) \partial \tau_\pi (a_1, \dots, a_n). \notag
\end{align}
\end{theorem}

\begin{proof}
Because we were able to write an infinitesimal moment as a
sum over $\NC(n)$ in equation (\ref{eq:infinitesimal moment
  cumulant}), we are entitled to use Möbius inversion (see
\cite[Lect. 11]{ns}) to write the cumulant $\nabla
\kappa_n(a _1, \dots, a_n)$ as a sum of moments, again
indexed by $\NC(n)$.
\end{proof}

\begin{remark}
As observed above, $\dot\kappa_n(a_1, \dots, a_n)$ depends
only on the joint distribution of $\{a_1, a_1^t, \dots, a_n,
a_n^t\}$. This leads to the question as to how do we write
$\{\dot\kappa_n\}_n$ in terms of the joint distribution of
$\{a_1, a_1^t, \dots, a_n, a_n^t\}$? There is an answer to
this question this when $a_1 = \cdots =a_n$ and $a_i =
a_i^t$ for $1 \leq i \leq n$.

Using Equation (\ref{eq:infinitesimal moment cumulant}), we
solve for $\dot\kappa_n$ in terms of moments.  \begingroup
\renewcommand{\arraystretch}{1.5}
\begin{center}
\begin{tabular}{c|l}
$n$ & $\dot\kappa_n$ \\ \hline
$2$ & $m_2 - m_1^2$ \\
$3$ & $3 m_3 - 9 m_1 m_2 + 6m_1^3$\\
$4$ & $6m_4 - 24 m_1 m_3 - 11m_2^2 + 58 m_1^2 m_2 - 29 m_1^4$ \\
$5$ & $10 m_5 - 50 m_1 m_4 - 45 m_2 m_3 + 145 m_1^2 m_3 + 1 35 m_1 m_2^2$ \\
    & $\qquad\mbox{} - 325 m_1^3 m_2 + 130 m_1^5$ \\
$6$ & $15 m_6 - 90 m_1 m_5 - 81 m_2 m_4 + 306 m_1^2 m_4 - 39 m_3^2 + 558 m_1 m_2 m_3$ \\
    & $\qquad\mbox{} - 780 m_1^3 m_3 + 88 m_2^3 - 1101 m_1^2 m_2^2 + 1686 m_1^4 m_2 - 562 m_1^6$ \\    
\end{tabular}
\end{center}
There does not seem to be a discernible pattern, however if
we turn the moments into cumulants then the pattern of
Notation \ref{notation:spatial derivative} emerges.
\begin{center}
\begin{tabular}{c|l}
$n$ &  $\dot\kappa_n$ \\ \hline
$2$ & $\kappa_2$ \\
$3$ & $3 \kappa_3$ \\
$4$ & $6\kappa_4 + \kappa_2^2$  \\
$5$ & $10 \kappa_5 + 5 \kappa_2 \kappa_3$ \\
$6$ & $15 \kappa_6 + 9 \kappa_2 \kappa_4 + 6 \kappa_3^2 + \kappa_2^3$\\
$7$ & $21 \kappa_7 + 14 \kappa_2 \kappa_5 + 21 \kappa_3 \kappa_4 
+ 7 \kappa_2^2 \kappa_3$\\
$8$ & $28 \kappa_8 + 20 \kappa_2 \kappa_6 + 32 \kappa_3 \kappa_5 + 18 \kappa^2_4 + 12 \kappa_2^2 \kappa_4 + \kappa_2^4$\\  
\end{tabular}
\end{center}
\endgroup There is a simple formula connecting the cumulant
generating function $C(z) = 1 + \sum_{n=1}^\infty \kappa_n
z^n$ and the one for the spatial derivatives $\dot C(z) =
\sum_{n=2}^\infty \dot \kappa_n z^n$:
\[
\dot C(x) =
\bigg(x  \frac{\partial}{\partial x} \log
\bigg[ \frac{x C(y) - y C(x)}{x - y} \bigg] \bigg) \Bigg\vert_{y \mapsto x}
 = \frac{1}{2} \frac{x^2 C''(x)}{C(x) - x C'(x)}.
\]
This formula will be used to define the real infinitesimal
$R$-transform. The proof will be presented elsewhere, since
we don't need it for the results in this paper. Note the
similarity to \cite[Eq.~(51)]{cmss}.
\end{remark}

\subsection{Genus calculations}\label{subsec:genus calculations}

In the proof of Lemma \ref{lemma:outside induction} we
deferred a calculation for the genus expansion used in
Equation (\ref{eq:joint genus expansion}). We will complete
it now using the notation presented in \S
\ref{subsec:symmetric annular non-crossing permutations}.

In the next two lemmas we suppose that we have pairings $p,
q \in \cP_2(r)$ with $r$ even. We let $\epsilon =
(\epsilon_1, \dots, \epsilon_r) = (1, -1, 1, -1, \dots, 1,
-1) \in \bZ_2^r$. As in \cite[\S3 p. 995]{m} we also regard
$\epsilon$ as a permutation in $S_{\pm r}$ by setting
$\epsilon( k) = k$ when $\epsilon_{|k|} = 1$ and $\epsilon(
k) = - k$ when $\epsilon_{|k|} = -1$. The permutation
$\epsilon p \delta q \epsilon \in S_{\pm n}$ will then
determine the order of the contribution to Equation
(\ref{eq:joint genus expansion}) of the term for the pair
$(p, q)$. We shall see that for each pair $(p, q)$ there
will be an integer $g_{p,q} \geq 0$ such that either $-n +
\#(p \vee q) + \#(\pi_\peq) = 1 -2 g_{p,q}$ or $-n + \#(p
\vee q) + \#(\pi_\peq) = - g_{p,q}$.

\begin{proposition}\label{prop:disconnected genus}
Suppose $\epsilon p \delta q \epsilon$ leaves $[n]$
invariant, then there is an integer $g_{p,q} \geq 0$ such
that $-n + \#(p \vee q) + \#(\pi_\peq) = 1 -2 g_{p,q}$.
\end{proposition}

\begin{proof}
Let $\rho_0 = \epsilon p \delta q \epsilon$. 
Since $K^\delta(\rho_0) \delta = \delta \gamma^{-1} \delta
\epsilon q \delta p \epsilon \gamma \delta = \delta
\gamma^{-1} \epsilon \delta p \delta q \delta \epsilon
\gamma \delta$ is a pairing, the cycles of
$K^\delta(\rho_0)$ occur in conjugate pairs $c, c'$ with $c'
= \delta c^{-1} \delta$. We choose one representative of
each pair and after taking the absolute value of each cycle
entry, we get a permutation $\pi_\peq \in S_n$.  This
permutation is \textit{pairity preserving} in the
terminology of \cite[Lemma 13]{mp}. We then have
$\#(\pi_\peq) = \frac{1}{2} \#(K^\delta(\rho_0))$. In this
Proposition we are in the case where one cycle of each
conjugate pair only permutes elements of $[n]$. So to create
$\pi_\peq$ we just choose these cycles.  By \cite[Lemma
  4]{mp} we have $\#(p \vee q) = \frac{1}{2} \#(p\delta q) =
\frac{1}{2} \#(\rho_o) = \#(\rho)$ where $\rho$ is the
restriction of $\rho_0$ to $[n]$.  Now we have
\[
\#(\rho) + \#(\rho^{-1}\gamma) + \#(\gamma) = n + 2 - 2 g_{p,q}
\]
for some integer $g_{p, q} \geq 0$. Also $ \frac{1}{2} \#(K^\delta(\rho_0)) = \#(\rho^{-1}\gamma)$. Thus, as claimed, 
\[
-n +  \#(p \vee q) + \#(\pi_\peq) = 1 -2 g_{p,q}.
\]
\end{proof}

\begin{proposition}\label{prop:connected genus}
Suppose $\epsilon p \delta q \epsilon$ does not leave $[n]$
invariant, then there is an integer $g_{p,q} \geq 0$ such
that $-n + \#(p \vee q) + \#(\pi_\peq) = - g_{p,q}$.
\end{proposition}

\begin{proof}
As above we let $\rho_= \epsilon p \delta q \epsilon$. In
this case, the subgroup generated by $rho_0$ and $\delta
\gamma^{-1} \delta \gamma$ acts transitively on $[\pm n]$,
so we have
\[
\#(\rho_0) + \#( K^\delta(\rho_0)) + \#(\delta \gamma^{-1} \delta \gamma)
=
2n + 2 - 2 g_{p,q},
\]
for some integer $g_{p, q} \geq 0$. As in the proof of
Prop. \ref{prop:disconnected genus} we have $\#(p \vee q) =
\frac{1}{2} \#(\rho_o)$, $\#(\pi_\peq) = \frac{1}{2}
\#(K^\delta(\rho_0))$. Thus
\[
-n +  \#(p \vee q) + \#(\pi_\peq) = - g_{p,q}.
\]
\end{proof}

\section{Real infinitesimal cumulants and real infinitesimal freeness}
\label{sec:real infinitesimal cumulants and real infinitesimal freeness}

In this section we shall prove Theorem \ref{thm:freeness and
  the vanishing of mixed cumulants} which shows that for a
tracial real infinitesimal probability space, real
infinitesimal freeness and the vanishing of real
infinitesimal cumulants are equivalent. The proof depends on
Theorem \ref{thm:product rule}, which gives the formula for
cumulants with products as entries. The proof of Theorem
\ref{thm:product rule} is lengthy and is postponed until
Section \ref{section:product formula}.

In a complex infinitesimal probability space the complex
infinitesimal cumulants $\{ \kappa^\sC_n\}_n$ are defined by
the moment-cumulant equation
\[
\tau'(a_1 \cdots a_n)=\sum_{\pi\in NC(n)}\partial\kappa_\pi^{\sC}(a_1, \dots, a_n). 
\]
The addition of the second term on the right hand side of
Equation (\ref{eq:infinitesimal moment cumulant}) then
affects the infinitesimal cumulants $\{\kappa_n'\}_n$. For
example when we take the limit distribution of the GOE, the
real infinitesimal cumulants now vanish; see
\cite[Thm. 24]{m}. In the case of a real Wishart matrix with
$c ' = 0$, we also have that the real infinitesimal
cumulants vanish; see \cite[Thm. 21]{mvb}. On the other
hand, in a recent paper of Popa, Szpojankowski, and Tseng,
the complex infinitesimal cumulants of the limit joint
distribution of a GUE matrix and its transpose are shown to
either $0$ or $1$, depending on the word. See \cite[\S
  5]{pst}.

\begin{theorem}\label{thm:freeness and the vanishing of mixed cumulants}
Let $(\mathcal{A},\tau,\tau')$ be a real tracial
infinitesimal non-commutative probability space and consider
unital symmetric subalgebras $\mathcal{A}_1, \dots,
\mathcal{A}_s\subset \mathcal{A}$. Then the following
statements are equivalent:
\begin{enumerate}
\item The algebras $\mathcal{A}_1, \dots, \mathcal{A}_s$ are real infinitesimally free.
\item Mixed free cumulants and mixed real infinitesimal cumulants of the subalgebras vanish.
\end{enumerate} 
\end{theorem}
In order to prove Theorem \ref{thm:freeness and the
  vanishing of mixed cumulants} we need the formula for
cumulants with products as entries, which is presented here
as Theorem \ref{thm:product rule}. By using Theorem
\ref{thm:product rule} to prove $(i) \Rightarrow (ii)$ in
Theorem \ref{thm:freeness and the vanishing of mixed
  cumulants}, we obtain a conceptually simpler proof. Since
Theorem \ref{thm:freeness and the vanishing of mixed
  cumulants} is not used in the proof of Theorem
\ref{thm:product rule}, this application does not present a
logical problem. First let us recall a lemma from Redelmeier
\cite[Lemma 3.3]{r2}.

\begin{lemma}\label{lemma:emilys lemma}
If $n$ is odd, all permutations $\sigma\in
S^\delta_{NC}(n,-n)$ have a cycle consisting of a
single-element or a cycle containing two neighbouring
elements.

If $n$ is even, the spoke diagram
$\left\{\Big(k,-(n/2+k)\Big):1\leq k \leq n\right\}$ is the
only permutation in $S^\delta_{NC}(n,-n)$ which does not
have any single-element cycles or any cycle containing two
neighbouring elements.\label{lem:spoke_diagram}
\end{lemma}

\begin{proof}[Proof of Theorem \ref{thm:freeness and the vanishing of mixed cumulants}]
That the vanishing of mixed cumulants implies real
infinitesimal freeness follows easily from the
moment-cumulant formula, as follows. Let $a_1, \dots, a_n$
be centred and cyclically alternating. We must show that
\[
\tau'(a_1 \cdots\ab a_n) = 0 
\]
for $n \geq 3$ and odd, or for $n \geq 2$ and even, that 
\[
\tau'(a_1 \cdots a_n) = \prod_{k=1}^{n/2} \tau(a_k a_{n/2 + k}^t). 
\]
We shall use Equation  (\ref{eq:infinitesimal moment cumulant}) and consider the two terms separately. 

Let's consider $\sum_{\pi \in \NC(n)} \partial
\kappa_\pi(a_1, \dots, a_n)$. We claim that $\partial
\kappa_\pi(a_1, \dots, a_n)\ab = 0$ for all $\pi \in
\NC(n)$. If $\pi$ contains an interval of length greater
than $1$, then there will be mixed cumulants, and so
$\partial \kappa_\pi(a_1, \dots, a_n) = 0$. If $\pi$ has two
or more singletons, one of them will contribute a factor of
$\kappa_1(a_l)$, which equals $0$ by our centering
assumption. So the only possibility is that $\pi$ has only
one interval, and that interval is of length $1$. This can
only happen of $n$ is odd and the singleton is at $(n +
1)/2$ and all other blocks are of size $2$. Then $a_1$ and
$a_n$ will be in a block of size $2$ and $\kappa_2(a_1 ,
a_n) = \kappa_2'(a_1, a_n) = 0$ by our cyclically
alternating assumption.

Now let us consider the second term. The only possible
$\sigma \in S_\NC^\delta(n, -n)$ for which
$\kappa_{\sigma/2}(a_1, \dots, a_n)$ doesn't have a mixed
cumulant is when $\sigma$ is a spoke diagram. This can only
happen when $n = 2m$ is even and $\sigma = (1, -(m+1))(2,
-(m+2)) \cdots (m, -2m) \ab (m+1, -1) \cdots (2m , -m)$. In
this case $\kappa_{\sigma/2}(a_1, \dots, a_n) =
\prod_{k=1}^{n/2} \tau(a_k a_{n/2 + k}^t)$. This proves that
$(i)$ and $(ii)$ of Proposition \ref{prop:non-tracial case}
hold.

For the other direction, $(i) \Rightarrow (ii)$, note first
that real infinitesimal freeness implies the vanishing of
$\kappa'_n(a_1,\ldots, a_n)$ whenever $a_1, \ldots, a_n$ are
centred and cyclically alternating. Indeed, let us prove
this by induction on $n \geq 2$. We have by Definition
\ref{def:real infinitesimal freeness} $(ii)$ and
Eq. (\ref{eq:infinitesimal moment cumulant})
\begin{align*}
0= \tau'(a_1a_2) & = \kappa_2'(a_1, a_2) + \kappa_1'(a_1) \kappa_1(a_2) + \kappa_1(a_1) \kappa_1'(a_2) + \kappa_2(a_1, a_2^t) \\
&  \mbox{} = \kappa_2'(a_1, a_2) + \tau(a_1 a_2^t) = \kappa_2'(a_1, a_2). 
\end{align*} 
Thus $\kappa_2'(a_1, a_2) = 0$. Using induction and the
argument about the intervals of $\pi$ used above to prove
$(ii) \Rightarrow (i)$ (that vanishing of mixed cumulants
implies real infinitesimal freeness) we get that
\[
\sum_{\pi < 1_n} \partial \kappa_\pi(a_1, \dots, a_n) = 0,
\]
and for $\sigma \in S_\NC^\delta(n, -n)$, we have by Lemma
\ref{lemma:emilys lemma}, $\kappa_{\sigma/2}(a_1, \dots,
a_n) = 0$ unless $\sigma$ is a spoke diagram. Thus for $n
\geq 2$ we have for $n$ even
\[
\tau'(a_1 \cdots a_n) = \kappa_n'(a_1, \dots, a_n) + \tau(a_1a_{n/2+1}^t) \cdots \tau(a_{n/2} a_n^t),
\] 
and when $n$ is odd
\[
\tau'(a_1 \cdots a_n) = \kappa_n'(a_1, \dots, a_n).
\]
This proves that when $a_1, \dots, a_n$ are centred and
cyclically alternating we have for $n \geq 2$ that
$\kappa_n'(a_1, \dots, a_n) = 0$.

Now let us show that the same conclusion holds when we only
assume that $a_1, \dots, a_n$ are cyclically alternating. We
achieve this by showing that $\kappa_n'(a_1, \dotsm a_n) =
0$ whenever there is $l$ such that $a_l = 1$.  For
convenience of notation let us assume that $l = n$. By
(\ref{eq:infinitesimal moment cumulant}) we have by
induction
\begin{align*}
\tau'(a_1 \cdots a_n) & = 
\sum_{\pi \in \NC(n)} \partial \kappa_\pi(a_1, \dots, a_{n-1}, 1)
+ \kern-0.5em\sum_{\sigma \in S_\NC^\delta(n, -n)} \kern-0.5em
\kappa_{\sigma/2}(a_1, \dots, a_{n-1}, 1)\\
&\stackrel{(*)}{=} \kappa_n'(a_1, \dots, a_{n-1}, 1) + 
\sum_{\pi \in \NC(n-1)}\partial\kappa_\pi(a_1, \dots, a_{n-1}) \\
&\qquad\mbox{}+ 
\sum_{\sigma \in S_\NC^\delta(n-1, -(n-1))} \kappa_{\sigma/2}(a_1, \dots, a_{n-1})\\
& = \kappa_n'(a_1, \dots, a_{n-1}, 1) + \tau'(a_1, \dots, a_{n-1}),
\end{align*}
where the equality $(*)$ holds because $\partial
\kappa_\pi(a _1, \dots, a_{n-1}, 1) = 0$ unless $n$ is a
singleton of $\pi$. Hence $\kappa_n'(a_1, \dots, a_{n-1}, 1)
= 0$. Thus for $n \geq 2$ and $\centre{a}_i = a_i -
\tau(a_i)$ we have
\[
\kappa'_n(a_1, \dots, a_n)
=
\kappa'_n(\centre{a}_1, \dots, \centre{a}_n).
\]

Now let us lift the requirement that the elements are
cyclically alternating, We prove our claim by induction on
$n \geq 2$. Now suppose $a_i \in \cA_{j_i}$ and $j_1 \not =
j_2$. Then
\[
\kappa'_2(a_1, a_2)
= \kappa'_2(\centre{a}_1, \centre{a}_2) = 0,
\]
because $\centre{a}_1, \centre{a}_2$ are centred and
cyclically alternating. Our induction hypothesis will be
that for $2 \leq l < n$ we have
\[
\kappa'_l(a_1, \dots, a_n) = 0
\]
whenever the $a_i$'s are not all from the same
subalgebra. We now prove the claim for $l = n$.

Given $a_1, \dots, a_n$ we let $m_1, \dots, m_r$ be such
that for $1 \leq l \leq r$ the elements
\[
a_{m_1 + \cdots + m_{l-1}+1}, \dots, a_{m_1 + \cdots + m_l}
\] 
are all from the same subalgebra, but for adjacent $l$'s are
from different subalgebras. By the traciality of $\tau$ and
$\tau'$ we may assume this holds for cyclically adjacent
$l$'s as well. Now let
\[
A_l = a_{m_1 + \cdots + m_{l-1}+1} \cdots a_{m_1 + \cdots + m_l}.
\]
Then $A_1, \dots, A_r$ are cyclically alternating, so by our
earlier discussion we have $\kappa_r'(A_1, \dots,\ab A_r) =
0$. By Theorem \ref{thm:product rule}, the formula for
cumulants with products for entries, we have
\begin{align*}\lefteqn{
0 =  \kappa'_r(A_1, \dots, A_r)} \\
&  = 
\mathop{\sum_{\pi \in \NC(n)}}_
{\pi \vee \rho_r = 1_n}
 \partial\kappa_\pi(a_1, \dots, a_n)
+
\mathop{\sum_{\sigma \in S_\NC^\delta(n, -n)}}_
{K^\delta(\sigma) \sep \pm M}
\kappa_{\sigma/2}(a_1, \dots, a_n)
\end{align*}
where $\rho_r$ is the interval partition $\{ (1, \dots,
m_1), \dots, (m_1 + \cdots + m_{r-1} + 1, \dots,\ab m_1 +
\cdots + m_r)\}$ and $M = \{m_1, \dots, m_1 + \cdots +
m_r\}$.  By $K^\delta(\sigma) \sep \pm M$ we mean that no
two points of $\pm M$ are in the same cycle of
$K^\delta(\sigma) = \delta \gamma_n^{-1} \delta \sigma^{-1}
\gamma_n$. See \S \ref{section:product formula} for more
explanation and \S \ref{section:small example} for a small
example illustrating the definitions and statements. By
induction on $n$, the first term simplifies to
\[
\mathop{\sum_{\pi \in \NC(n)}}_
{\pi \vee \rho_r = 1_n}
 \partial\kappa_\pi(a_1, \dots, a_n) = \kappa_n'(a_1, \dots, a_n).
\]
Thus, we only have to prove
\[
\mathop{\sum_{\sigma \in S_\NC^\delta(n, -n)}}_
{K^\delta(\sigma) \sep \pm M}
\kappa_{\sigma/2}(a_1, \dots, a_n) = 0.
\]
This amounts to showing that if $K^\delta(\sigma) \sep \pm
M$ then the subgroup $\langle \sigma , \rho_r\rangle$
generated by $\sigma$ and $\rho_r$ acts transitivity on
$[\pm n]$, as this will force $\sigma$ to contain a cycle
that connects two subalgebras, which then means we we can
apply the rule for vanishing of mixed (first order)
cumulants.

Since $\sigma^{-1} = \delta \sigma \delta$ the orbits of
$\langle \sigma , \rho_r\rangle$ are symmetric with respect
to $\delta$. Suppose there is more than one orbit of
$\langle \sigma , \rho_r\rangle$, then as it is a union of
cycles of $\rho_r$ and $\delta \rho_r^{-1}\delta$ there are
$j_1, j_2$ and $k_1, k_2$ such that the orbits are contained
in
\begin{align*}\lefteqn{
\{ m_1 + \cdots + m_{j_1 -1}+1, \dots, m_1 + \cdots +m_{j_2} \}  }\\
& \mbox{} \cup \{ -(m_1 + \cdots + m_{k_1 -1}+1), \dots, -(m_1 + \cdots m_{k_2}) \}
\end{align*}
or
\begin{align*}\lefteqn{
\{ m_1 + \cdots + m_{k_1 -1}+1, \dots, m_1 + \cdots +m_{k_2} \} }\\
& \mbox{} \cup
\{ -(m_1 + \cdots + m_{j_1 -1}+1), \dots, -(m_1 + \cdots m_{j_2}) \}  
\end{align*} 
Thus $K^\delta(\sigma)$ does not separate the points of $\pm M$. Hence
\[
\mathop{\sum_{\sigma \in S_\NC^\delta(n, -n)}}_
{K^\delta(\sigma) \sep \pm M}
\kappa_{\sigma/2}(a_1, \dots, a_n) = 0
\]
and thus $\kappa_n'(a_1, \dots, a_m) = 0$. Hence mixed
cumulants vanish. This proves $(i) \Rightarrow (ii)$.
\end{proof}

\section{The product formula}\label{section:product formula}

The product formula is a key tool in free probability for
computing cumulants. It gives an explicit formula for
computing the cumulants of products of random variables,
e.g. $\kappa_3(a_1a_2, a_3, a_4 a_5 a_6)$ in terms of the
cumulants of $\{a_1, \dots, a_6\}$. See \cite[Lecture
  14]{ns} for a discussion and examples. In particular by
considering free compressions by matrices of finite rank
plus scalar the results of \cite{chs} and \cite{s} can be
recovered, see \cite[\S 5]{mt}. In \S \ref{section:small
  example} we give an example for the real free
infinitesimal cumulants of the limit distribution of the
square of a GOE random matrix.

In \cite{mst} the product formula was extended to second
order cumulants and very recently to third order cumulants
\cite{ams}. Unfortunately we cannot obtain Theorem
\ref{thm:product rule} from these results because of our
symmetry condition involving $\delta$. The remainder of the
paper will be devoted to proving Theorem \ref{thm:product
  rule} below.

Throughout we shall suppose $m_1 , \dots, m_r \geq 1$, $m =
m_1 + \cdots + m_r$, and
\[
M = \{m_1, m_1 + m_2, \dots, m_1 + \cdots + m_r\}.
\] 
$\gamma_m = (1, 2, \dots, m) \in S_m$. $K(\pi) = \pi^{-1}
\gamma_m$. We say a permutation $\sigma$ \textit{separates}
the point of $M$ if each point of $M$ is in a different
cycle of $\sigma$. Let $\sigma|_M$ be the permutation of $M$
given by the first return map\footnote{For $a \in M$ we set
$\sigma|_M(a) = \sigma^k(a)$ where $k \geq 1$ is the
smallest integer such that $\sigma^k(a) \in M$.} under
$\sigma$. Then $\sigma$ separates the points of $M$ if and
only if $\sigma|_M = \id_M$. Let
\begin{multline*}
\gamma_{\vm}
=
(1, \dots, m_1)(m_1 + 1, \dots, m_1 + m_2)
\cdots \\
(m_1 + \cdots + m_{r-1} + 1, \dots,
m_1 + \cdots + m_r).
\end{multline*}
When necessary we shall also consider $\gamma_\vm$ to be the
partition whose blocks are the cycles of $\gamma_\vm$.

For $1 \leq k \leq m$ let $I_k = \{ m_1 + \cdots + m_{k-1}
+1, \dots, m_1 + \cdots + m_k\}$. Given $V \subseteq [r]$ we
let $V_\vm = \cup_{k\in V} I_k \subseteq [m]$. Given $\pi
\in \cP(r)$ with $\pi = \{V_1, \dots, V_l\}$ we let $\pi_\vm
\in \cP(m)$ be the partition with blocks $\{ V_{1, \vm},
\dots, V_{l,\vm}\}$. If $\pi \in \NC(r)$ then $\pi_\vm \in
\NC(m)$, and conversely. We also view $\pi_\vm$ as the
permutation with cycles being the blocks of $\pi_\vm$ and
the elements in increasing order. See also
\cite[Def.~9]{mst}.

We shall let $A_1 = a_1 \cdots a_{m_1}$, $A_2 = a_{m_1 + 1}
\cdots\ab a_{m_1 + m_2}$, \dots, and $A_r = a_{m_1 + \cdots
  + m_{r-1} + 1} \cdots a_{m_1 + \cdots + m_{r}}$. Our goal
is to compute $\kappa_r'(A_1, \dots, A_r)$ in terms of the
cumulants and infinitesimal cumulants of $\{ a_1, \dots,
a_m\}$ as stated in the next theorem. This should be
compared to Equation (\ref{eq:complex product formula})
which gives the formula in the complex case. In the theorem
below, $K^\delta(\sigma)$ is the symmetric Kreweras
complement defined in Definition \ref{def:symmetric kreweras
  complement}. The equivalence between $K(\pi)$ separating
the points of $M$ and $\pi \vee \gamma_\vm = 1_m$ is proved
in Theorem 15 of \cite{mst}; see Lemma \ref{lemma:separation
  lemma} and the discussion in \cite{mst} preceding theorem
15.

\begin{theorem}\label{thm:product rule}
\begin{equation}\label{eq:the product rule}
\kappa_r'(A_1, \dots, A_r)
= \kern-1em
\sum_{\pi \in \NC(m)} \partial \kappa_\pi(a_1, \dots, a_m)
+ \kern-1.2em
\sum_{\sigma \in S_\NC^\delta(m, -m)} \kern-1.2em \kappa_{\sigma/2}
(a_1, \dots, a_m),
\end{equation}
where the first sum is over all $\pi$ such that $\pi \vee
\gamma_\vm = 1_m$ $($equivalently that $K(\pi)$ separates
the points of $M)$ and the second term is over all $\sigma$
such that $K^\delta(\sigma)$ separates the points of $\pm
M$.
\end{theorem}

\section{Small Examples with the Square of a Semi-Circle}\label{section:small example}

To illustrate the notation let us examine the infinitesimal
free cumulants of the square of a semi-circular operator,
where the infinitesimal law is that of the GOE \cite{m}. The
results of this section are not needed in the rest of the
paper, but the example will make it easier to follow the
notation and logic of the proof.

Suppose $r = m_1 = m_2 = 2$ and $a_1 = a_2 = a_3 = a_4 = s$
where $s$ is a standard semi-circular operator with mean $0$
and variance $1$ and infinitesimal law\footnote{We know the
infinitesimal moments from \cite{m} and if we use
(\ref{eq:infinitesimal moment cumulant}) we get the
vanishing of the infinitesimal free cumulants of $s$.}
$\kappa_r' = 0$ for $r = 1, 2, \dots$. Let $x = s^2$. Recall
that for the GOE we have $\tau(s^2) = 1$, $\tau(s^4) = 2$,
$\tau'(s^2) = 1$, and $\tau'(s^4) = 5$, see \cite[Lemma
  23]{m}. By (\ref{eq:infinitesimal moment cumulant})
\[
\tau'(x^2) = \kappa_2'(x, x) + 2 \kappa_1(x) \kappa_1'(x) + \kappa_2(x, x^t).
\]
Now $x = s^2$ is a free Poisson operator, so $\kappa_2(x, x) = 1$. Thus
\[
\kappa_2'(x, x) = \tau'(s^4) - 2 \tau(s^2) \tau'(s^2) - \kappa_2(x, x)
= 2.
\]
If one compares this with Figure \ref{fig:four-four-annulus}
below, one sees that of the $5$ elements of $\NC_2^\delta(4,
-4)$, only the two pairings
\[
\sigma_1 = (1, -4)(-1, 4)(2, 3)(-2, -3) \qquad
\sigma_2 = (1, 4)(-1, -4)(2, -3)(-2, 3)
\]
with Kreweras complements
\[
K^\delta(\sigma_1) = (1, 3, -4)(2)(-1, 4, -3)(-2)
\]
\[
K^\delta(\sigma_2) = (1, -2, 3)(4)(-1, -3, 2)(-4)
\]
have the property that $K^\delta(\sigma)$ separates the
points $2$ and $4$. Thus these two ways of computing
$\kappa_2'(x, x)$, using either (\ref{eq:infinitesimal
  moment cumulant}) or (\ref{eq:the product rule}), agree.

One can do this again to find $\kappa_3'(x, x, x)$. By
(\ref{eq:infinitesimal moment cumulant}) we have
\begin{align*}\lefteqn{
\tau'(x^3) }\\
& = 
\kappa_3'(x, x, x) + 3 \kappa_1'(x) \kappa_2(x, x) + 3 \kappa_1(x) \kappa_2'(x, x) \\
&\mbox{} + 
3 \kappa_1'(x) \kappa_1(x)^2 +
\sum_{\sigma \in S_\NC^\delta(3, 3)} \kappa_{\sigma/2}(x, x, x). 
\end{align*}
To evaluate the \textsc{lhs} we have $\tau'(x^3) =
\tau'(s^6) = 22$, by \cite[Lemma 23]{m}. As for the
\textsc{rhs} we have that $\kappa_{\sigma/2}(x, x, x) = 1$
for all $\sigma$, as $x$ is a free Poisson. By
\cite[Prop. 19]{mvb}, $|S_\NC^\delta(3, -3)| = 6$. Thus the
\textsc{rhs} of the equation above is
\[
\kappa_3'(x, x, x) + 3 + 6 +  3 + 6 = 18 + \kappa_3'(x, x, x).
\]
Solving for $\kappa_3'(x, x, x)$, we get $\kappa_3'(x, x, x) = 4$. 
\begin{figure}[t]
\noindent
\includegraphics{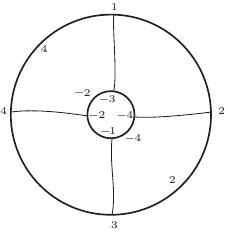}\hfill \includegraphics{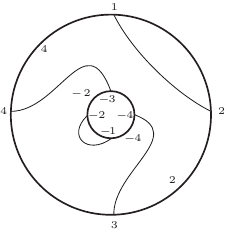}\hfill \includegraphics{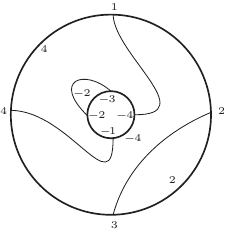}

\noindent
\hfill\includegraphics{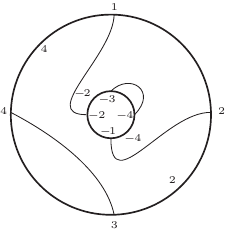}\hfill \includegraphics{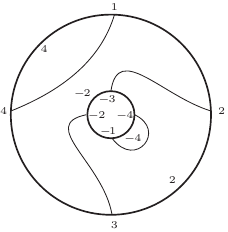}\hfill\hbox{}
\caption{\label{fig:four-four-annulus}\small The 5
  non-crossing pairings of a $(4, -4)$-annulus mentioned in
  \S \ref{section:small example}. We have marked the
  positions of the points $\{2, -2, 4, -4\}$ in the Kreweras
  complement. Only the third and the fifth have the property
  that $K^\delta(\pi)$ separates the points of $\{2, -2, 4,
  -4\}$. These are the two that contribute to $\kappa_2'(x,
  x)$. \hbox{}\hfill$\diamond$}
\end{figure}

Now turning to equation (\ref{eq:the product rule}) we have to evaluate
\[
\mathop{\sum_{\sigma \in S_\NC^\delta(6, -6)}}_%
{K^\delta(\sigma) \sep \{\pm 2, \pm 4, \pm 6\}}
\kappa_{\sigma/2}(s, s, s, s, s, s, s).
\]
Since $s$ is semi-circular, the only $\sigma$'s that appear
are pairings and their contribution is the same for all
pairings $\sigma$. In \cite[Lemma 23]{m} we found that the
number of pairings is 22. If one examines these 22 annular
pairings\footnote{The details are not provided here, but it
is instructive to examine this case.} one finds that only
the following 4 satisfy the condition that
$K^\delta(\sigma)$ separates the points of $\{ \pm 2, \pm 4,
\pm 6\}$
\[
(1, -4)(2, -5)(3, -6)(-1, 4)(-2, 5)(-3, 6),
\]
\[
(1, -6)(2, 3)(4, 5)(-1, 6)(-2, -3)(-4, -5),
\]
\[
(1, 6)(2, 3)(4, -5)(-1, -6)(-2, -3)(-4, 5),
\]
\[
(1, 6)(2, -3)(4, 5)(-1, -6)(-2, 3)(-4, -5).
\]
Thus using equation (\ref{eq:product rule}) we also get the
conclusion that $\kappa_3'(x, x, x) = 4$. One may (naively)
conjecture that $\kappa'_n(x, \dots, x) = 2^{n-1}$, based on
three examples. However, this conjecture turns out to be
correct; indeed it is shown in \cite{mvb2} that the real
infinitesimal $r$-transform, $r(z) = (1 - 2 z)^{-1}$.

We can also compare the infinitesimal law of $x$ to that of
$y$ which has the limit distribution of a real Wishart
matrix with $c = 1$ and $c' = 0$ (see \cite[Corollary
  18]{mvb}). Then the base distribution of $x$ and $y$ are
both Marchenko-Pastur with parameter $1$ (see
\cite[Def.~2.11]{ms}). However we have $\kappa'_2(x, x) = 2$
and $\kappa_2'(y, y) = 1$, see \cite[Cor.~18]{mvb}. So the
infinitesimal laws of $x$ and $y$ are different even though
the base distributions are the same.

\section{The proof of Theorem \ref{thm:product rule}: outline}
\label{sec:product rule outline}

In this section we will present the outline of the proof of
Theorem \ref{thm:product rule} and the show that the proof
can be reduced to demonstrating Equations (\ref{eq:product
  first step}) and (\ref{eq:product second step}). The proof
of the claim about (\ref{eq:product first step}) will be
given in Proposition \ref{prop:first} and the proof of the
claim about (\ref{eq:product second step}) will be given in
Proposition \ref{prop:third bijection}.

\begin{lemma}[{\cite[Lemma 14]{mst}}]\label{lemma:separation lemma}
Suppose $\rho \in \NC(m)$. Then $\rho \vee \gamma_\vm =
\pi_\vm$ if and only if $\rho^{-1}\pi_\vm$ separates the
points of $M$.
\end{lemma}

\begin{lemma}[{\cite[Prop. 11.12]{ns}}]\label{lemma:zeroth order product}
For $\pi \in \NC(r)$ we have 
\begin{equation}\label{eq:product rule}
\kappa_\pi(A_1, \dots, A_r)
=
\sum_{\rho \in \NC(m)} \kappa_\rho(a_1, \dots, a_m)
\end{equation}
where the sum is over all $\rho$ such that $\rho \vee \gamma_\vm = \pi_\vm$. 
\end{lemma}

\noindent
To set up the proof of Theorem \ref{thm:product rule} we
shall set the following notation. Let

\begin{itemize}\itemsep1em
\item
$N_1 = \{ \pi \in \NC(m) \mid K(\pi) \sep M \}$ and  

\item
$N_2 = \{ \pi \in \NC(m) \mid K(\pi)|_M \not= \textit{id}_M\}$.
\end{itemize} 

\medskip\noindent
Then $\NC(m) = N_1 \cup N_2$. 

For $\sigma \in S_\NC^\delta(m, -m)$ let $K^\delta(\sigma) = \delta \gamma_m^{-1} \delta \sigma^{-1} \gamma_m$.  Let
\begin{itemize}\itemsep1em

\item
$S_1 = \{ \sigma \in S_\NC^\delta(m, -m) \mid K^\delta(\sigma) \sep \pm M\}$, 

\item
$S_2 = \{ \sigma \in S_\NC^\delta(m, -m) \mid K^\delta(\sigma)|_{\pm M} \not= \id_{\pm M}$ but $K^\delta(\sigma)|_{\pm M}$ has no through cycles$\}$

\item
$S_3 = \{ \sigma \in S_\NC^\delta(m, -m) \mid K^\delta(\sigma)|_{\pm M}$ has through cycles$\}$.

\end{itemize}\medskip

\noindent
Then $S_1 \cup S_2 \cup S_3 = S_\NC^\delta(m, -m)$.

First we prove (\ref{eq:the product rule}) when $r =
1$. Then, using induction on $r$, we may, for each $\pi \in
\NC(r) \setminus \{ 1_r \}$ expand $\partial \kappa_\pi(A_1,
\dots, A_r)$ using (\ref{eq:the product rule}) and use this
to prove the theorem for $\pi = 1_r$. The main idea is to
expand
\[
\tau'(A_1 \cdots A_m)
=
\tau'(a_1 \cdots a_m)
\]
using the moment-cumulant formula (\ref{eq:infinitesimal
  moment cumulant}), on p.~ \pageref{eq:infinitesimal moment
  cumulant}, in two ways and compare the results.

First we expand $\tau'(A_1 \cdots A_m)$ using the moment-cumulant formula:
\begin{equation}\label{eq:infinitesimal moment cumulant products}
\tau'(A_1 \cdots A_m)
=
\sum_{\pi \in \NC(r)} \partial \kappa_\pi(A_1, \dots, A_r)
\end{equation}
\[
\mbox{} + 
\sum_{\sigma \in S_\NC^\delta(r, -r)} \kappa_{\sigma/2}(A_1, \dots, A_r).
\]
We write the first term on the right-hand side of
(\ref{eq:infinitesimal moment cumulant products}) as
\begin{equation}\label{eq:induction term one}
\kappa'_r(A_1, \dots, A_r)
+
\mathop{\sum_{\pi \in\NC(r)}}_{\pi \not= 1_r}
\partial\kappa_\pi(A_1, \dots, A_r).
\end{equation}
Now applying (\ref{eq:infinitesimal moment cumulant}) to
(\ref{eq:product rule}) we will show in Proposition
\ref{prop:first}, that the second term on the right-hand
side of (\ref{eq:induction term one}) is
\begin{equation}\label{eq:product first step}
\sum_{\rho \in N_2} \partial \kappa_\rho(a_1, \dots, a_m)
+
\sum_{\tau \in S_2} \kappa_{\tau/2}(a_1, \dots, a_m).
\end{equation}
Then we will show in Proposition \ref{prop:third bijection},
that the second term on the right-hand side of
(\ref{eq:infinitesimal moment cumulant products}) equals
\begin{equation}\label{eq:product second step}
\sum_{\tau \in S_3}
\kappa_{\tau/2}(a_1, \dots, a_m).
\end{equation}
When we put these two results together we obtain 
\begin{align}\label{eq:combined formulas}
& 
\kappa'_r(A_1, \dots, A_r) +
\sum_{\rho \in N_2} \partial \kappa_\rho(a_1, \dots, a_m) 
+
\sum_{\tau \in S_2}  \kappa_{\tau/2}(a_1, \dots, a_m) \notag\\
& \mbox{} +
\sum_{\tau \in S_3} \kappa_{\tau/2}(a_1, \dots, a_m) 
= \tau'(A_1 \cdots A_r) = \tau'(a_1 \cdots a_m)\notag \\
& \mbox{} =
\sum_{\rho \in \NC(r)} \partial \kappa_\rho(a_1, \dots, a_m)
 + 
\sum_{\tau \in S_\NC^\delta(r, -r)} \kern-1em
\kappa_{\tau/2}(A_1, \dots, A_r).
\end{align}
Thus
\begin{align*}\lefteqn{
\kappa'_r(A_1, \dots, A_r)} \\
&  =
\sum_{\rho \in \NC(m) \setminus N_2} \partial \kappa_\rho(a_1, \dots, a_m)
\kern1em + \kern-2em
\sum_{\tau \in S_\NC^\delta(m, -m) \setminus (S_2 \cup S_3)} \kern-2.5em
 \kappa_{\tau/2}(a_1, \dots, a_m)  \\
& =
\sum_{\rho \in N_1} \partial \kappa_\rho(a_1, \dots, a_m)
+
\sum_{\tau \in S_1}  \kappa_{\tau/2}(a_1, \dots, a_m) 
\end{align*}
which is exactly the claim of Theorem \ref{thm:product rule}.

\section{First Step: Proposition \ref{prop:first}}
\label{sec:product formula first step}

Since in Proposition \ref{prop:first} below, we exclude the
case $\pi = 1_r$ we may use induction on $r$ and thus we may
use (\ref{eq:the product rule}):
\[
\kappa_s'(A_1, \dots, A_s)
= \kern-1em
\sum_{\pi \in \NC(m)} \partial \kappa_\pi(a_1, \dots, a_m)
+ \kern-1.2em
\sum_{\sigma \in S_\NC^\delta(m, -m)} \kern-1.2em \kappa_{\sigma/2}
(a_1, \dots, a_m),
\]
for $s < r$, provided that we start the induction by proving
(\ref{eq:the product rule}) when $s = 1$. However in this
case (\ref{eq:the product rule}) reduces to the moment
cumulant formula (\ref{eq:infinitesimal moment cumulant}),
but the $s = 1$ step is just the definition of real
infinitesimal cumulants. This clears the way to start our
induction. Let us recall the definition of $N_2$ and $S_2$:

\[
N_2 = \{ \pi \in \NC(m) \mid K(\pi)|_M \not= \textit{id}_M\} \mbox{\ and}
\]
\begin{multline*}
S_2 = \{ \sigma \in S_\NC^\delta(m, -m) \mid
K^\delta(\sigma)|_{\pm M} \not= \id_{\pm M} \mbox{\ but\ }
\\ K^\delta(\sigma)|_{\pm M} \mbox{\ has no through
  cycles\ }\}.
\end{multline*}

\begin{proposition}\label{prop:first}
\[
\mathop{\sum_{\pi \in\NC(r)}}_{\pi \not= 1_r}
\partial\kappa_\pi(A_1, \dots, A_r)
=
\sum_{\rho \in N_2} \partial \kappa_\rho(a_1, \dots, a_m)
+
\sum_{\tau \in S_2} \kappa_{\tau/2}(a_1, \dots, a_m)
\]
\end{proposition}

The idea is that for each $\rho \in N_2$ we can associate a
unique $\pi(\rho) \in \NC(r) \setminus \{1_r\}$ and for each
$\tau \in S_2$ we can associate a unique $\pi(\tau) \in
\NC(r) \setminus \{1_r\}$ such that
\begin{equation}\label{eq:two part strategy}
\partial\kappa_\pi(A_1, \dots, A_r)
=
\mathop{\sum_{\rho \in N_2}}_{\pi(\rho) = \pi}
 \partial \kappa_\rho(a_1, \dots, a_m)
+
\mathop{\sum_{\tau \in S_2}}
_{\pi(\tau) = \pi}
\kappa_{\tau/2}(a_1, \dots, a_m).
\end{equation}

For $\pi \in \NC(r)$, recall that $K(\pi) =
\pi^{-1}\gamma_r$, and for $\rho \in \NC(m)$, $K(\rho) =
\rho^{-1}\gamma_m$, and $\psi: [r] \rightarrow M$ is given
by $\psi(k) = m_1 + \cdots + m_k$.

\begin{lemma}\label{lemma:first restriction}
For $\rho \leq \pi_\vm$ with $\rho \vee \gamma_\vm = \pi_\vm$, we have ${K(\rho)}|_\sM \psi = \psi K(\pi)$. 
\end{lemma}

\begin{proof}
Since $K(\pi_\vm)|_{\sM} = \psi K(\pi) \psi^{-1}$ we just
have to show that $K(\rho)|_\sM = K(\pi_\vm)|_\sM$. However
we have $K(\rho) = \rho^{-1}\pi_\vm \, K(\pi_\vm)$ and
$\pi_\vm^{-1}\gamma_m$ acts trivially on $M^c$. We have that
$K(\rho)|_\sM = (\rho^{-1}\pi_\vm)|_\sM \cdot
K(\pi_\vm)|_\sM$, by \cite[Lemma 6]{mst} . By assumption
$\rho^{-1}\pi_\vm|_\sM = \id_\sM$, so we have $K(\rho)|_\sM
= K(\pi_\vm)|_\sM$.
\end{proof}

\begin{figure}
\setbox1=\hbox{\includegraphics{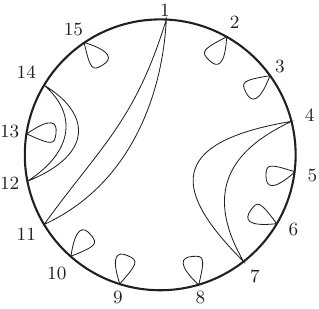}}
\setbox2=\hbox{\includegraphics{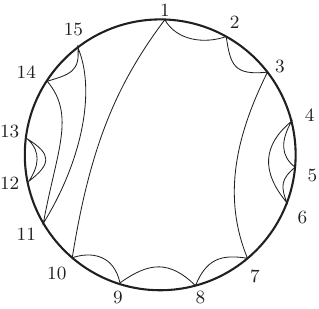}}
\setbox3=\hbox{\includegraphics{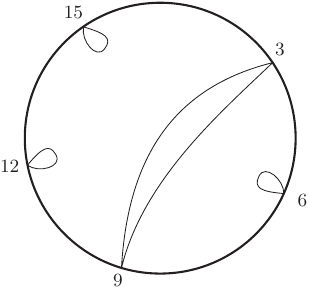}}
\setbox4=\hbox{\includegraphics{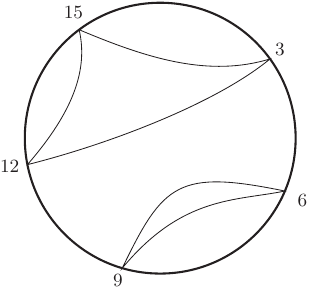}}
\setbox5=\hbox{\includegraphics{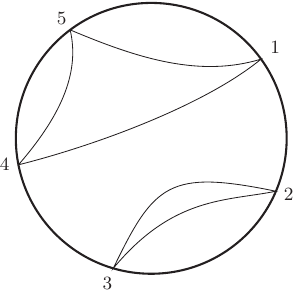}}
\setbox6=\hbox{\includegraphics{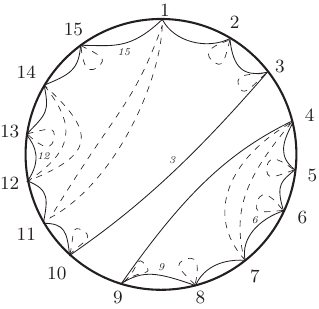}}

\leavevmode
\vbox{\raggedright\hsize\wd1\box1
{\tiny 
\begin{center}$\rho = (1,11)(2)(3)(4, 7)(5)(6)(8)(9)\ab(10)(12, 14)(13)(15)$\end{center}}} \hfill
\vbox{\raggedright\hsize\wd2\box2
{\tiny 
\begin{center}$K(\rho) = (1,2,3,7,8,9,10)(4, 5,6)\ab(11,14,15)(12, 13)$\end{center}}}

\leavevmode
\vbox{\raggedright\hsize\wd3\box3
{\tiny 
\begin{center}$K(\rho)|_M = (3,9)(6)(12)(15)$\end{center}}} 
\hfill
\vbox{\raggedright\hsize\wd4\box4
{\tiny 
\begin{center}$K^{-1}[K(\rho)|_M] = (3,12,15)(6,9)$\end{center}}}

\leavevmode
\vbox{\raggedright\hsize\wd5\box5
{\tiny 
\begin{center}$\pi = \psi^{-1}(K(\rho)|_M)\psi$\end{center}}}
\hfill
\vbox{\raggedright\hsize\wd6\box6
{\tiny 
\begin{center}$\pi_\vm = (1,2,3,10,11,12,13,14,15)\ab(4,5,6,7,8,9)$\end{center}}}

\caption{\label{fig:first term}\small The six steps in Lemma
  \ref{lemma:first bijection}. We are given $\rho$ in the
  upper left; its Kreweras complement in the upper right;
  the restriction to $M$ in the centre left; the inverse
  Kreweras complement of the restriction of the Kreweras
  complement in the centre right; $\pi$ in the lower left,
  and $\pi_\vm$ in the lower right, with $\rho$ in dashed
  lines. Note that $\rho^{-1}\pi_\vm$ separates the points
  of $M$ (\textit{marked} {\tiny\textit{3, 6, 9, 12, 15})}.}
\end{figure}

\begin{lemma}\label{lemma:first bijection}
Let 
\begin{multline}
\tilde{N}_2 =
\{ \rho \in \NC(m) \mid \exists\,
1_r \not =\pi \in \NC(r) \mbox{\ such that\ } \\
\rho \leq \pi_\vm 
\mbox{\ and\ }
\rho^{-1}\pi_\vm \sep M\}.
\end{multline}
Then $N_2 = \tilde{N}_2$.
\end{lemma}

Before reading the proof one should inspect the example in
Figure \ref{fig:first term}.

\begin{proof}
Let $\rho \in \tilde{N}_2$. By Lemma \ref{lemma:first
  restriction}, $K(\rho)|_\sM \psi = \psi K(\pi)$. By
assumption $\pi \not = 1_r$ so $K(\pi) \not= 0_r$. Hence as
a permutation $K(\pi)$ is not the identity, thus
$K(\rho)|_\sM \not= \id_\sM$.

Conversely, suppose $\rho \in N_2$. We obtain $\pi$ as
follows.  Start with $\rho$, take its Kreweras complement,
restrict to $M$, then take the inverse Kreweras
complement. This produces a permutation on $M$. Finally use
$\psi$ to turn this into a permutation on $[r]$.

Let $\gamma_\sM = (m_1, m_1 + m_2, \dots, m_1 + \cdots +
m_r) \in S(M)$ be the permutation with one cycle and let
$\tilde\pi$ be the inverse Kreweras complement of
$K(\rho)|_\sM$, namely $\tilde\pi = \gamma_\sM
(K(\rho)|_\sM)^{-1}$, and finally let $\pi \in S_r$ be given
by $\tilde\pi\psi = \psi\pi$. In \cite[Notation 4]{mst} it
was shown that $K(\rho)|_\sM \in \NC(M)$, thus $\pi \in
\NC(r)$, as $\psi$ is order preserving. As $K(\pi_\vm)$ acts
trivially on $M^c$, we have $K(\pi_\vm)|_\sM = \psi K(\pi)
\psi^{-1}$ and so $K(\rho)|_\sM = K(\pi_\vm)|_\sM$.

Note that 
\[
K(\rho)|_\sM = (\rho^{-1} \gamma_m )|_\sM = 
(\rho^{-1} \pi_\vm \pi_\vm^{-1} \gamma_m)|_\sM
=
(\rho^{-1} \pi_\vm K(\pi_\vm)|_\sM
\]
Again as $K(\pi_\vm)$ acts trivially on $M^c$, we have by
\cite[Lemma 6]{mst} that
\[
(\rho^{-1} \pi_\vm K(\pi_\vm)|_\sM
=
(\rho^{-1} \pi_\vm)|_\sM (K(\pi_\vm)|_\sM.
\]
Thus $K(\rho)|_\sM = (\rho^{-1}\pi_\vm)|_\sM K(\pi_\vm
)|_\sM = (\rho^{-1}\pi_\vm)|_\sM K(\rho )|_\sM$. By
cancelling $K(\rho)|_\sM$ we have $\rho^{-1}\pi_\vm|_\sM =
\id_\sM$.

Finally note that the actions of $K(\rho)$ and $K(\pi_\vm )$
on $M$ are the same and $K(\pi_\vm)$ acts trivially on
$M^c$, so $K(\pi_\vm) \leq K(\rho)$. Hence $\rho \leq
\pi_\vm$. This shows that $\rho \in N_2$. In addition, our
formula for $\pi$ shows that $\pi$ is unique.
\end{proof}

Now we can explain our proof strategy for Equation
(\ref{eq:two part strategy}). First fix $\pi \in \NC(r)$
with $\pi \not= 1_r$. For $V \in \pi$, $V_\vm \subseteq [m]$
is defined in \S\ref{section:product formula}.
\[
\partial\kappa_\pi(A_1, \dots, A_r) = \sum_{V \in \pi} 
\Big\{
\kappa_{|V|}'(A_1, \dots, A_r \mid V)
\prod_{W \not=V} \kappa_{|W|}(A_1, \dots, A_r | W) \Big\} 
\]
\[
 \stackrel{(*)}{=}
\sum_{V \in \pi} \Big\{
\kappa_{|V|}'(A_1, \dots, A_r \mid V)
\prod_{W \not=V}
\mathop{\sum_{\rho \in \NC(W_{\vv m})}}_{K(\rho) \sep M \cap W_{\vv m}}
\kern-1em
\kappa_\rho(a_1, \dots, a_m \mid W_{\vv m}) 
\]
\begin{multline}\label{eq:parts one and two}
\stackrel{(**)}{=}
\sum_{V \in \pi} \Bigg\{
\prod_{W \not=V} \kern-1em
\mathop{\sum_{\rho \in \NC(W_{\vv m})}}_{K(\rho) \sep M \cap W_{\vv m}}
\kern-1em
\kappa_\rho(a_1, \dots, a_m \mid W_{\vv m}) \\
\mbox{} \times
\bigg\{ \kern-1em
\mathop{\sum_{\rho \in \NC(V_{\vv m})}}_{K(\rho) \sep M \cap V_{\vv m}}
\partial \kappa_\rho(a_1, \dots, a_m \mid V_{\vv m})   \\
\mbox{} + \kern-2em
\mathop{\sum_{\tau \in S_\NC^\delta(V_{\vv m}, -V_{\vv m})}}
_{K^\delta(\tau) \sep \pm (M \cap V_{\vv m})}
\kappa_{\tau/ 2}(a_1, \dots, a_m | V_{\vv m}) \bigg\}\Bigg\}
\end{multline}
where $(*)$ holds by the product formula in the disc and
$(**)$ holds by the induction hypothesis. Our strategy is
now to split this last line into two parts; then prove the
first part equals $\mathop{\sum\limits_{\rho \in
    N_2}}\limits_{\pi(\rho) = \pi} \partial \kappa_\rho(a_1,
\dots, a_m)$ and the second equals
$\mathop{\sum\limits_{\tau \in S_2}}\limits _{\pi(\tau) =
  \pi} \kappa_{\tau/2}(a_1, \dots, a_m)$.

\subsection{\normalsize The First Part of Equation (\ref{eq:parts one and two})}

We break the last expression into two parts and consider the first part. 
\[
\sum_{V \in \pi}\Bigg\{
\prod_{W \not=V} \kern-1.0em
\mathop{\sum_{\rho \in \NC(W_{\vv m})}}_{K(\rho) \sep M \cap W_{\vv m}}
\kern-2em
\kappa_\rho(a_1, \dots, a_m \mid W_{\vv m})
\Bigg\}  \kern-1.5em
\mathop{\sum_{\rho \in \NC(V_{\vv m})}}_{K(\rho) \sep M \cap V_{\vv m}} \kern-1.5em
\partial \kappa_\rho(a_1, \dots, a_m \mid V_{\vv m})
\]
\[
\mbox{} \stackrel{(***)}{=} \kern-1.5em
\mathop{\sum_{\rho \in \NC(m),\, \rho \leq \pi_\vm}}_{\rho^{-1}\pi_\vm \sep M} \kern-1.5em
\partial \kappa_\rho(a_1, \dots, a_m)
\]
where in $(*\!*\!*)$ we combined all the $\rho$'s into a
single $\rho$. To justify this notice that when we have for
each $W \in \pi$, a $\rho_\sW \in \NC(W_{\vv m})$ such that
$\rho_\sW \vee \gamma_{\sW_{\vv m}} = 1_{\sM \cap \sW_{\vv
    m}}$, we get that $\prod_{W \in \pi} \rho_\sW^{-1}
\gamma_{\sW_{\vv m}}$ separates the points of $M$. Then
$\rho^{-1} \pi_{\vv m} = \prod_{\sW \in \pi} \rho_\sW^{-1}
\gamma_{\sW_{\vv m}}$ separates the points of $M$.  Thus by
Lemma \ref{lemma:first bijection}, this first part equals
\[
\sum_{\rho \in N_2} \partial \kappa_\rho(a_1, \dots, a_m).
\]

\subsection{\normalsize The Second Part of Equation (\ref{eq:parts one and two})}

Now let us consider the second part of the expression
(\ref{eq:parts one and two}) above (\textit{recall that $\pi
  \not = 1_r$ has been fixed}):
\begin{equation}\label{eq:second term}
\sum_{V \in \pi}
\prod_{W \not=V} \kern-1.0em
\mathop{\sum_{\rho \in \NC(W_{\vv m})}}_{K(\rho) \sep M \cap W_{\vv m}}
\kern-2em
\kappa_\rho(a_1, \dots, a_m \mid W_{\vv m}) \kern-2em
\mathop{\sum_{\tau \in S_\NC^\delta(V_{\vv m}, -V_{\vv m})}}
_{K^\delta(\tau) \sep \pm(M \cap V_\vm)} 
\kern-2em
\kappa_{\tau}(a_1, \dots, a_m | V_{\vv m}). 
\end{equation}
We will show that this equals $\mathop{\sum\limits_{\tau \in S_2}}\limits
_{\pi(\tau) = \pi} \kappa_{\tau/2}(a_1, \dots, a_m)$. 

We need to combine all the $\rho$'s for each $W$ and the
$\tau$ for $V$ to obtain a single $\tau$ as follows. If
$\rho_\sW \in \NC(W_\vm)$, then we double this to obtain
$\delta\rho_\sW^{-1}\delta \rho_\sW \in \NC(W_\vm) \times
\NC(-W_\vm)$. For the $\tau_\sV \in S_\NC^\delta(V_\vm,
-V_\vm)$ we let it stand unchanged. The $\tau$ we want is
then the product of all these parts
\begin{equation}\label{eq:tau construction}
\tau = \tau_\sV \prod_{W \not= V} \delta \rho_\sW^{-1} \delta \rho_\sW.
\end{equation}
We want to show that the $\tau$'s obtained this way are such
that $K^\delta(\tau)$ does not separate the points of $\pm
M$ but does separate $M$ from $-M$. For this we shall need
some additional notation. As before we start with $\pi \in
\NC(r)$ and construct $\pi_\vm \in \NC(m)$. Then $\delta
\pi_\vm^{-1} \delta \pi_\vm \in \NC(m) \times \NC(-m)$. For
each $V \in \pi$ we get a block $V_\vm$ of $\pi_\vm$. If we
fix $V \in \pi$, we let $\cU_\sV$ be the partition of $[ \pm
  m]$ obtained from $\delta \pi_\vm^{-1} \delta \pi_\vm$ by
joining $V_\vm$ with $\delta(V_\vm)$.  Then each cycle of
$\pi_\vm$ is contained in a block of $\cU_\sV$, and thus
$(\cU_\sV, \pi_\vm)$ is a partitioned permutation in the
sense of \cite[\S1]{mst}, (\textit{see p.}~4754). The
crucial part is that now for the $\tau$ constructed above
(\ref{eq:tau construction}), $(0_\tau, \tau)$ is
non-crossing with respect to $(\cU_\sV, \delta\pi_\vm^{-1}
\delta \pi_\vm)$, or equivalently $(0_\tau, \tau) \leq
(\cU_\sV, \delta\pi_\vm^{-1} \delta \pi_\vm)$ in the sense
of \cite[Cor. 38 $(iii)$]{mst}.

Thus we may write the second term (\ref{eq:second term}) as
\[
\sum_{V \in \pi}
\prod_{W \not=V} \kern-1.0em
\mathop{\sum_{\rho \in \NC(W_{\vv m})}}_{K(\rho) \sep M \cap W_{\vv m}}
\kern-2em
\kappa_\rho(a_1, \dots, a_m \mid W_{\vv m}) \kern-2em
\mathop{\sum_{\tau \in S_\NC^\delta(V_{\vv m}, -V_{\vv m})}}
_{K^\delta(\tau) \sep \pm(M \cap V_\vm)} 
\kern-2em
\kappa_{\tau}(a_1, \dots, a_m | V_{\vv m})
\] 
\[
\mbox{}\stackrel{(*)}{=}
\sum_{V \in \pi} \kern-0.75em \mathop{\mathop{\sum_{\tau \in S_\NC^\delta(m, -m)}}_
{(0_\tau, \tau) \leq (\cU_\sV, \delta\pi_\vm^{-1} \delta \pi_\vm)}}_
{K^\delta_{\pi_{\vv m}}(\tau) \sep \pm M}
\kappa_{\tau/2}(a_1, \dots, a_m),
\]
where $K^\delta_{\pi_\vm}(\tau) := \delta \pi_\vm^{-1}
\delta \tau^{-1} \pi_\vm$ is the \textit{relative} Kreweras
complement of $\tau$ with respect to $\pi_\vm$. The
justification for this last equality $(*)$ is the same as
above: all of the separation conditions are local; so when
we put the cycles together to form $\tau$ we get that
$K^\delta_{\pi_\vm}(\tau)$ \textit{still} separates the
points of $M$. Since the non-crossing condition is
characterized by a metric property, \cite[Notation 4]{mst},
non-crossing on each piece, plus that the pieces don't cross
give us that $\tau \in S_\NC^\delta(m, -m)$. See Figure
\ref{fig:two} for a simple example.

\medskip\noindent
Let
\begin{multline*}
\tilde{S}_2 = 
\{ \tau \in S_\NC^\delta(m, -m) \mid \exists!\,
\pi \in \NC(r) \mbox{\ and\ } V \in \pi
\mbox{\ such that\ } \pi \not = 1_r,\\
(0_\tau, \tau) \leq (\cU_\sV, \delta\pi_\vm^{-1} \delta \pi_\vm),  
 \mbox{\ and\ } 
K_{\pi_\vm}^\delta(\tau) \sep \pm M\}.
\end{multline*}

First we have the annular version of Lemma \ref{lemma:first
  restriction}.

\begin{lemma}\label{lemma:annular restrictions}
Suppose that $\tau \in S_\NC^\delta(m, -m)$, $\pi \in
\NC(r)$, $V \in \pi$ with $(0_\tau, \tau) \leq (\cU_\sV,
\pi_\vm)$ and $K_{\pi_\vm}^\delta(\tau)|_{\pm \sM} =
\id_{\pm \sM}$. Then \[ K^\delta(\tau)|_{\pm \sM} = \ab
K^\delta(\pi_\vm) |_{\pm \sM}.\]
\end{lemma}

\begin{proof}
First note that $\pi_\vm^{-1} \gamma_m|_{[m] \setminus M} =
\id_{[m] \setminus M}$. So by \cite[Lemma 6]{mst},
\begin{multline*}
\pi_\vm^{-1} \gamma_m |_{\pm \sM} = K^\delta_{\pi_\vm}(\tau) |_{\pm \sM} \pi_\vm^{-1} \gamma_m |_{\pm \sM} \\
\mbox{}= \delta \pi_\vm^{-1} \delta \tau^{-1} \pi_\vm \cdot \pi_\vm^{-1} \gamma_m |_{\pm \sM} =
\delta \pi_\vm^{-1} \delta \tau^{-1} \gamma_m|_{\pm \sM}.
\end{multline*}
Likewise, $\delta \gamma_m^{-1}\pi_\vm \delta|_{[-m]
  \setminus -M} = \id_{[-m] \setminus -M}$. So by
\cite[Lemma 6]{mst},
\[
\delta \gamma_m^{-1} \pi_\vm \delta|_{\pm\sM} \delta \pi_\vm^{-1} \delta \tau^{-1} \gamma_m|_{\pm \sM}
= \delta \gamma_m^{-1} \delta \tau^{-1} \gamma_m |_{\pm \sM} = K^\delta(\tau)|_{\pm \sM}.
\]
Combining these two identities we have 
\[
K^\delta(\pi_\vm)|_{\pm \sM} = \delta \gamma_m^{-1} \delta \tau \gamma_m |_{\pm \sM} = K^\delta(\tau)|_{\pm \sM}.
\]
\end{proof}
Second, we have the annular version of Lemma
\ref{lemma:first bijection}.

\begin{lemma}\label{lemma:second bijection}
$S_2 = \tilde{S}_2$.
\end{lemma}

\begin{proof}
For convenience, let us recall the definition of $S_2$:
\begin{multline*}
S_2 = \{ \sigma \in S_\NC^\delta(m, -m) \mid
K^\delta(\sigma)|_{\pm M} \not= \id_{\pm M} \mbox{\ but\ }
\\ K^\delta(\sigma)|_{\pm M} \mbox{\ has no through
  cycles\ }\}.
\end{multline*}
Let $\tau \in \tilde S_2$, then by Lemma \ref{lemma:annular
  restrictions} we have $K^\delta(\tau)|_{\pm \sM} =
K^\delta(\pi_\vm)|_{\pm \sM}$. Since $\pi \not = 1_r$ we
cannot have that $K(\pi)$ is trivial. Thus
$K^\delta(\tau)|_{\pm \sM} \not = \id_{\pm\sM}$. In
addition, as $K^\delta(\pi_\vm)|_{\pm \sM}$ has no through
cycles, we have that $K^\delta(\tau)|_{\pm \sM}$ has no
through cycles. Thus $\tau \in S_2$.

Now let $\tau \in S_2$. Let us recall some notation from
Lemma \ref{lemma:first bijection}. By construction
$K^\delta(\tau)|_{\pm M}$ leaves $M$ invariant\footnote{Note
the similarity to \cite[Lemma 20]{mst}}. Let $\gamma_\sM =
(m_1, m_1 + m_2, \dots,\ab m_1 + \cdots + m_r) \in S(M)$ be
the permutation with one cycle and let $\tilde\pi$ be the
inverse Kreweras complement of $K^\delta(\tau)|_\sM$, namely
$\tilde\pi = \gamma_\sM (K^\delta(\tau)|_\sM)^{-1}$, and
finally let $\pi \in S_r$ be given by $\tilde\pi\psi =
\psi\pi$. In \cite[Notation 4]{mst} it was shown that
$K(\tilde\pi)|_\sM \in \NC(M)$, thus $\pi \in \NC(r)$, as
$\psi$ is order preserving. As $K(\pi_\vm)$ acts trivially
on $M^c$, we have $K(\pi_\vm)|_\sM = \psi K(\pi) \psi^{-1}$
and so $K^\delta(\tau)|_\sM = K(\pi_\vm)|_\sM$. By the
symmetry of $\tau$ we also have $K^\delta(\tau)|_{-\sM} =
K^\delta(\delta\pi_\vm^{-1}\delta)|_{-\sM}$. Combining this
with the result on $M$ we have $K^\delta(\tau)|_{\pm\sM} =
K^\delta(\delta\pi_\vm^{-1}\delta \pi_\vm)|_{\pm
  \sM}$. Since $K^\delta(\delta\pi_\vm^{-1}\delta \pi_\vm) =
\delta \gamma_m^{-1} \pi_\vm \delta \pi_\vm^{-1} \gamma_m$
acts trivially on $(\pm M)^c$, we have $K^\delta(\delta
\pi_\vm^{-1} \delta \pi_\vm) \leq K^\delta(\tau)$ in the
sense of \cite[Notation 4]{mst}. This then implies that
\begin{equation*}\label{eq:relative complement}
|K^\delta(\delta \pi_\vm^{-1}\delta\pi_\vm)| + |K^\delta_{\pi_\vm}(\tau)| = |K^\delta(\tau)|.
\end{equation*}
The fact that $\tau \in S_\NC^\delta(m, -m)$ means
\begin{equation*}\label{eq:non-crossing tau}
|\tau| + |K^\delta(\tau)| = |\delta \gamma_m^{-1}\delta \gamma_m| + 2.
\end{equation*}
Finally as $\pi_\vm \in \NC(m)$ we have
\begin{equation*}\label{eq:non-crossing pi}
|\delta \pi_\vm^{-1} \delta \pi_\vm| + |K^\delta( \delta \pi_\vm^{-1} \delta \pi_\vm) | = |\delta \gamma_\vm^{-1} \delta \gamma_\vm|.
\end{equation*}
Putting the last three equations together we have
\begin{align}
|\tau| &   + |K^\delta_{\pi_\vm}(\tau)| = \big(2 + |\delta \gamma_m^{-1} \delta \gamma_m|  - |K^\delta(\tau)|\big) \\
&\qquad\mbox{} + \big(
|K^\delta(\tau)| - |K^\delta(\delta \pi_\vm^{-1} \delta \pi_\vm)| \big)\notag \\ 
& \mbox{} =
2 + |\delta \gamma_m^{-1} \delta \gamma_m| -  |K^\delta(\delta \pi_\vm^{-1} \delta \pi_\vm)| = |\delta \pi_\vm^{-1}\delta \pi_\vm| + 2. \notag
\end{align}

So now we show that there exists a unique $ V \in \pi$ such that
\begin{enumerate}
\item[$(a)$]
$\#(\cU_\sV) = 2 \#(\pi_\vm) - 1$, and

\item[$(b)$]
$\tau \vee K^\delta_{\pi_\vm}(\tau) = \cU_\sV$. 
\end{enumerate}
Property $(a)$ will show that 
\[
|(\cU_\sV, K^\delta(\delta\pi_\vm^{-1} \delta \pi_\vm))| = |\delta \pi_\vm^{-1}\delta \pi_\vm| + 2
= | \tau| + |K^\delta_{\pi_\vm}(\tau)|. 
\]
Combining this with property $(b)$ we will have $(0_\tau,
\tau) \leq (\cU_\sV, \delta \pi^{-1}_\vm \delta \pi_\vm)$;
which implies that $\tau \in \tilde{S}_2$. So let us prove
$(a)$ and $(b)$.

As $\tau$ has through cycles, there must be at least one
block, $V$, of $\pi_\vm$ that meets a through cycle of
$\tau$.  Suppose that $V$ and $V'$ are distinct blocks of
$\pi_\vm$ that meet through cycles of $\tau$. As there two
such blocks we may let $x$ and $y$ be the first and last
elements of $V$. By this we mean that proceeding in cyclic
order, starting at $V'$, we meet $x$ before $y$. Then
$\gamma_m^{-1}(x)$ and $y$ are in the same cycle of
$K^\delta(\delta\pi^{-1}_\vm\delta\pi_\vm) \leq
K^\delta(\tau)$; the same also holds for $-\gamma_m^{-1}(x)$
and $-y$. Thus, considering $\tau$ as a partition, either
\[
\tau \leq 
\{(1, \dots, \gamma_m^{-1}(x), \gamma_m(y), \dots, m, -1, \dots, -m),
(x, \dots, y)\}, \mbox{\ or}
\]
\[
\tau \leq
\{(1, \dots, \gamma_m^{-1}(x), \gamma_m(y), \dots, m), 
(x, \dots, y, -1, \dots, -m)\}.
\]
Since we assumed that $V$ meets a through cycle of $\tau$ we must have 
\[
\tau \leq
\{(1, \dots, \gamma_m^{-1}(x), \gamma_m(y), \dots, m), 
(x, \dots, y, -1, \dots, -m)\}.
\]
As $V' \subseteq (1, \dots, \gamma_m^{-1}(x), \gamma_m(y),
\dots, m)$, we see that $V'$ cannot meet a through cycle of
$\tau$.  Then $\tau$ connects $V$ to $-V$, so $\tau \vee
K^\delta_{\pi_\vm}(\tau) = \tau \vee \delta\pi_\vm^{-1}
\delta \pi_\vm = \cU_\sV$ and $\#(\cU_\sV) = \#(\delta
\pi_\vm^{-1} \delta \pi_\vm) -1 = 2\#(\pi_\vm) -1 $. This
proves the claims $(a)$ and $(b)$. Hence $\tau \in \tilde
S_2$
\end{proof}

\subsection{\normalsize The two parts combined}

To complete the proof of Proposition \ref{prop:first} we
only have to combine these two terms to obtain a proof of
Equation (\ref{eq:two part strategy}):
\[
\partial\kappa_\pi(A_1, \dots, A_r)
=
\sum_{\rho \in N_2}
\partial \kappa_\rho(a_1, \dots, a_m) 
+
\sum_{\tau \in S_2}
\kappa_{\tau/2}(a_1, \dots, a_m),
\]
and then sum over $\pi \in \NC(r) \setminus \{1_r\}$.

\begin{figure}[t]
\includegraphics[width=112.5pt]{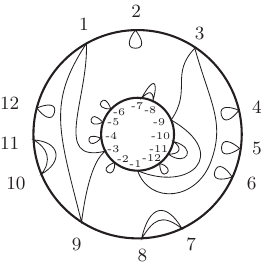}\hfill
\includegraphics[width=112.5pt]{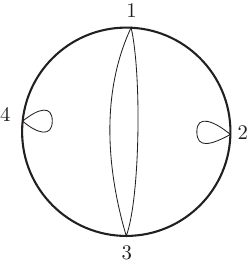} \hfill
\includegraphics[width=112.5pt]{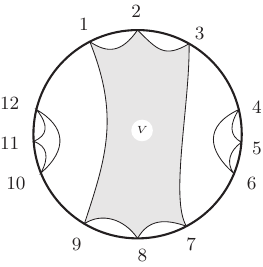}
\caption{\label{fig:two}\small This figure illustrates the
  proof of Lemma \ref{lemma:second bijection}. We have $r =
  4$ with $m_1 = m_2 = m_3 = m_4 = 3$. $\tau$ is on the
  left, $\pi$ is in the centre, $\pi_\vm$ is on the right,
  and $V= \{1,2,3,7,8,9\}$.
\begin{center}
$K^\delta_{\pi_\vm}(\tau) = (1, 2, -8, -3)(-1, 3, 8, -2)
(4, 5, 6)(-6,  -5, -4)(7)(-7)\ab (9)(-9)(10)(-10)(11, 12)(-11, -12)$\\
$K^\delta(\tau) =
(1, 2, -8, -6, -5, -4, -3)(3, 4, 5, 6, 8, -2, -1)(7)(-7)\ab (9,11,12)(-12,-11,-9)(10)(-10)$
\end{center}
Note that $K^\delta(\pi_\vm)|_{\pm \sM} =
K^\delta(\tau)|_{\pm \sM}$, as
$K_{\pi_\vm}^\delta(\tau)|_{\pm \sM} = \id_{\pm
  \sM}$. \hbox{}\hfill$\diamond$}
\end{figure}

\section{The second step: Proposition \ref{prop:third bijection}}
\label{sec:product formula second step}

As noted above, if we can prove Equation (\ref{eq:combined
  formulas}), then we will have concluded the proof of
Theorem \ref{thm:product rule}.  Proposition \ref{prop:third
  bijection} below will provide the proof.  Let us recall
the definition of $S_3$:

\[
S_3 = \{ \sigma \in S_\NC^\delta(m, -m) \mid K^\delta(\sigma)|_{\pm M} 
\mbox{\ has through cycles\ } \}. 
\]

\begin{proposition}\label{prop:third bijection}
\begin{equation}\label{eq:s three}
\sum_{\sigma \in S_\NC^\delta(r, -r)} \kappa_{\sigma/2}(A_1, \dots, A_r)
=
\sum_{\tau \in S_3}
\kappa_{\tau/2}(a_1, \dots, a_m)
\end{equation}
\end{proposition}

\begin{proof}
It may be helpful to use the example in Figure
\ref{fig:second term} to follow the proof.

Recall from \cite[Lemma 24]{mvb} that for every $\sigma \in
S_\NC^\delta(r, -r)$ there are $1 \leq j < k \leq r$ such
that $\sigma$ is non-crossing with respect to $\gamma_\sz$
where
\[
\gamma_\sz
= (1, \dots, j-1, -(k-1), \dots, -j, k, k+1, \dots, r)
\]
\[
\qquad\mbox{}\times
(-r, \dots, -k, j, j+1, \dots, k-1, -(j-1), \dots, -1).
\]
If we let $\hat\gamma = (1, \dots, j-1, -(k-1), \dots, -j,
k, k+1, \dots, r)$, then $\gamma_\sz = \delta
\hat\gamma^{-1} \delta \hat\gamma$.  We let $\hat\sigma$ be
the permutation consisting of the cycles of $\sigma$
contained in $\hat\gamma$. Then
\[
\kappa_{\sigma/2}(A_1, \dots, A_r) = 
\kappa_{\hat\sigma}(A_1, \dots, A_{j-1}, A_{k-1}^t, \dots, A_j^t, A_k, \dots, A_r). 
\]

Next we let $\sigma_\vm$ be the permutation of $[\pm m]$
defined as follows. For $k \in [m] \setminus M$ we set
$\sigma_\vm(k) = \gamma_m(k)$. For $k \in M$ with $k = m_1 +
\cdots + m_l$ we define \[ \sigma_\vm(k) =
\begin{cases} 
m_1 + \cdots + m_{\sigma(l) -1 } + 1 & \mbox{if\ } \sigma(l) \in [r] \\
-(m_1 + \cdots + m_{-\sigma(l)}) & \mbox{if\ }  \sigma(l) =\in [-r].
\end{cases}
\]
Next for $k \in [-m]\setminus -\gamma_m(M)$ we set
$\sigma_\vm(k) = \delta\gamma_m^{-1}\delta(k)$. If $k \in
-\gamma_m(M)$ with $k = -(m_1 + \cdots + m_{l-1} +1)$ we set
\[
\sigma_\vm(k) =
\begin{cases} 
-(m_1 + \cdots + m_{\sigma(-l)}) & \mbox{if\ } \sigma(-l) \in [-r] \\
m_1 + \cdots + m_{\sigma(-l)-1}+ 1 & \mbox{if\ }  \sigma(-l) \in [r].
\end{cases}
\]
Note that if $l \in [r]$ and $\sigma(l) \in [-r]$ then
$\sigma_\vm(m_1 + \cdots + m_l) =$ $-(m_1 + \cdots +
m_{-\sigma(l)})$ and thus $K^\delta(\sigma_\vm)(-(m_1 +
\cdots + m_{-\sigma(l)})) = m_1 + \cdots + m_l$. So
$K^\delta(\sigma_\vm)|_{\pm M}$ always has a through
cycle. Also $\sigma_\vm = \delta \hat \sigma_\vm^{-1} \delta
\hat\sigma_\vm$.

Let $I_l = \{ m_1 + \cdots + m_{l-1} +1, \dots, m_1 + \cdots
+ m_l\}$ and $J_l = \{ -(m_1\ab + \cdots +m_l), \dots, -(m_1
+ \cdots + m_{l-1}+1)\}$ (note the reversal of order).  We
let $ [\hat m] = I_1 \cup \cdots \cup I_{j-1} \cup J_{k-1}
\cup \cdots \cup J_j \cup I_k \cup \cdots \cup I_r$ and
\begin{multline*}
\widehat{M}
= \{ m_1, \dots, m_1 + \cdots + m_{j-1}, -(m_1 + \cdots + m_{k-2}+1), 
\dots, \\
-(m_1 + \cdots + m_{j-1} + 1), m_1 + \cdots + m_k, \dots, m_1 + \cdots + m_r\}.
\end{multline*}
Note that $\widehat{M} \subseteq [\hat m]$. We let
$\hat\sigma_\vm$ be the restriction of $\sigma_\vm$ to
$[\hat m]$.

\begin{figure}
\setbox1=\hbox{\includegraphics{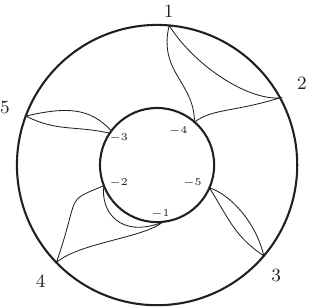}}
\setbox2=\hbox{\includegraphics{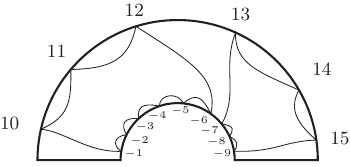}}
\setbox3=\hbox{\includegraphics{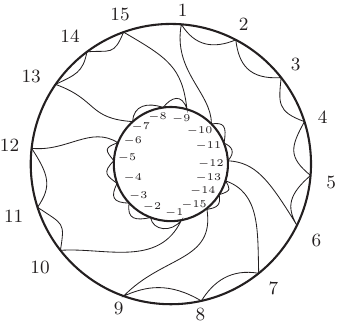}}
\setbox4=\hbox{\includegraphics{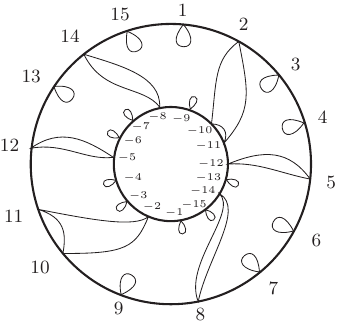}}
\setbox5=\hbox{\includegraphics{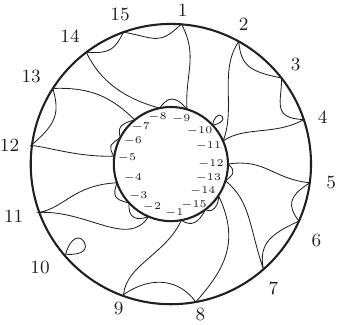}}
\setbox6=\hbox{\includegraphics{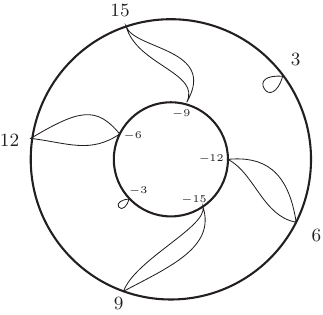}}

\leavevmode
\vbox{\raggedright\hsize\wd1\box1
{\tiny 
\begin{center}$\sigma = (1, 2, -4) (3, -5)(4, -2, -1)(5, -3)$\end{center}}} \hfill
\vbox{\raggedright\hsize\wd2\box2
{\tiny 
\begin{center}$\sigma_\vm = (10, 11, 12, -6, -5, -4, -3, -2, -1)\ab (13, 14, 15, -9, -8, -7)$

$\widehat{M} = \{12, 15, -7, -4, -1 \}$

$\hathat M = \{3, 6, 9, -13, -10\}$
\end{center}}}

\kern-0.4em

\leavevmode
\vbox{\raggedright\hsize\wd3\box3
{\tiny 
\begin{center}$\hat\sigma_\vm = (10, 11, 12, -6, -5, -4, -3, -2, -1)\ab (13, 14, 15, -9, -8, -7)(1, 2, 3, 4, 5, 6, \ab -12, -11, -10)(13, 14, 15, -9, -8, -7)$\end{center}}} 
\hfill
\vbox{\raggedright\hsize\wd4\box4
{\tiny 
\begin{center}$\tau= (1)(2, -11, -10)(3)(4)(5, -12) (6) (7) \ab (8, -14) (9) (10, 11, -2) (12, -5) (13) (14, -8)\ab (-1)(-3)(-4)(-6)(-7) (-9) (-13)(-15)$\end{center}}} 

\kern-0.4em

\leavevmode
\vbox{\raggedright\hsize\wd5\box5
{\tiny 
\begin{center}$K^\delta(\tau)$\end{center}}}
\hfill
\vbox{\raggedright\hsize\wd6\box6
{\tiny 
\begin{center}$K^\delta(\tau)|_{\pm M}$\end{center}}}

\kern-1em 

\caption{\label{fig:second term}\small The six steps in
  Proposition \ref{prop:third bijection}. We are given
  $\sigma$ in the upper left; we construct $\sigma_\vm$ in
  the upper right; then $\hat\sigma_\vm$ in the centre
  left. In the centre right we have a $\tau$ that produces
  the given $\sigma$; $K^\delta(\tau)$ in the lower left,
  and $K^\delta(\tau)|_{\pm M}$ in the lower right.  }
\end{figure}

Since we reversed the order of elements in the
$J$-intervals, we can just say that $\widehat{M}$ consists
of the right hand endpoints of the intervals $\{I_1, \dots,
I_{j-1},\ab J_{k-1}, \dots, J_j, I_k, \dots, I_r\}$.

Thus when we expand $\kappa_{\hat\sigma}(A_1, \dots,
A_{j-1}, A_{k-1}^t, \dots, A_j^t, A_k, \dots, A_r)$ we get
by \cite[Thm.~15]{mst}
\begin{multline}\label{eq:local product expansion} 
\mathop{\sum_{\tau_\sz \in \NC([\hat m])}}_
{\tau_\sz^{-1} \hat\sigma_\vm \sep \widehat M}
\kappa_{\tau_\sz}(a_1, \dots, a_{m_1 + \cdots + m_{j-1}},
a_{m_1 + \cdots + m_{k-1}}^t, \dots, \\ 
a_{m_1 + \cdots + m_{j-1}+1}^t,
a_{m+1 + \cdots + m_{k-1}+1}, \dots, a_{m+1 + \cdots + m_r}).
\end{multline}
Now for $\tau_\sz \in \NC([\hat m])$ with $\tau_\sz^{-1}\hat
\sigma_\vm |_{\widehat M} = \id_{\widehat M}$, we let $\tau
= \delta \tau_\sz^{-1} \delta \tau_\sz$. We claim that
$K^\delta(\tau)|_{\pm M} = K^\delta(\sigma_\vm)|_{\pm
  M}$. This will prove that $\tau \in S_3$.

Note that 
\[
K^\delta(\sigma_\vm)  = \delta \gamma_m^{-1} \hat\sigma_\vm  \delta \cdot \hat\sigma_\vm^{-1} \gamma_m
\mbox{\ and\ } 
K^\delta(\tau)  = \delta \gamma_m^{-1} \tau_\sz \delta \cdot \tau_\sz^{-1}\gamma_m 
\]

\noindent
Now let 
\hfill$
\hathat M = \{
m_1 + \cdots + m_j, \dots, m_1 + \cdots + m_{k-1} \} \ \cup
$\hfill\hbox{}
\[ \{ 
-(m_1 + \cdots + m_{k-1} + 1), \dots, -1, -(m_1 + 1), 
\dots, -(m_1 + \cdots + m_{j-2} + 1)\}
\]
Then $\delta\gamma_m\delta(\pm M) = \widehat{M} \cup \hathat
M$. Let $\rho = \tau_\sz^{-1} \hat\sigma_\vm \delta \tau_\sz
\hat\sigma_\vm^{-1} \delta$. Then $\rho|_{\widehat{M} \cup
  \hathat M} = \id_{\widehat{M} \cup \hathat M}$. So we let
$\tilde\rho = \delta \gamma_m^{-1} \delta \rho \delta
\gamma_m \delta$. Then $\tilde\rho|_{\pm M} = \id_{\pm
  M}$. Next we observe that $\tilde \rho K^\delta(\hat
\sigma_\vm) = K^\delta(\tau)$. Indeed
\begin{align*}
\tilde \rho  K^\delta(\hat \sigma_\vm) 
& =
\delta \gamma_m^{-1} \delta \rho \delta \gamma_m \delta \cdot \delta \gamma_m^{-1}
\hat \sigma_\vm \delta \hat \sigma_\vm^{-1} \gamma_m \\
& =
\delta \gamma_m^{-1} \delta \big[ \tau_\sz^{-1} \hat\sigma_\vm \delta \tau_\sz \hat \sigma_\vm^{-1} \delta \big] \delta \gamma_m \delta \cdot \delta \gamma_m^{-1} \hat\sigma_\vm \delta \hat\sigma_\vm^{-1} \gamma_m \\
& =
\delta \gamma_m^{-1} \delta  \cdot \tau_\sz^{-1} \hat \sigma_\vm \cdot \delta \tau_\sz \hat \sigma_\vm^{-1} \delta \cdot \delta \hat\sigma_\vm \delta \cdot \hat\sigma_\vm^{-1} \gamma_m  \\
& =
\delta \gamma_m^{-1} \delta  \cdot \underbrace{\tau_\sz^{-1} \hat \sigma_\vm} \cdot 
 \underbrace{\delta \tau_\sz  \delta} \cdot \hat\sigma_\vm^{-1} \gamma_m \\
& \stackrel{(*)}{=}
\delta \gamma_m^{-1} \delta  \cdot  \underbrace{\delta \tau_\sz  \delta} \cdot 
\underbrace{\tau_\sz^{-1} \hat \sigma_\vm} \cdot  \hat\sigma_\vm^{-1} \gamma_m = K^\delta(\tau),\\
\end{align*}
where the equality $(*)$ holds because $\tau_\sz^{-1} \hat\sigma_\vm$ and $\delta \tau_\sz \delta$ commute as they act on disjoint subsets of $[\pm m]$. As 
$K^\delta(\hat \sigma_\vm)$ acts trivially on $(\pm M)^c$ (the complement of $\pm M$) we have by \cite[Lemma 6]{mst} that
\[
 K^\delta(\hat \sigma_\vm) |_{\pm M} = \tilde \rho|_{\pm M}  K^\delta(\hat \sigma_\vm) |_{\pm M} = ( \tilde \rho  K^\delta(\hat \sigma_\vm)) |_{\pm M}
= K^\delta(\tau)|_{\pm M}
\]
as claimed. Since $K^\delta(\sigma_\vm)|_{\pm M}$ has a
through cycle (as observed above) we see that
$K^\delta(\tau)|_{\pm M}$ has a through cycle and thus $\tau
\in S_3$.

So now we have shown that the left hand side of (\ref{eq:s
  three}) can be written as a sum over $\tau$'s with each
$\tau$ in $S_3$. We must further show that each $\tau \in
S_3$ occurs once and only once in the expansion
(\ref{eq:local product expansion}). To achieve this we must
show how to recover $\sigma$ from $\tau$, exactly as in the
second part of the proof of Lemma \ref{lemma:second
  bijection}. We will take $\sigma \in S_{\pm r}$ to be the
inverse Kreweras complement of the restriction of the
Kreweras complement of $\tau$ to $\pm M$.

Indeed, we consider $K^\delta(\tau)|_{\pm M}$; this is a
permutation of $\pm M$. Then we conjugate by $\psi:[\pm r]
\rightarrow \pm M$, where $\psi$ is the map $\psi(k) = m_1 +
\cdots + m_k$ and $\psi(-k) = -(m_1 + \cdots + m_k)$ for $k
> 0$. Then $\psi^{-1} (K^\delta(\tau)|_{\pm M}) \psi$ is a
permutation of $[\pm r]$. We seek $\sigma$ such that
$K^\delta(\sigma) = \psi^{-1} (K^\delta(\tau)|_{\pm M})
\psi$. To this end we let
\[
\sigma =  \gamma_r (\psi^{-1} (K^\delta(\tau)|_{\pm M}) \psi)^{-1} \delta \gamma_r^{-1}\delta.  
\]
Then $K^\delta(\sigma) = \psi^{-1} (K^\delta(\tau)|_{\pm M})
\psi$. Thus $K^\delta(\sigma_\vm) |_{\pm M} =
K^\delta(\tau)|_{\pm M}$.  This shows that every term
appears once and only once and this completes the proof of
Proposition \ref{prop:third bijection}. \end{proof}

\thebottomline\end{document}